\documentclass[11pt, a4paper]{article}

\usepackage{amsmath}
\usepackage{amsfonts}
\usepackage{amssymb}
\usepackage[english]{babel}
\usepackage{refcheck}
\usepackage{xcolor}
\usepackage{euscript}
\usepackage{eufrak}
\usepackage{dsfont}

\def\english{\selectlanguage{english}}

\providecommand\mathbb{\bf}
\newcommand\R{{\mathbb R}}
\newcommand\N{{\mathbb N}}
\newcommand\Z{{\mathbb Z}}

\newcommand\Y{{\mathbb Y}}

\newtheorem{thm}{Theorem}[section]
\newtheorem{lemma}{Lemma}[section]
\newtheorem{pro}{Proposition}[section]
\newtheorem{defi}{Definition}[section]
\newtheorem{coro}{Corollary}[section]
\newtheorem{remark}{Remark}[section]

\newcounter{Remark}

\renewcommand\theRemark{\arabic{Remark}}

\newcounter{steps}
\newenvironment{proof}[1][]{%
\par\medbreak\setcounter{steps}{0}
{\noindent\bfseries Proof#1. }} {\hfill\fbox{\ }\medbreak}

\newcounter{substeps}[steps]

\newcommand{\QFP}[0]{
Q_{\mathrm{FP}}}

\newcommand{\calE}[0]{
\mathcal{E}}

\newcommand{\calF}[0]{
\mathcal{F}}

\newcommand{\calV}[0]{
\mathcal{V}}

\newcommand{\calH}[0]{
\mathcal{H}}

\newcommand{\calT}[0]{
\mathcal{T}}

\newcommand{\intv}[1]{
\int _{\R ^d} \!#1 \;\mathrm{d}v}

\newcommand{\intvtt}[1]{
\int _{\R ^2} \!#1 \;\mathrm{d}v}

\newcommand{\intxtt}[1]{
\int _{\R ^2} \!#1 \;\mathrm{d}x}

\newcommand{\intxvtt}[1]{
\int _{\R ^2} \int_{\R^2} \!#1 \;\mathrm{d}v\mathrm{d}x}

\newcommand{\intTxtt}[1]{
\int_{0}^{T} \int _{\R ^2}  \!#1 \;\mathrm{d}x \mathrm{d}t}

\newcommand{\intTxvtt}[1]{
\int_{0}^{T} \int _{\R ^2} \int_{\R^2} \!#1 \;\mathrm{d}v\mathrm{d}x \mathrm{d}t}

\newcommand{\intxtp}[1]{
\int _{\R ^2} \!#1 \;\mathrm{d}x^\prime}

\newcommand{\intvtp}[1]{
\int _{\R ^2} \!#1 \;\mathrm{d}v^\prime}

\newcommand{\fe}[0]{
f ^\varepsilon}

\newcommand{\Divx}[0]{
\mathrm{div}_x}

\newcommand{\Divv}[0]{
\mathrm{div}_v}

\newcommand{\eps}[0]{
\varepsilon}

\newcommand{\calO}[0]{
\mathcal{O}}

\begin{document}
\english

\title{Asymptotic behavior of the two-dimensional Vlasov-Poisson-Fokker-Planck equation with a strong external magnetic field}

\author{Miha\"i BOSTAN \thanks{Aix Marseille Universit\'e, CNRS, Centrale Marseille, Institut de Math\'ematiques de Marseille, UMR 7373, Ch\^ateau Gombert 39 rue F. Joliot Curie, 13453 Marseille FRANCE. E-mail : {\tt mihai.bostan@univ-amu.fr}}, 
Anh-Tuan VU \thanks{Aix Marseille Universit\'e, CNRS, Centrale Marseille, Institut de Math\'ematiques de Marseille, UMR 7373, Ch\^ateau Gombert 39 rue F. Joliot Curie, 13453 Marseille FRANCE. E-mail : {\tt anh-tuan.vu@univ-amu.fr}}
}

\date{(\today)}

\maketitle

\begin{abstract}
The subject matter of the paper concerns the Vlasov-Poisson-Fokker-Planck (VPFP) equations in the context of magnetic confinement. We study the long-time behavior of the VPFP system with an intense external magnetic field, neglecting the curvature of the magnetic lines. When the intensity of the magnetic field tends to infinity, the long-time behavior of the particle concentration is described by a first-order nonlinear hyperbolic equation of the Euler type for fluid mechanics. More exactly, when the magnetic field is uniform, we find the vorticity formulation of the incompressible Euler equations in two-dimensional space. Our proofs rely on the modulated energy method.

\end{abstract}

\paragraph{Keywords:} Vlasov-Poisson-Fokker-Planck equations, Guiding center approximation, Modulated energy.

\paragraph{AMS classification:} 35Q75, 78A35, 82D10
\\
\\

\section{Introduction}
\label{Intro}
We consider $f = f(t, x, v)$ the density of a population of charged particles of mass $m$, charge $q$ depending on time $t$, position $x$ and velocity $v$. We are interested in the Vlasov-Poisson system, in the presence of an external magnetic field, taking into account the collisions between charged particles. Neglecting the curvature direction of magnetic field lines, we assume that the external magnetic field has a constant direction orthogonal to $Ox_1, Ox_2$ but a variable amplitude. In the two-dimensional setting $x = (x_1, x_2)$, $v = (v_1, v_2)$, the Vlasov-Fokker-Planck equation is written 
\begin{align}
\label{VPFP2D-nonScale}
{\partial _t}f + v \cdot {\nabla _x}f + \dfrac{q}{m}\left\{ {E\left[ f(t) \right]\left( {x} \right) + {B}\left( x \right){}^ \bot v} \right\} \cdot {\nabla _v}f = {\QFP}\left( f \right),\,\,\left( {t,x,v} \right) \in {\R_ + } \times {\R^2} \times {\R^2}.
\end{align}
Here the notation ${}^ \bot \left( \cdot \right)$ stands for the rotation of angle $-\pi /2$, $\mathit{i.e.,}$ ${}^ \bot v  = (v_2, - v_1)$, $v=(v_1,v_2)$ and the magnetic field is written $\textbf{B}(x) = (0, 0, B(x))$, where $B(x)$ is a given function. The electric field $E[f(t)] = -\nabla_x \Phi[f(t)]$ derives from the potential
\[
\Phi[f(t)](x) = -\dfrac{q}{2\pi \epsilon_0}\intxtp{\ln |x-x'|\left(\intvtp{f(t,x',v')}-D(x')\right)},
\]
which satisfies the Poisson equation
\[
- {\epsilon _0}{\Delta _x}\Phi \left[ {f\left( t \right)} \right](x) = q\left(\intvtt{f\left( {t,x,v} \right)} - D(x)\right) ,\,\,(t,x)\in \R_ +  \times \R^2,
\]
whose fundamental solution is $z\to -\frac{1}{2\pi}\ln|z|, z\in\R^2\backslash \left\{0\right\}$. Here, the function $D= D(x)$ is the concentration of a background of positive charges and is assumed to be given. The constant $\epsilon_0$ represents the electric permittivity of the vacuum. For any particle density $f = f(t, x, v)$, the notation $E[f(t)] ( x)$ represents the Poisson electric field 
\begin{align}
\label{ElecField-nonScale2D}
E\left[ f(t) \right](x) =\dfrac{q}{2\pi {\epsilon _0}}\intxtp{\left( {\intvtp{f\left( {t,x',v'}  \right) } } - D(x') \right)\dfrac{{x - x'}}{\left| {x - x'} \right|^2}},
\end{align}
and $n[f(t)]$, $j[f(t)]$ stand for the concentration and the current density respectively
\[
n\left[ f(t) \right] = \intvtt {f\left( {t,\cdot ,v} \right)},\,\,\,\,j\left[ f(t) \right] =  q\intvtt {vf\left( {t,\cdot ,v} \right)}. 
\]
In the equation \eqref{VPFP2D-nonScale}, the operator $\QFP$ is the linear Fokker-Planck operator acting on velocities
\begin{align*}
\QFP\left( f \right) = \dfrac{1}{\tau}\Divv \left( {\sigma {\nabla _v}f + vf} \right),
\end{align*}
where $\tau$ is the relaxation time and $\sigma$ is the velocity diffusion, see \cite{Cha1949} for the introduction of
this operator, based on the principle of Brownian motion. We complete the above system by the initial condition
\begin{align}
\label{Initial-nonScale2D}
f\left( {0,x,v} \right) = f_\mathrm{in}\left( {x,v} \right),\,\,\left( {x,v} \right) \in \R^2 \times \R^2.
\end{align}

In this work, we analyze the evolution of the distribution density $f$ over a long time, in the regime of an intense magnetic field (gyro-kinetic), in order to observe the drift phenomenon in the directions orthogonal to the magnetic field. Indeed, it is well known that the velocities of electric cross field drift and the magnetic gradient drift are proportional to $\frac{1}{B}$ and consequently it is necessary  to observe the drift movements over a large time proportional to $B$. Namely, we consider
\begin{align*}
f\left( {t,x,v} \right) = {\fe }\left( {\bar t,x,v} \right), \,\, B^\eps(x)=\dfrac{B(x)}{\eps},\,\,\, \bar t = \eps t.
\end{align*}
Here $\eps > 0$ is a small parameter related
to the ratio between the cyclotronic period and the advection time scale. Hence ${\partial _t}f = \eps {\partial _{\bar t}}{f^\eps }$. Then in the  equation \eqref{VPFP2D-nonScale}, the term $\partial_t$ is to be replaced by $\eps\partial_{\bar t}$ or by $\eps\partial _t$ to simplify our notation, and the Vlasov-Poisson-Fokker-Planck system \eqref{VPFP2D-nonScale}, \eqref{ElecField-nonScale2D} becomes
\begin{align}
 & \eps {\partial _t}{f^\eps } + v \cdot {\nabla _x}{f^\eps } + \dfrac{q}{m}E\left[ {f^\eps }(t) \right]\cdot\nabla_v f^\eps + \dfrac{{{\omega _c}( x )}}{\eps }{}^ \bot v \cdot {\nabla _v}f^\eps = Q_{FP}(f^\eps ), \label{VPFP2d-Scale} \\
 & E\left[ f^\eps  \right] =  - {\nabla _x}\Phi[f^\eps],\, \,\, - {\epsilon _0}{\Delta _x}\Phi [ f^\eps ] = q\left({{n^\eps }} - D\right)  = q\left( \intvtt{f^\eps \left( {t,\cdot,v} \right)} -  D \right), \label{Poisson2D-Scale}
\end{align}
where $\omega_c (x)= \frac{qB(x)}{m}$ stands for the cyclotron frequency. We complete with an initial condition
\begin{align}
{f^\eps }\left( {0,x,v} \right) = f^\eps_\mathrm{in} \left( {x,v} \right),\,\,(x,v)\in \R^2\times\R^2. \label{Initial2D-Scale} 
\end{align}

The existence theory of the weak and classical solution of the VPFP system is now well developed and understood. Let us summarize the literature concerning existence results for this problem. In the absence of the external magnetic field $\mathit{i.e.,}$ $ B(x)=0$, several existence results for the VPFP system are known. The classic solutions have been studied by Degond in \cite{Degond1986} which showed the global/local existence and the uniqueness of the strong solution in one and two/three dimensions respectively, without friction term $\mathit{i.e.,}$ $\QFP=\sigma\Delta_v $. Victory and O'Dwyer obtained in \cite{VicDwy} the same result of existence of classical solution using the fundamental solution of the operator $\partial_t + v\cdot\nabla_x - \nabla_v\cdot(\sigma\nabla_v + v)$. In \cite{ReinWeck}, G. Rein and J. Weckler gave sufficient conditions to show the global existence of classical solutions in three dimensions. Regarding weak solutions, we can mention the works of Victory in \cite{Victory91}, J. A. Carrillo and J. Soler in \cite{CarSol95} with an initial data in the space $L^p$. With the magnetized VPFP system, we show the global existence in time of weak solutions, in the sense of Definition $2.1$. 


We study the asymptotic behavior of the solutions $(f^\eps)_{\eps >0}$ of the problem \eqref{VPFP2d-Scale}, \eqref{Poisson2D-Scale}, and \eqref{Initial2D-Scale} when $\eps$ tends to $0$. By investigating the balance of free energy associated with the VPFP system, we formally show in Section $4$ that the family $(f^\eps)_{\eps >0}$ converges to the limit distribution function $f\left( {t,x,v} \right) =n(t,x)M(v)= n(t,x)\frac{1}{{2\pi\sigma }}{e^{\frac{{ - {{\left| v \right|}^2}}}{2\sigma}}}$, $(t,x,v)\in\R_+ \times\R^2\times\R^2$, where the limit concentration $n$ verifies the first-order nonlinear hyperbolic equation
\begin{equation}
\label{equ:LimMod2D}
{\partial _t}n + \Divx \left[ n\left( \dfrac{^ \bot E[n]}{B(x)} - \sigma\dfrac{ ^ \bot \nabla \omega _c(x)}{\omega _c^2 (x)} \right) \right] = 0,\,\,(t,x)\in\mathbb{R}_+ \times\mathbb{R}^2,
\end{equation}
coupled to the Poisson equation
\begin{equation}
\label{LimPoisson2D}
E[n] = -\nabla_x \Phi[n],\,\, -\epsilon_0 \Delta_x \Phi[n] = q (n-D),
\end{equation}
with the initial condition
\begin{equation}
\label{LimInit2D}
n(0,x) = n_{\mathrm{in}}(x) =  \intvtt{f(0,x,v)}.
\end{equation}
Let us observe the limit equation \eqref{equ:LimMod2D}, we see that the concentration $n$ is advected along the vector field $\left( \frac{^ \bot E}{B(x)} - \sigma\frac{^ \bot \nabla \omega _c}{\omega _c^2 (x)} \right)$
which is the drift velocity respectively to the sum of the electric cross field drift $\frac{^\perp E}{B}$ and the magnetic gradient drift $\sigma \frac{ ^\perp \nabla\omega_c}{\omega_c ^2(x) }$. These drift velocities were mentioned in \cite{HerRod, DegFil16}. In the case of the uniform magnetic field $\mathit{i.e.,}$ $B(x) = B$, the system \eqref{equ:LimMod2D}, \eqref{LimPoisson2D}, and \eqref{LimInit2D} becomes
\begin{equation*}
\left\{\begin{array}{l}
  {\partial _t}n + \dfrac{{^\perp E[n]}}{B} \cdot {\nabla _x}n = 0,\,\,(t,x)\in\R_{+}\times\R^2,  \\
  E[n] =  - {\nabla _x}\Phi[n] ,\,\, - {\epsilon _0}{\Delta _x}\Phi[n]  = q\left( {n - D} \right),\,\, (t,x)\in\R_{+}\times\R^2, \\
  n\left( {0,x} \right) = n_{\mathrm{in}}(x),\,\, x\in\R^2 , \\ 
\end{array}\right. 
\end{equation*}
that is to say, the vorticity formulation of the two-dimensional incompressible Euler equations, with the cross electric field drift velocity $\frac{^\perp E}{B}$ and the vorticity $\mathrm{rot}_x ^ \bot E = - \frac{q}{ \epsilon_0}(n-D)$. Notice that when the background of positive charges $D=0$, the same model  was obtained by F. Golse, L. Saint-Raymond in \cite{GolRay99}, L. Saint-Raymond \cite{Ray2002} and E. Miot \cite{Miot} from the two-dimensional Vlasov-Poisson system without collisions. The authors justified rigorously the convergence towards the two-dimensional Euler equation of incompressible fluids in another approach.
Concerning the collisions between charged particles, we can mention the work of M. Herda and L.M. Rodrigues in \cite{HerRod}. In this paper, the authors are interested in the limit $\eps\searrow 0$ of the VPFP system \eqref{VPFP2d-Scale}, \eqref{Poisson2D-Scale}, and \eqref{Initial2D-Scale} in three-dimensional version $(t,x,v)\in \R_{+}\times\mathbb{T}^3\times \R^3$ (where $\mathbb{T} = \R / \Z$ is a torus one-dimensional). They formally show that the family $(f^\eps, \Phi^\eps:=\Phi[f^\eps])_{\eps>0}$ converges to the limit distribution function $f$ and the limit electric potential $\phi$ which have reached an adiabatic regime along the magnetic field
\[
f(t,x,v)= n(t,x)\dfrac{1}{(2\pi)^{3/2}}e^{-\frac{|v|^2}{2}},\,\,(t,x,v)\in \R_{+}\times\mathbb{T}^3\times \R^3,
\]
where the concentration $n$ is the anisotropic Boltzmann-Gibbs density
\[
n(t,x)= N(t, x_{\perp})\dfrac{e^{-q\phi(t,x)}}{\int_{\mathbb{T}}{e^{-q \phi(t,x_{\perp}, x_{\|})}}\mathrm{d}x_{\|}},\,\, x = (x_{\perp}, x_{\|})\in \mathbb{T}^2\times\mathbb{T}.
\]
The limit model is derived by the reduced macroscopic density $N: \R_{+}\times\mathbb{T}^2\to \R_+$ in the perpendicular direction, satisfying
\[
\partial _t N - \mathrm{div}_{x_{\perp}}\left( N ^\perp\left( \nabla_{x_{\perp}}\tilde{\phi} \right)  \right) = 0,
\]
where $\tilde{\phi}: \R_{+}\times \mathbb{T}^2\to \R$ is the average potential
\[
\tilde{\phi}(t, x_\perp) = - q\ln \left(   \int_{\mathbb{T}}{{e^{-q \phi(t,x_{\perp}, x_{\|})}}\mathrm{d}x_{\|}}\right),
\]
with the initial condition
\[
N(0, x_{\perp}) = N_{\mathrm{in}}(x_{\perp}) = \int_{\mathbb{T}}\intvtt{f_{0}(x_{\perp}, x_{\|}, v)\mathrm{d}x_{\|}}.
\]
Their results of passing to the limit concerned a linear model where the electric field is given $\mathit{i.e.,}$ $E[f^\eps] = E = -\nabla_x\phi$, for a given potential $\phi$. However, in the non-linear case of the VPFP type, they do not completely justify the passage to the limit model from the kinetic equation.

To the best of our knowledge, there has been no result on the asymptotic regime when the magnetic field is non-uniform. In the current work, the asymptotic behavior will be investigated by appealing to the relative entropy or modulated energy method, as introduced in \cite{Yau1991}. This relative entropy method relies on the smooth solution of the limit system. By this technique, one gets strong convergences.
Many asymptotic regimes were obtained using this technique, see \cite{Bre2000, BreMauPue2003, GolSaintQuas2003, PueSaint2004}
for quasineutral regimes in collisionless plasma physics, \cite{Saint2003, BerVas2005} for hydrodynamic limits in gaz dynamics, \cite{GouJabVas2004} for fluid-particle interaction, \cite{BosGou08, Bos2007, GouNiePouSol} for high electric or magnetic field limits in plasma physics.

Before establishing our main result, we define the modulated energy $\calE[n^\eps(t)|n(t)] $ by
\[
\calE[n^\eps(t)|n(t)] = \sigma\intxtt{n(t) h\left(\dfrac{n^\eps(t)}{n(t)}\right)} + \dfrac{\epsilon_0}{2m}\intxtt{|\nabla_x\Phi[n^\eps]-\nabla_x\Phi[n]|^2},
\]
where $h : \R_+ \to \R_+$ is the convex function defined by $h(s) = s\ln s -s +1, s\in\R_+$. This quantity splits into the standard $L^2$ norm of the electric field plus the relative entropy between
the particle density $n^\eps$ of \eqref{VPFP2d-Scale}, \eqref{Poisson2D-Scale}, and \eqref{Initial2D-Scale} and the particle concentration $n$ of the limit model \eqref{equ:LimMod2D}, \eqref{LimPoisson2D}, and \eqref{LimInit2D}. For any nonnegative integer $k$ and $p\in [1,\infty]$, $W^{k,p} = W^{k,p}(\R^d)$ stands for the $k$-th order $L^p$ Sobolev space. $C_b^k$ stands for $k$ times continuously differentiable functions, whose partial derivatives, up to order $k$, are all bounded and $C^k ([0, T ]; E)$ is the set of $k$-times continuously differentiable
functions from an interval $[0, T ] \subset \R$ into a Banach space $E$. $L^p (0, T ; E)$ is the set of measurable
functions from an interval $(0, T )$ to a Banach space $E$, whose $p$-th power of the $E$-norm is Lebesgue
measurable. The main result of this paper is the following

\begin{thm}$\;$\\
\label{MainThm2D}
Let $T>0$. Let $B\in C^3_b(\R^2)$ be a smooth magnetic field, such that $\inf_{x\in\R^2} B(x) = B_0 >0$ and $D$ be a fixed background density  verifying 
$
|x|D\in L^1(\R^2),\,\, D\in W^{1,1}(\R^2)\cap W^{2,\infty}(\R^2).
$
Assume that the initial particle densities $(f^\eps_\mathrm{in})_{\eps>0}$ satisfy the hypotheses $H1$, $H2$, and $H3$ (see Section $2$ below) and $M_\mathrm{in} := \sup_{\eps>0} M^\eps_\mathrm{in} <+\infty$, $U_\mathrm{in} := \sup_{\eps>0} U^\eps_\mathrm{in}< +\infty $ where
\[
M^\eps_\mathrm{in} := \intxvtt{f^\eps(x,v)},\,\, U^\eps_\mathrm{in} := \intxvtt{\dfrac{|v|^2}{2}f^\eps_\mathrm{in}(x,v)}+\dfrac{\epsilon_0}{2m}\intxtt{|\nabla_x\Phi[f^\eps_\mathrm{in}]|^2}.
\]
Let $f^\eps$ be the weak solutions of the VPFP system \eqref{VPFP2d-Scale}, \eqref{Poisson2D-Scale}, and \eqref{Initial2D-Scale} with initial data $f^\eps_\mathrm{in}$ provided by Theorem \ref{main_weak_sol}. We also assume that the initial concentration $n_\mathrm{in}$ verifies the hypotheses $H4$, $H5$ (see Section 5 below) and let $n$ be the unique smooth solution of the limit system \eqref{equ:LimMod2D}, \eqref{LimPoisson2D}, and \eqref{LimInit2D} with initial condition $n_\mathrm{in}$ constructed in Proposition \ref{main_sol_Lim}. We suppose that
\[
\lim_{\eps\searrow 0} \sigma\intxvtt{n^\eps_\mathrm{in}M(v) h\left(\dfrac{f^\eps_\mathrm{in}}{n^\eps_\mathrm{in}M}\right)}=0,\,\, \lim_{\eps\searrow 0}\calE[n^\eps_\mathrm{in}|n_\mathrm{in}] =0,
\]
where $n^\eps_\mathrm{in} = \intvtt{f^\eps_\mathrm{in}}$, $\eps>0$. Then we have
\[
\lim_{\eps\searrow 0}\sup_{0\leq t\leq T} \sigma\intxvtt{n^\eps M(v) h\left(\dfrac{f^\eps}{n^\eps M}\right)}=0,\,\, \lim_{\eps\searrow 0}\sup_{0\leq t\leq T}\calE[n^\eps(t)|n(t)] =0,
\]
\[
\lim_{\eps\searrow 0}\dfrac{1}{\eps}\intTxvtt{\dfrac{|\sigma \nabla_v f^\eps + vf^\eps|^2}{f^\eps}} =0.
\]
In particular we have the convergences $\lim_{\eps\searrow 0}f^\eps =nM$ in $L^\infty(0,T;L^1(\R^2\times\R^2))$ and $\lim_{\eps\searrow 0}\nabla_x\Phi[f^\eps] = \nabla_x\Phi[n]$ in  $L^\infty(0,T;L^2(\R^2))$.
\end{thm}

\begin{remark}$\;$\\
\label{RemElecL22D}
In two dimensional setting, the initial potential energy $\frac{\epsilon_0}{2m}\intxtt{|\nabla_x\Phi[f^\eps_\mathrm{in}]|^2}$ may not be finite (or the electric field $E[f^\eps_\mathrm{in}]$ cannot belong to $L^2(\R^2)$) even if the initial datum $f^\eps_\mathrm{in}$ lies in $C^\infty_0(\R^2\times\R^2)$. This is due to the fact that the kernel $x/|x|^d$ does not belong to $L^2(\R^2)$ at infinity, see \cite{GouNiePouSol} for a disscusion. For these
reasons one needs to slightly modify the Poisson equation adding a fixed background density $D$ satisfying the global neutrality relation $H3$, see Section $2$ below. 
\end{remark}
The paper is structured as follows. Section 2 is devoted to establish the global existence of weak solutions to the VPFP system with external magnetic field. In Section $3$, we derive a priori estimates with respect to the small parameter $\eps>0$ on the weak solutions from the evolution of physical quantities associated to the VPFP system. Section $4$ is devoted to the formal derivation of the limit model. The well-posedness of the limit model is studied in the next section. We establish existence and uniqueness results for the strong solution. The convergence towards the limit model is justified rigorously in Section $5$. We obtain strong convergence for well prepared initial conditions.

\section{Global existence of weak solutions of the VPFP equations}
In this section we will study the global existence of weak solution for the VPFP equation in the presence of an external magnetic field for fixed $\eps > 0$ under suitable assumptions on the initial data. In order to  simplify the proofs of existence of the solution, as we do not want any uniform
estimate with respect to $\eps$, we will take $\eps = 1$ and omit all the subscripts. Thus we first consider the following problem  
\begin{align}
 & \partial _t f + v \cdot \nabla _x f  + E[f]\cdot\nabla_v f + B(x) {}^ \bot v \cdot \nabla _v f = \Divv(\sigma\nabla_v f + vf), \label{eq:VPFP-NonEps} \\
 & E[f ] =  - {\nabla _x}\Phi \left[ {{f }} \right],\, \,\, - \Delta _x\Phi[f] =  \intvtt{f \left( {t,\cdot,v} \right)} -  D , \label{eq:Poi-NonEps}\\
 & f(0,x,v) = f_{\mathrm{in}} \left( {x,v} \right),\,\, (x,v)\in \R^2\times\R^2 .\label{eq:Init-NonEps}
\end{align}
The dependency on the small parameter $\eps > 0$ will be taken into account when establishing a priori estimates uniform in $\eps$ in the next section. We assume that the initial data $f_{\mathrm{in}}$  satisfies the hypotheses
\begin{enumerate}
\item[H1)] $f_{\mathrm{in}} \geq 0,\,\, f_{\mathrm{in}}\in (L^1 \cap L^\infty)(\R^2\times\R^2),\,\, (|x|+ |v|^2 +  |\ln f_{\mathrm{in}}|)f_{\mathrm{in}} \in L^1(\R^2\times\R^2)$,\label{Hypothesis1}
\item[H2)] $(1+|v|^2)^{\gamma/2}f_{\mathrm{in}} \in L^\infty(\R^2\times\R^2),\,\, \gamma >2$,\label{Hypothesis2}
\item[H3)] $\intxvtt{f_\mathrm{in}(x,v)} = \intxtt{D(x)}$.
\end{enumerate} 
We first introduce a notion of weak solution to the problem \eqref{eq:VPFP-NonEps}, \eqref{eq:Poi-NonEps}, and \eqref{eq:Init-NonEps} and our result on the global-in-time existence of
weak solutions.
\begin{defi}$\;$\\
\label{DefWeakSol}
For a given $T\in ]0,\infty[$. We say that the pair $(f,E[f])$ is a weak solution to the system \eqref{eq:VPFP-NonEps}, \eqref{eq:Poi-NonEps}, and \eqref{eq:Init-NonEps} if and only if the following conditions are satisfied
\begin{enumerate}
\item[(i)] 
$
f\geq 0,\,\,f \in L^\infty(0,T; L^1 \cap L^\infty (\R^2\times\R^2)),\,\,\ E[f] \in L^\infty((0,T)\times\R^2),
$
\item[(ii)] for any $\varphi\in C^\infty _0 ([0,T[\times\R^2\times\R^2)$, we have
\begin{align*}
&\intTxvtt{f\left( \partial_t \varphi + v\cdot\nabla_x\varphi + ( E[f]+B(x){}^\perp v)\cdot\nabla_v\varphi \right)} \\
&+ \intTxvtt{f\left( \sigma\Delta_v\varphi - v\cdot\nabla_v\varphi \right)} + \intxvtt{f_\mathrm{in}(x,v) \varphi(0,x,v)} = 0.
\end{align*}
\end{enumerate}
\end{defi} 
We will provide the global existence of the weak solution to the VPFP system \eqref{eq:VPFP-NonEps}, \eqref{eq:Poi-NonEps}, and \eqref{eq:Init-NonEps}  based on a compactness argument. For this purpose, we need the following velocity averaging lemma obtained in \cite{KarMelTri}, see also \cite{PerSou}. 
\begin{lemma}$\;$\\
\label{VelAver}
Let $(g^k)_k$ be bounded in $L^p _{\mathrm{loc}}((0,T)\times\R^2\times\R^2)$ with $1<p<\infty$, and $(G^k)_k$  be bounded in $L^p _{\mathrm{loc}}((0,T)\times\R^2\times\R^2)$. If for any $k$,  $g^k$ and $G^k$ satisfy the equation
\begin{align*}
\partial_t g^k + v\cdot\nabla_x g^k = \nabla_v^l G^k,\,\,\, g^k(t=0)= g_0\in L^p(\R^2\times\R^2),
\end{align*}
for some multi-index $l$, then for any $ \psi\in C^1 _c (\R^2\times\R^2)$ we have $\left(\intv{f^k \psi}\right)_k
$
is relatively compact in $L^p _{\mathrm{loc}}((0,T)\times\R^2)$.
\end{lemma}
The averaging lemma allows to pass to the limit in the VPFP equation including the nonlinear term $E[f]f$ in the sense of distribution, see \cite{PerLions}. The only difficult term is the Poisson equation $- \Delta_x \Phi[f] = n[f] -D$, since the velocity averaging lemma cannot be directly applied to conclude compactness of the density $n[f]$ (recall that only quantities of the type $\int_{|v|\leq R} g^k\mathrm{d}v$, $R$ being finite, converge to $\int_{|v|\leq R} g\mathrm{d}v$). In order
to get such compactness, one uses the previous lemma to show the following result, see Lemma $2.8$ in \cite{KarMelTri}.
\begin{lemma}$\;$\\
\label{Compactness}
Let $(g^k)_k$ and $(G^k)_k$ be as in Lemma \ref{VelAver} and we assume that
\begin{align*}
g^k\,\, \text{is bounded in} \,\, L^p((0,T)\times\R^2\times\R^2),\\
(|v|^2 + |x|)g^k \,\,\text{is bounded in}\,\, L^\infty(0,T;L^1(\R^2\times\R^2)).
\end{align*}
Then for any $\psi(v)$ such that $ |\psi(v)|\leq c|v|$ and $1< q<\frac{4}{3}$, the sequence $\left(\intv{g^k \psi}\right)_k$ is compact in $L^q ((0,T)\times\R^2)$.
\end{lemma}
We now state the existence results for this type of solutions to the VPFP system.
\begin{thm}$\;$\\
\label{main_weak_sol}
Let $T>0$. Let $B\in L^\infty(\R^2)$ be a smooth magnetic field and $D$ be a fixed background density verifying $|x|D\in L^1(\R^2)$, $D\in L^1(\R^2)\cap L^\infty(\R^2)$.
Assume that the initial condition $f_{\mathrm{in}}$ satisfies the hypotheses $\mathrm{H1}$, $\mathrm{H2}$ and $\mathrm{H3}$. Then  the problem \eqref{eq:VPFP-NonEps}, \eqref{eq:Poi-NonEps}, and \eqref{eq:Init-NonEps} has a global weak solution $f\geq 0$ in the sense of Definition \ref{DefWeakSol}, satisfying the following properties:
\begin{align}
\label{PropWeakSol}
 f\in L^\infty(0,T; L^1 \cap L^\infty (\R^2\times\R^2)),\,\, \left(1+ |v|^2  \right)^{\gamma/2}f\in L^\infty((0,T)\times\R^2\times\R^2),\nonumber\\ 
(|x|+|v|^2 +  |\ln f|)f\in L^\infty(0,T;L^1(\R^2\times\R^2)),\\
E[f]\in L^\infty((0,T)\times\R^2)^2,\,\, E[f]\in L^\infty(0,T;L^2(\R^2))^2\nonumber.
\end{align}
Furthermore, we have $f\in L^2([0,T]\times\R^2_x,H^1(\R^2_v))$.
\end{thm}

The proof of Theorem \ref{main_weak_sol} will be devided in 5 steps. The first is devoted to regularize the VPFP system by introducing regularization parameter $\eta$, the second stablishs some uniform-in-$\eta$ estimates, the third to a convergence of the sequence, the fourth passes to the limit using the velocity averaging lemma, Lemma 2.2 and the last step studies the properties of the solution.

\paragraph*{Step 1: Regularized VPFP system.}
For the existence of weak solutions to \eqref{eq:VPFP-NonEps}-\eqref{eq:Init-NonEps}, we first regularize this system
with respect to regularization parameters $\eta > 0$ as follows: we mollify the singular interaction potential $\ln |x|$ by parameter $\eta$ and consider a solution $f^\eta$ to that regularized system
\begin{align}
\label{eq:VFP2DBis}
\partial_t f^{\eta} + v\cdot\nabla_x f^{\eta} + E^\eta\cdot\nabla_v f^{\eta} + B (x) {}^\perp v\cdot \nabla_v f^{\eta} = \sigma \Delta_v f^{\eta} + \Divv(v f^{\eta}),
\end{align}
subject to the initial data
$
f^\eta_0 = f^{\eta}(0,x,v)=f_{\mathrm{in}}(x,v),\,\,(x,v)\in\R^2\times\R^2,
$
where the electric field $E^\eta$ is given as
\[
E^\eta = -\dfrac{1}{4\pi}\nabla_x W^\eta \star (n^\eta -D) = -\dfrac{1}{4\pi} \nabla_x \left[ \ln(|x|^2 +\eta) \right]\star (n^\eta -D), \,\, n^\eta = \intvtt{f^\eta}.
\]
We note that $\nabla_x W^\eta$ is bounded and Lipschitz continuous. Thus, the global-in-time existence of weak solutions to \eqref{eq:VFP2DBis} follows by the standard existence theory for kinetic equations, more exactly a fixed point argument in the following sense: for any $\bar{E}\in L^\infty(0,T;L^\infty(\R^2))^2$,  let $f$ be the solution of
\begin{equation*}
\left\{\begin{array}{l}
\partial_t f  + v\cdot \nabla_x f + \bar{E}(t,x)\cdot\nabla_v f + B(x) {^\perp}v \cdot \nabla_v f = \sigma \Delta_v f + \Divv(vf),\\
f(0,x,v) = f_\mathrm{in}(x,v).
\end{array}\right.
\end{equation*}  
We then define a map $T$ by
\begin{align*}
T: L^\infty(0,T;L^\infty(\R^2))^2 &\to L^\infty(0,T;L^\infty(\R^2))^2\\
\bar{E} &\mapsto T(\bar{E}) = E^\eta = -\dfrac{1}{4\pi} \nabla W^\eta \star (n-D), \,\, n = \intvtt{f}.
\end{align*}
Thanks to Theorem \ref{ExiVFP2D} in Appendix A, the operator $T$ is well-defined. The existence of a fixed point for the mapping $T$ comes from almost the same argument in \cite{KarMelTri}, Theorem 6.3. 

\paragraph*{Step 2: Uniform-in-$\eta$ estimates.} 
Thanks to Theorem \ref{ExiVFP2D}, the solution $f^\eta$ satisfies the following estimations:
\begin{equation}
\label{IneqNormSequ}
 f^{\eta} \geq 0,\,\, \|f^{\eta}(t)\|_{L^p(\R^2\times\R^2)}\leq C(T)\|f_{\mathrm{in}}\|_{L^p(\R^2\times\R^2)},\,\,p\in[1,\infty],\,\, t\in[0,T].
\end{equation}
\begin{equation}
\label{IneqKinEnerSequ}
\intxvtt{f^{\eta}\dfrac{|v|^2}{2}}  <  C(T, \|E^\eta\|_{L^\infty})\intxvtt{f_{\mathrm{in}}|v|^2},\,\, t\in [0,T].
\end{equation}
\begin{equation}
\label{IneqPosition}
\intxvtt{f^{\eta}|x|}  <  C(T)\intxvtt{f_{\mathrm{in}}|x|},\,\, t\in [0,T].
\end{equation}
\begin{equation}
\label{IneqDissSequ}
\|\sigma\nabla_v  f^{\eta}/{\sqrt{f^{\eta}}}\|_{L^2(0,T;L^{2}(\R^2\times\R^2))} \leq C(\|E^\eta\|_{L^\infty},T, f_\mathrm{in}, \sigma)+   \intxvtt{\sigma f_\mathrm{in}|\ln f_\mathrm{in}|}.
\end{equation}
We will now establish the uniform estimates with respect to $\eta$ of the electric field $E^\eta$, that means
$
\sup_{[0,T]}\|E^\eta\|_{L^\infty(\R^2)} < C
$
for some constant $C>0$, not depending on $\eta$. Thanks to Lemma B.1 in \cite{Degond1986}, Lemma 5.3 in Section 5, and the estimate \eqref{IneqNormSequ}, it suffices to show that for all $\eta >0$, the following inequality 
\[
\|Y^\eta (t)\|_{L^\infty(\R^2)}< C,\,\,  t \in [0,T],
\]
where we denote $Y^{\eta}(t) = (1+|v|^2)^{\gamma/2}f^{\eta}(t,x,v)$.
\begin{lemma}$\;$\\
\label{LInftyNormVeloc}
Let $f_\mathrm{in}$ be an initial data verifying the hypothesis $H2$, that means
\[
\| Y^0 \|_{L^\infty(\R^2\times\R^2)} = \| (1+|v|^2)^{\gamma/2} f_\mathrm{in}\|_{L^\infty(\R^2\times\R^2)} < \infty,\,\, \gamma>2.
\]
Then there exists a constant $C>0$ independent of $\eta$ satisfying for all $\eta>0$
\begin{align*}
\| Y^{\eta}(t) \|_{L^{\infty}(\R^2\times\R^2)} \leq C,\,\, t\in [0,T].
\end{align*}
\end{lemma}
\begin{proof}
Since the proof can be easily obtained similarly as in \cite{Degond1986}, Lemma 3.1, so we left this lemma to the reader.
\end{proof}
By Lemma \ref{LInftyNormVeloc}, we deduce that the constants in the inequalities \eqref{IneqKinEnerSequ}, \eqref{IneqDissSequ} respectively are independent with respect to $\eta$. Moreover, together this Lemma with Lemma B.1 in \cite{Degond1986} and \eqref{IneqNormSequ} gives the uniform bound of the sequence $(n^\eta)_{\eta >0}$ in $L^\infty(0,T;L^p(\R^2))$, for any $p\in [1,\infty]$.

\paragraph*{Step 3: Compactness and convergence.}
It follows from the uniform bound of the sequences that there exist a limit $(f, n, E)$  such that up to extraction of a subsequence, it holds as $\eta\to 0$ that
\begin{align*}
f^{\eta} \rightharpoonup  f \,\,\,\,\text{weak}\, \star\,\text{in}\,\,\, L^\infty (0,T;L^p(\R^2\times\R^2)), \,\, p\in ]1,\infty],\\
n^{\eta} \rightharpoonup  n \,\,\,\,\text{weak}\, \star\,\text{in}\,\,\, L^\infty (0,T;L^p(\R^2)), \,\, p\in ]1,\infty],\\
E^{\eta} \rightharpoonup  E  \,\,\,\,\text{weak}\, \star\,\text{dans}\,\,\, L^\infty ((0,T)\times\R^2).
\end{align*}
Furthermore, by using Lemma \ref{Compactness} with $\psi(v) =1$ we get the strong convergence 
\begin{equation}
\label{strongconv}
n^{\eta} \to n \,\,\text{in}\,\, L^q ((0,T)\times\R^2),\,\,\, q\in]1, 4/3[.
\end{equation}
Indeed, by uniform estimates \eqref{IneqNormSequ}, \eqref{IneqKinEnerSequ} and \eqref{IneqPosition}, the conditions in Lemma \ref{Compactness} are verified. Let us write
\[
G^{\eta} := \sigma \nabla_v f^{\eta} + vf^{\eta} - E^{\eta}f^{\eta} - B(x){}^\perp v f^{\eta}.
\]
Then the equation \eqref{eq:VFP2DBis} can be written as
\[
\partial_t f^{\eta} + v\cdot \nabla_x f^{\eta} = \nabla_v G^{\eta}.
\]
 We now claim that the sequence $(G^\eta)_{\eta >0}$ is bounded in $L^q((0,T)\times\R^2\times\R^2)$ to apply the averaging lemma, Lemma \ref{Compactness}. Hence, we need to prove the following lemma.
\begin{lemma}$\;$\\
\label{BoundGk}
For any $q\in [1,2]$, there exists a constant $C$ independent of $\eta$ such that for every $\eta >0$ we have
\[
\| G^{\eta}\|_{L^q ((0,T)\times\R^2\times\R^2)} \leq C,\,\, q\in [1,2].
\]
\end{lemma}
\begin{proof}$\;$\\
Since the proof is similar to \cite{KarMelTri}, Lemma 3.5, we briefly give the idea of that.
As the sequence of electric fields $E^\eta$ is bounded in $L^\infty((0,T)\times\R^2)$ and the magnetic field $B$ belongs to $L^\infty(\R^2)$ we obtain
\[
\| G^{\eta}\|_{L^q } \leq \|\sigma\nabla_v f^{\eta} \|_{L^q } + (1+  \|B\|_{L^\infty}) \| v f^{\eta} \|_{L^q ((0,T)\times\R^2\times\R^2)} + C \|f^{\eta}  \|_{L^q },
\]
for some positive constant $C$ not depending on $\eta$.\\
From \eqref{IneqNormSequ} it is easily seen that
$
\|f^{\eta}  \|_{L^q ((0,T)\times\R^2\times\R^2)}\leq T^{1/q} \|f_\mathrm{in}\|_{L^q(\R^2\times\R^2)}.
$
On the other hand, since $\|vf^{\eta}  \|_{L^q ((0,T)\times\R^2\times\R^2)}\leq T^{1/q} \sup_{[0,T]} \|vf^{\eta}\|_{L^q(\R^2\times\R^2)}$ and thanks to Hölder's inquality for $q\in[1,2[$ we have
\begin{align*}
\|vf^{\eta}\|_{L^q(\R^2\times\R^2)} \leq \left(\intxvtt{|v|^2 f^\eta}\right)^{1/2} \| f^\eta\|^{q/2}_{L^{q/(2-q)}(\R^2\times\R^2)}.
\end{align*}
When $q=2$ we also get
\[
\|vf^{\eta}\|_{L^2(\R^2\times\R^2)} \leq \| f^\eta \|^{1/2}_{L^\infty(\R^2\times\R^2)}\left(\intxvtt{|v|^2f^\eta}\right)^{1/2}.
\]
Consequently, the sequence $(vf^\eta)_\eta$ is bounded in $L^q((0,T)\times\R^2\times\R^2)$, for any $q\in [1,2]$. It remains to uniformly bound the sequence $\|\sigma\nabla_v f^{\eta} \|_{L^q((0,T)\times\R^2\times\R^2) }$. Using Hölder's inequality again for $q\in [1,2[$, we have
\begin{align*}
\intTxvtt{|\nabla_v f^{\eta}|^q}
\leq C(T)\|f^{\eta}\|^\frac{q}{2}_{L^\infty(0,T;L^{\frac{q}{2-q}})}\left(\intTxvtt{\dfrac{|\nabla_v f^{\eta}|^2}{f^{\eta}}}\right)^{\frac{q}{2}},
\end{align*}
and when $p=2$ we also get
\begin{align*}
\intTxvtt{|\nabla_v f^{\eta}|^2} &= \intTxvtt{f^{\eta}\dfrac{|\nabla_v f^{\eta}|^2}{f^{\eta}}}\\ &\leq \|f^{\eta}\|_{L^\infty((0,T)\times\R^2\times\R^2)} \intTxvtt{\dfrac{|\nabla_v f^{\eta}|^2}{f^{\eta}}}.
\end{align*}
Thanks to Lemma \ref{LInftyNormVeloc} and \eqref{IneqDissSequ}, we deduce that the sequence $(|\nabla_v f^{\eta}|^2/f^{\eta})_{\eta >0}$ is bounded in $L^1((0,T)\times\R^2\times\R^2)$. Therefore, the sequence $(\nabla_v f^{\eta})_{\eta >0}$ is bounded in $(L^q((0,T)\times\R^2\times\R^2))^2$ with $q\in[1,2]$. Altogether the above estimates we conclude the result of Lemma \ref{BoundGk}.
\end{proof}
\paragraph*{Step 4: Passing to the limit.}
Thanks to the weak convergences obtained in Step 3, we see
that to pass to the limit in the weak formulation of equation \eqref{eq:VFP2DBis} it suffices to show convergence towards $0$ for any test function $\varphi\in C^\infty _0 ([0,T[\times\R^2\times\R^2)$ of the non-linear contribution
\begin{align}
\label{NonLinearEner}
&\intTxvtt{ \left[  (\nabla W^\eta\star n^{\eta}) f^{\eta} - (\nabla W\star n) f   \right] \varphi}\nonumber\\
&= \intTxvtt{[(\nabla(W^\eta -W)\star n^\eta)f^\eta]\varphi}    \nonumber\\
&+\intTxvtt{ \left[ \nabla W\star (n^{\eta}-n)f^\eta \right]\varphi} + \intTxvtt{    (\nabla W\star n)\varphi (f^{^\eta} - f ) }.
\end{align}
The convergence of the contribution \eqref{NonLinearEner} to $0$ as $\eta\to 0$ can follow from almost the same analysis as in \cite{CaChoJu}, Section 5. Therefore we obtain $f$ is the weak solution of VPFP system \eqref{eq:VPFP-NonEps}, \eqref{eq:Poi-NonEps}, and \eqref{eq:Init-NonEps} with the electric field $E$ satisfying $E = -\frac{q}{2\pi\epsilon _0}\nabla_x\ln|\cdot|\star (n-D)$. Furthermore, since the sequence $(f^\eta)_{\eta >0}$ belongs to $L^2([0,T]\times\R^2_x,H^1(\R^2_v))$ it is easily check that $f\in L^2([0,T]\times\R^2_x,H^1(\R^2_v))$ by using the Theorem \ref{LionsThm}.

\paragraph*{Step 5: Properties \eqref{PropWeakSol} of solutions.}
 The nonegative limit function $f$  is a direct consequence of the weak-$\star$ convergence of the nonegative sequence $(f^\eta)_{\eta > 0}$ in $L^\infty((0,T)\times\R^2\times\R^2)$. In particular,  $f\in L^\infty((0,T)\times\R^2\times\R^2)$. Moreover, we also have $(1+|v|^2)^{\gamma/2}f \in L^\infty((0,T)\times\R^2\times\R^2)$ since the sequence $((1+|v|^2)^{\gamma/2}f^\eta)_{\eta >0}$ is bounded in $L^\infty((0,T)\times\R^2\times\R^2)$. Now, let $\varphi $ be any nonnegative function in $ C^\infty _0 ([0,T[)$  and $R>0$ be a constant. To prove $f \in L^\infty(0,T;L^1(\R^2\times\R^2))$ we use the function $\psi_R(t,x,v) = \varphi(t)\mathds{1}_{\left\{|x|\leq R, |v|\leq R\right\}}$. Hence by the weak-$\star$ convergence of $(f^\eta)_{\eta >0}$ to $f$ we deduce that
\[
\int_{0}^{T}\varphi(t) \intxvtt{f(t,x,v)\mathds{1}_{\left\{|x|\leq R, |v|\leq R\right\}}}\mathrm{d}t \leq \limsup_{\eta\to 0}\int_{0}^{T}\varphi(t) \intxvtt{f^\eta(t,x,v)}\mathrm{d}t.
\]
Taking now the limit $R\to \infty$ and apply the dominated convergence theorem to get $f \in L^\infty(0,T;L^1(\R^2\times\R^2))$. Similarly, if we choose the test function $\psi_R(t,x,v) = \varphi(t)(|x|+|v|^2)\mathds{1}_{\left\{|x|\leq R, |v|\leq R\right\}}$ then we can show that $(|x|+|v|^2)f\in L^\infty(0,T;L^1(\R^2\times\R^2))$. We complete  the property of the solution by showing that $f\ln f \in L^\infty(0,T;L^1(\R^2\times\R^2))$. Indeed, we have the identity
\[
f |\ln f| = f \ln f \mathds{1}_{\left\{ f\geq 1 \right\}} - f\ln f\mathds{1}_{\left\{0\leq f\leq 1 \right\}}.
\]
Since $f \ln f \mathds{1}_{\left\{ f\geq 1 \right\}}\leq f^2$ and $f\ln f\mathds{1}_{\left\{0\leq f\leq 1 \right\}}\leq C e^{-(|x|+|v|^2)} + (|x|+|v|^2)f$, for some constant $C>0$ together with $f\in L^\infty(0,T;L^2(\R^2\times\R^2))$ and $(|x|+|v|^2) f\in L^\infty(0,T;L^1(\R^2\times\R^2))$, we deduce that $f\ln f \in L^\infty(0,T;L^1(\R^2\times\R^2))$.

The following lemma provides the property on the potential $\Phi[f]$ and the electric field $E[f] = -\nabla_x\Phi[f]$ of the Poisson equation on $\R^2$, so as to control the potential energy under the hypothesis H3. We refer to Lemma 3 in \cite{GouNiePouSol}.
\begin{lemma}$\;$\\
\label{PropSolPoi}
Let $\rho\in L^p(\R^2)$ with any $p\in [1,\infty]$ be such that 
\[
\intxtt{(1+|x|)|\rho(x)|} < +\infty,\,\, \intxtt{\rho(x)} = 0.
\]
Consider the potential $\Phi$ given by $ \Phi(x) = -\frac{1}{2\pi}\int_{\R^2}\ln |x-y|\rho(y)\mathrm{d}y$. Then, $\Phi$ is a continuous and bounded function such that $\lim_{|x|\to\infty}\Phi(x) =0$. Furthermore, we also have $\Phi\in L^2(\R^2)$ and $\nabla\Phi\in (L^2(\R^2))^2$. 
\end{lemma}
\section{A priori estimates}
The aim of this section is the derivation of a priori estimates, uniform  with respect to $\eps$, on the weak solution $f^\eps$ provided by Theorem \ref{main_weak_sol}. These estimates are deduced from the conservation properties of the system and from the dissipation mechanism
due to the collisions. We recall that $(f^\eps, E[f^\eps])$ is a weak solution to the problem \eqref{VPFP2d-Scale}, \eqref{Poisson2D-Scale}, and \eqref{Initial2D-Scale} on $[0,T]$ with any $T>0$, if for any the test function $\varphi\in C^\infty_0([0,T[\times\R^2\times\R^2)$ we have
\begin{align}
\label{WeakSolScaleVPFP}
&\intTxvtt{f^\eps\left( \eps \partial_t \varphi + v\cdot\nabla_x\varphi + \frac{q}{m}( E[f^\eps]+\frac{B(x)}{\eps}{}^\perp v)\cdot\nabla_v\varphi \right)} \\
&+ \dfrac{1}{\tau}\intTxvtt{f^\eps\left( \sigma\Delta_v\varphi - v\cdot\nabla_v\varphi \right)} + \intxvtt{\eps f^\eps_\mathrm{in}(x,v) \varphi(0,x,v)} = 0.\nonumber
\end{align}
Let us define the free energy of the VPFP system \eqref{VPFP2d-Scale}, \eqref{Poisson2D-Scale}, and \eqref{Initial2D-Scale} as
\[
\calE[f^\eps] = \intxvtt{(\sigma f^\eps \ln f^\eps + f^\eps\dfrac{|v|^2}{2})} + \dfrac{\epsilon_0}{2m}\intxtt{|E[f^\eps]|^2}.
\]
\begin{pro}$\;$\\
\label{WeakFreeEnergy2D}
Let $(f^\eps, E[f^\eps])$ be a weak solution of the system \eqref{VPFP2d-Scale}, \eqref{Poisson2D-Scale}, and \eqref{Initial2D-Scale} provided by Theorem \ref{main_weak_sol}. Then, we have the mass conservation and the balance of the free energy
\[
\dfrac{\mathrm{d}}{\mathrm{dt}}\intxvtt{f^\eps(t)} =0,\,\,\,\,   \eps \dfrac{\mathrm{d}}{\mathrm{dt}}\calE[f^\eps(t)] = - \dfrac{1}{\tau}\intxvtt{\dfrac{|\sigma \nabla_v f^\eps + v f^\eps|^2}{f^\eps}}.
\]
\end{pro}
The mass conservation  formally follows by integrating \eqref{VPFP2d-Scale} in $v$, which gives the continuity equation for the mass density, and then integrating in $x$. On the other hand, the law for the balance of the total energy is formally derived  by summing up these relations below. First, multiplying the equation \eqref{VPFP2d-Scale} by $\frac{|v|^2}{2}$ to obtain the balance of kinetic energy
\[
\eps\dfrac{\mathrm{d}}{\mathrm{dt}}\intxvtt{\dfrac{|v|^2}{2}f^\eps} = \dfrac{q}{m}\intxvtt{ E[f^\eps]\cdot vf^\eps} - \dfrac{1}{\tau}\intxvtt{(\sigma\nabla_v f^\eps + vf^\eps)\cdot v}.
\]
Then, thanks to the continuty equation $\eps\partial_t n[f^\eps] + \Divx \intvtt{v f^\eps} =0$, we multiply this equation by $\Phi[f^\eps]$ and use the Poisson equation to find the balance of potential energy
\[
 \dfrac{\epsilon_0 \eps}{2m} \dfrac{\mathrm{d}}{\mathrm{dt}}  \intxtt{|E[f^\eps]|^2} = - \dfrac{q}{m} \intxvtt{E[f^\eps]\cdot vf^\eps}.
\]
Finally, multiplying the equation \eqref{VPFP2d-Scale} by $\sigma(1+ \ln f^\eps)$ to get the balance of entropy
\[
\eps \dfrac{\mathrm{d}}{\mathrm{dt}}\intxvtt{\sigma f^\eps \ln f^\eps } = -\dfrac{1}{\tau}\intxvtt{(\sigma\nabla_v f^\eps + vf^\eps)\cdot \dfrac{\sigma \nabla_v f^\eps}{f^\eps}}.
\]
As for weak solutions, we shall follow the same scheme. We find relations analogous to previous relations in Lemmas below. The difficulty is in overcoming the lack of regularity and the need to justify operations that are taken for granted when the solutions are smooth. We will prove these properties of solutions by combining the formal arguments above with the choice of an appropriate sequence of test functions in \eqref{WeakSolScaleVPFP} for every studied property. A similar rigorous approach that the one given in \cite{BonCarSol} and \cite{BouDol} can be easily adapted for the properties studied in our weak solution.

We start with the balance of kinetic energy.
\begin{lemma}$\;$\\
\label{BalanceKin2D}
Let $f^\eps$ be the weak solution of the problem \eqref{VPFP2d-Scale}, and \eqref{Poisson2D-Scale}, \eqref{Initial2D-Scale} provided by Theorem \ref{main_weak_sol}. Then we have
\[
\eps\dfrac{\mathrm{d}}{\mathrm{dt}}\intxvtt{\dfrac{|v|^2}{2}f^\eps} = \intxvtt{\frac{q}{m} E[f^\eps]\cdot vf^\eps} - \dfrac{1}{\tau}\intxvtt{(\sigma\nabla_v f^\eps + vf^\eps)\cdot v}.
\]
\end{lemma}
\begin{proof}$\;$\\
Let $\chi$ be a nonegative function of class $C^\infty _0(\R)$ such that
\[
\chi(s) = 1,\,\,\, \text{on}\,\,|s|\leq 1,\,\,\,\chi(s) = 0\,\,\text{on}\,\,|s|\geq 2,
\]
we define the function  $\chi_R$ as $\chi_R(z) = \chi\left( \frac{|z|}{R} \right)$. Then $\chi_R(z) =1$ on $|z|\leq R$, $\chi_R(z)=0$ on $|z|\geq 2R$ and $\|\nabla_z\chi_R\|_{L^\infty}\leq \frac{\|\chi'\|_{\infty}}{R}$. In the definition of weak solution \eqref{WeakSolScaleVPFP}, we use the test functions $\varphi(t,x,v) = \phi(t)\chi_R(x)\chi_R(v)\frac{|v|^2}{2}$ with $\phi\in C^\infty _0([0,T[)$, we obtain
\begin{align*}
&\intTxvtt{f^\eps \left[\eps \partial_{t}\phi(t)\chi_{R}(x) + v\cdot\nabla_{x}\chi_{R}(x)\phi(t)\right]\chi_{R}(v)\dfrac{|v|^2}{2}}\\
&+ \intTxvtt{f^\eps \left( E[f^\eps]+ \frac{B(x)}{\eps} {^\perp} v\right)\cdot\nabla_v\left( \chi_{R}(v)\dfrac{|v|^2}{2} \right)\phi(t)\chi_{R}(x)}\\
&+ \dfrac{1}{\tau}\intTxvtt{f^\eps\left(\sigma\Delta_{v}-v\cdot\nabla_{v}  \right)\left( \chi_{R}(v)\dfrac{|v|^2}{2} \right)\phi(t)\chi_{R}(x) }\\
&+ \intxvtt{\eps f^\eps _\mathrm{in}(x,v)\phi(0)\chi_{R}(x) \chi_{R}(v)\dfrac{|v|^2}{2}}  = 0.
\end{align*}
For each $\eps>0$, using the Theorem \ref{main_weak_sol} on the solution, we have $(1+|v|^2)f^\eps\in L^\infty(0,T;L^1(\R^2\times\R^2))$ and $E[f^\eps]\in L^\infty((0,T)\times\R^2)$. Letting $R\to\infty$, one gets, by the dominated convergence theorem, the following relation for any $\phi\in C^\infty_0([0,T[)$
\begin{align*}
\int_{0}^{T}\partial_t\phi(t) \intxvtt{\eps \dfrac{|v|^2}{2}f^\eps(t,x,v)}\mathrm{d}t + \int_{0}^{T}\phi(t)\intxvtt{E[f^\eps]\cdot v f^\eps}\mathrm{d}t \\ + \dfrac{1}{\tau}\int_{0}^{T}\phi(t)\intxvtt{(2\sigma - |v|^2)f^\eps(t,x,v)}\mathrm{d}t  + \intxvtt{\eps \dfrac{|v|^2}{2}f^\eps_{\mathrm{in}}(x,v)\phi(0)} =0.
 \end{align*}
On the other hand, by Theorem \ref{main_weak_sol}, our weak solution $f^\eps$ belongs to $L^2([0,T]\times\R^2_x, H^1(\R^2_v))$ and tends to $0$ at infinity since $(1+|v|^2)^{\gamma/2} f^\eps \in L^\infty $, thus by the divergence theorem we have
\[
\dfrac{1}{\tau}\intxvtt{(2\sigma - |v|^2)f^\eps(t,x,v)} =  -\dfrac{1}{\tau}\intxvtt{(\sigma \nabla_v f^\eps + v f^\eps)\cdot v}.
\]
Substituting into the previous relation, we easily deduce the assertions on the lemma.
\end{proof}

In the following lemma we obtain the balance of the potential energy.
\begin{lemma}$\;$\\
\label{BalancePot2D}
Let $f^\eps$ be the weak solution of the problem \eqref{VPFP2d-Scale}, \eqref{Poisson2D-Scale}, and \eqref{Initial2D-Scale} provided by Theorem \ref{main_weak_sol}. Then we have
\[
\dfrac{\epsilon_0 \eps}{2m} \dfrac{\mathrm{d}}{\mathrm{dt}}  \intxtt{|E[f^\eps]|^2} = - \dfrac{q}{m} \intxvtt{E[f^\eps]\cdot vf^\eps} = - \dfrac{1}{m}\intxtt{E[f^\eps]\cdot j[f^\eps]}.
\]
\end{lemma}
\begin{proof}$\;$\\
First, we show that $\Phi[f^\eps]$, $E[f^\eps]$ and $\partial_t E[f^\eps]$ belong to $L^\infty(0,T;L^2(\R^2))$. We will apply Lemma \ref{PropSolPoi}. The conditions in Lemma \ref{PropSolPoi} are fulfilled by the properties on the solution $f^\eps$ and the background densities $D(x)$ by  assumption $\mathrm{H}3$. Hence one gets $\Phi[f^\eps]$ and $E[f^\eps]$ lie in $L^\infty(0,T;L^2(\R^2))$. It remains to prove that  $\partial_t E[f^\eps]$ belong to $L^\infty(0,T;L^2(\R^2))^2$.
Thanks to the continuity equation on
$[0, T[  \times \R^2$ in the sense of distributions
\[
\eps \partial_t n[f^\eps] + \Divx{\intvtt{v f^\eps}} =0,
\]
see Lemma \ref{ConservationLaw} below, together with the Poisson equation \eqref{Poisson2D-Scale}, we deduce that 
\[
\partial_t E[f^\eps(t)](x) = -\dfrac{q}{2\pi\epsilon_0} \nabla_x \ln |\cdot|\star \partial_t (n[f^\eps(t)] -D) = \dfrac{1}{2\pi\epsilon_0 \eps} \nabla_x \ln |\cdot|\star (\Divx j[f^\eps(t)]).
\]
In order to estimate $\partial_t E[f^\eps(t)]$, we will use the Calderon-Zygmund inequality, see \cite{NiPouSo}, Lemma 10, but in the dual version. Let $\eta$ be a test function in $C^\infty_0([0,T[\times \R^2)$. We have
\[
\left<\dfrac{1}{2\pi\epsilon _0\eps} \nabla_x \ln |\cdot|\star (\Divx j[f^\eps])  ,\eta\right> = \intxtt{\dfrac{1}{2\pi\epsilon_0\eps} \partial^2_x \ln |\cdot|\star \eta(t,\cdot) \cdot j[f^\eps(t)]}.
\]
By a simple esimate, one gets $j[f^\eps]\in L^\infty(0,T;L^2(\R^2\times\R^2))$. Therefore we deduce that
\begin{align*}
\left|
\left<\dfrac{1}{2\pi\epsilon_0\eps} \nabla \ln |\cdot|\star (\Divx j[f^\eps])  ,\eta\right>  \right| &\leq  \dfrac{1}{2\pi\epsilon_0\eps} \| \partial^2_x \ln|\cdot|\star \eta \|_{L^2}\|j[f^\eps(t)]\|_{L^2}\\
&\leq C \|\eta\|_{L^2}\|j[f^\eps(t)]\|_{L^2}.
\end{align*}
It allows to conclude that $\partial_t E[f^\eps]$ belongs to $L^\infty (0,T;L^2(\R^2))$.

Now let $\nu>0$ and let $\kappa\in C^\infty_0(\R^2)$  be a standard mollifier. Define the regularization kernel $\kappa_\nu(x) : = \frac{1}{\nu ^2}\kappa(\frac{x}{\nu})$. Convoluting with $ \kappa_\nu $ in the equation $\Divx(\partial_t E[f^\eps]+ \frac{1}{\epsilon_0 \eps}j[f^\eps]) =0$ we obtain
\[
\Divx(\partial_t E^\nu[f^\eps]+ \dfrac{1}{\epsilon_0 \eps}j^\nu[f^\eps]) =0,
\]
where $E^\nu[f^\eps] = E[f^\eps]\star\kappa_\nu$, $j^\nu[f^\eps] = j[f^\eps]\star \kappa_\nu$. Multiplying the previous equation by $\Phi^\nu[f^\eps]\chi_R(x)$ and integrate by parts to find that
\begin{align*}
\intxtt{\partial_t E^\nu[f^\eps]\cdot E^\nu[f^\eps]\chi_R(x)} + \dfrac{1}{\epsilon_0\eps}\intxtt{E^\nu[f^\eps]\cdot j^\nu[f^\eps]\chi_R(x)} \\+ \intxtt{(\partial_t E^\nu{[f^\eps] + \dfrac{1}{\epsilon_0\eps}j^\nu[f^\eps])\cdot \Phi^\nu[f^\eps] \nabla\chi_R}} = 0,
\end{align*}
where $\Phi^\nu[f^\eps] = \Phi[f^\eps]\star \kappa_\nu$ and $\chi_R$ stands for the family of smooth cut-off functions, defined in Lemma \ref{BalanceKin2D}.
Let $\nu\to 0$. The terms on the left side converge as a consequence of the theorem of smooth approximations from the first arguments on $\Phi[f^\eps], E[f^\eps], \partial_t E[f^\eps]$. Then we obtain
\begin{align*}
\intxtt{\partial_t E[f^\eps]\cdot E[f^\eps]\chi_R(x)} + \dfrac{1}{\epsilon_0\eps}\intxtt{E[f^\eps]\cdot j[f^\eps]\chi_R(x)} \\+ \intxtt{(\partial_t E{[f^\eps] + \dfrac{1}{\epsilon_0\eps}j[f^\eps])\cdot \Phi[f^\eps] \nabla\chi_R}} = 0.
\end{align*}
Letting $R\to \infty$, the dominated convergence theorem yields
\[
\intxtt{\partial_t E[f^\eps]\cdot E[f^\eps]} + \dfrac{1}{\epsilon_0\eps}\intxtt{E[f^\eps]\cdot j[f^\eps]}  =0,
\]
which gives the result in the lemma.
\end{proof}

Finally, let us deduce the balance of entropy.
\begin{lemma}$\;$\\
\label{BalanceEntropy}
Let $f^\eps$ be a weak solution of the problem \eqref{VPFP2d-Scale}, \eqref{Poisson2D-Scale}, and \eqref{Initial2D-Scale} provided by Theorem \ref{main_weak_sol}. Then we have
\[
\eps \dfrac{\mathrm{d}}{\mathrm{dt}}\intxvtt{\sigma f^\eps \ln f^\eps } = -\dfrac{1}{\tau}\intxvtt{(\sigma\nabla_v f^\eps + vf^\eps)\cdot \dfrac{\sigma \nabla_v f^\eps}{f^\eps}}.
\]
\end{lemma}
\begin{proof}$\;$\\
Since our solution $f^\eps\in L^\infty(0,T;L^1\cap L^\infty(\R^2\times\R^2))$ and $E[f^\eps]\in L^\infty((0,T)\times\R^2)$, the proof is very similar to \cite{BouDol}, Proposition 2.3, we briefly give the idea of that.
We first show that for any $\Psi\in C^2(\R)$ such that $\Psi''\in L^\infty(\R)$ and $\Psi(0)=0$, $\Psi(f)$ solves the following equation in the sense of distribution on $[0,T[\times\R^2\times\R^2$
\begin{align}
\label{NonlinearEqu}
\eps \partial_t \Psi(f^\eps) + v\cdot \nabla_{x}\Psi(f^\eps) + \frac{q}{m} E[f^\eps]\cdot\nabla_{v}\Psi(f^\eps) + \frac{q}{m}\frac{ B(x)}{\eps} {^\perp} v\cdot\nabla_{v}\Psi(f^\eps)\nonumber \\-\dfrac{1}{\tau} v\cdot\nabla_{v}\Psi(f^\eps) - \dfrac{\sigma}{\tau}\Delta_{v}\Psi(f^\eps) = \dfrac{2}{\tau} f^\eps \Psi'(f^\eps) - \dfrac{\sigma}{\tau} \Psi''(f^\eps)\left| \nabla_v f^\eps \right|^2. 
\end{align}
The proof of equality \eqref{NonlinearEqu} comes from the same argument in \cite{BouDol}, Proposition 2.2.
We then apply \eqref{NonlinearEqu} for the following function
\[
\psi_{\delta}(f^\eps) = \left( \delta + f^\eps \right)\ln \left(  1+ \dfrac{f^\eps}{\delta}\right)+ f^\eps\ln\delta,\,\,\delta>0,
\]
and consider the test function $\varphi(t,x,v)=\phi(t)\chi_R(x)\chi_R(v)$, where the function $\chi_R$ was defined in Lemma \ref{BalanceKin2D}, and $\phi(t)\in C^1_0([0,T[)$.
Passing to the limit as $R\to\infty$, we easily deduce the following relation
\begin{align*}
\eps \dfrac{\mathrm{d}}{\mathrm{dt}}\intxvtt{\psi_{\delta}(f^\eps(t))}  + \dfrac{2}{\tau}\intxvtt{\psi_{\delta}(f^\eps)}\mathrm{d}\tau \nonumber\\
 = \dfrac{2}{\tau} \intxvtt{f^\eps\left( \ln\left(1+ \dfrac{f^\eps}{\delta}\right)+\ln\delta +1\right)} - \dfrac{\sigma}{\tau}\intxvtt{\dfrac{|\nabla_v f^\eps|^2}{\delta + f^\eps}},
\end{align*}
which is equivalent to
\begin{align}
\label{equ:LimNonLinear}
\eps \dfrac{\mathrm{d}}{\mathrm{dt}}\intxvtt{\psi_{\delta}(f^\eps(t))} + \dfrac{\sigma}{\tau}\intxvtt{\dfrac{|\nabla_v f^\eps|^2}{\delta + f^\eps}} = \dfrac{2}{\tau} \intxvtt{\left( f^\eps - \delta\ln\left(1+ \dfrac{f^\eps}{\delta}\right) \right)}.
\end{align}
Finally, 
by using the dominated convergence theorem when $\delta $ goes to zero in \eqref{equ:LimNonLinear} and then the divergence theorem  for the integral on the right hand side, we obtain
\[
\eps \dfrac{\mathrm{d}}{\mathrm{dt}}\intxvtt{f^\eps\ln f^\eps} =- \dfrac{\sigma}{\tau}\intxvtt{\dfrac{|\nabla_v f^\eps|^2}{ f^\eps}} - \dfrac{1}{\tau}\intxvtt{ v\cdot\nabla_v f^\eps}.
\]
So we complete the proof of Lemma.
\end{proof}

Now we are ready to prove Proposition 3.1.
\begin{proof}(of Proposition \ref{WeakFreeEnergy2D})$\;$\\
The mass conservation can be deduced by testing the test function $\varphi(t,x,v) = \phi(t)\chi_R(|x|)\chi_R(|v|)$ in \eqref{WeakSolScaleVPFP}.
On the other hand, using Lemmas \ref{BalanceKin2D}, \ref{BalancePot2D} and \ref{BalanceEntropy}, we imply the desired result for the balance of energy $\calE[f^\eps]$.
\end{proof}

We establish now uniform bounds for the kinetic energy.
\begin{lemma}$\;$\\
\label{BoundKinEner2D}
Assume that the initial particle densities $(f^\eps_{\mathrm{in}})$ satisfy $f^\eps _{\mathrm{in}} \geq 0$, $M_{\mathrm{in}}:= \sup_{\eps >0} M^\eps _{\mathrm{in}} < +\infty$, $U_{\mathrm{in}} := \sup_{\eps >0} U^\eps _{\mathrm{in}} < +\infty$, where for any $\eps >0$
\[
M^\eps _{\mathrm{in}} := \intxvtt{f^\eps _{\mathrm{in}} (x,v)},\,\, U^\eps _{\mathrm{in}} := \intxvtt{\dfrac{|v|^2}{2}f^\eps _{\mathrm{in}} (x,v)} + \dfrac{\epsilon_0}{2m}\intxtt{|\nabla_x \Phi[f^\eps _{\mathrm{in}}]|^2}.
\]
We assume that $(f^\eps)_{\eps>0}$ are weak solutions of \eqref{VPFP2d-Scale}, \eqref{Poisson2D-Scale}, and \eqref{Initial2D-Scale} given by Theorem 2.1. Then we have 
\[
\eps \sup_{0\leq t\leq T} \left\{ \intxvtt{\dfrac{|v|^2}{2}f^\eps(t,x,v)}+ \dfrac{\epsilon_0}{2m}\intxtt{|\nabla_x \Phi[f^\eps]|^2} \right\} \leq \eps U_{\mathrm{in}} + \dfrac{2\sigma T}{\tau}M_{\mathrm{in}},
\]
and
\[
\dfrac{1}{\tau} \intTxvtt{|v|^2 f^\eps(t,x,v)} \leq \eps U_{\mathrm{in}} + \dfrac{2\sigma T}{\tau}M_{\mathrm{in}}.
\]
\end{lemma}
\begin{proof}$\;$\\
Using Lemmas \ref{BalanceKin2D} and \ref{BalancePot2D} yields
\[
\eps \dfrac{\mathrm{d}}{\mathrm{d}t} \left\{ \intxvtt{\dfrac{|v|^2}{2}f^\eps(t,x,v)}+ \dfrac{\epsilon_0}{2m}\intxtt{|\nabla_x \Phi[f^\eps]|^2} \right\} = \dfrac{2\sigma}{\tau}M^\eps _{\mathrm{in}} -  \dfrac{1}{\tau}\intxvtt{|v|^2 f^\eps},
\]
and therefore we obtain
\begin{align*}
&\eps  \left\{ \intxvtt{\dfrac{|v|^2}{2}f^\eps(t,x,v)}+ \dfrac{\epsilon_0}{2m}\intxtt{|\nabla_x \Phi[f^\eps]|^2} \right\} + \dfrac{1}{\tau}\int_{0}^{t}\intxvtt{|v|^2 f^\eps}\mathrm{d}s \\
&= \eps U^\eps _{\mathrm{in}}+  \dfrac{2\sigma t}{\tau} M^\eps _{\mathrm{in}}
\end{align*}
which yields the results.
\end{proof}
\section{Formal derivation of the limit model}
The asymptotic behavior of the Vlasov-Fokker-Planck-Poisson equation \eqref{VPFP2d-Scale}, \eqref{Poisson2D-Scale}, and \eqref{Initial2D-Scale} when $\eps$ becomes small comes from
the balance of the free energy functional $\calE[f^\eps]$. Thanks to  Proposition \ref{WeakFreeEnergy2D}, we deduce that
\[
\eps \calE[f^\eps(t)] + \dfrac{1}{\tau}\int_{0}^{t}\intxvtt{\dfrac{|\sigma\nabla_v f^\eps + v f^\eps|^2}{f^\eps}}\mathrm{d}s = \eps \calE[f^\eps(0)].
\]
Since the dissipation term can rewrite as 
\[
\dfrac{1}{\tau}\int_{0}^{t}\intxvtt{\dfrac{|\sigma M\nabla_v (f^\eps/M)|^2}{f^\eps}}\mathrm{d}s,
\]
where $M$ stands for the Maxwellian equilibrium $M(v) = (2\pi\sigma)^{-1}\exp\left(-\frac{|v|^2}{2\sigma}\right), v\in\R^2$.
Therefore, at least formally, we deduce that $f^\eps = f+ \calO(\eps)$, as $\eps\searrow 0$, where the leading order distribution function $f$
satisfies
\[
\dfrac{1}{\tau}\intxvtt{\dfrac{|\sigma M\nabla_v (f/M)|^2}{f}}=0,\,\, t\in \R_+.
\]
Hence, we obtain $f(t,x,v) = n(t,x)M(v), (t,x,v)\in \R_+\times\R^2\times\R^2$.
Then, the question is to determine the evolution equation satisﬁed by the concentration $n(t, x) = \intvtt{f(t,x,v)}$.\\
We are looking for the model for the concentration $n[f^\eps] = \intvtt{f^\eps}$. First, by integrating
the equation \eqref{VPFP2d-Scale} with  $1$ and $v$, we straightforwardly get the local conservation laws
satisfied by the first two moments.
\begin{lemma}$\;$\\
\label{ConservationLaw}
Let $\eps>0$. Let $f^\eps$ be a weak solution of the system \eqref{VPFP2d-Scale}, \eqref{Poisson2D-Scale}, and \eqref{Initial2D-Scale} provided by Theorem \ref{main_weak_sol}. Then the following conservation laws hold in the distributional sense
\begin{equation}
\label{ContinuLaw}
\partial_t n[f^\eps] + \dfrac{1}{\eps}\Divx \left( \dfrac{j[f^\eps] }{q}\right)=0.
\end{equation}
\begin{equation}
\label{MomentLaw}
\eps \partial_t \left( \dfrac{j[f^\eps]}{q} \right) + \Divx\intvtt{v\otimes v f^\eps} - \dfrac{q}{m} E[f^\eps] n[f^\eps] - \dfrac{\omega_c(x)}{\eps}\dfrac{{^\perp}j[f^\eps]}{q} = - \dfrac{1}{\tau}\dfrac{j[f^\eps]}{q}.
\end{equation}
\end{lemma}
\begin{proof}$\;$\\
For each $\eps>0$, $(f^\eps, E[f^\eps])$ solves \eqref{VPFP2d-Scale} in the sense of distribution given by equation \eqref{WeakSolScaleVPFP} and satisfies $(1+|v|^2)f^\eps\in L^\infty(0,T;L^1(\R^2\times\R^2))$, $E[f^\eps]\in L^\infty((0,T)\times\R^2)$. Then, we test \eqref{WeakSolScaleVPFP} on the test functions of the form $\varphi(t,x,v) = \phi(t)\chi_R(x)\chi_R(v)$ and $\varphi(t,x,v) = \phi(t)\chi_R(x)\chi_R(v)v$, where the function $\chi_R$ was defined in Lemma \ref{BalanceKin2D}, and $\phi\in C^\infty_0([0,T[)$. Letting $R\to\infty$, one gets, by
dominated convergence theorem, the relations \eqref{ContinuLaw} and \eqref{MomentLaw} which hold in the distribution sense on $[0,T[\times\R^2$ and are respectively the continuity equation and
the momentum equation.
\end{proof}

Then, we apply the rotation $v \mapsto {^\perp} v$ to the equation \eqref{MomentLaw} and eliminating $\frac{1}{\eps}\frac{j[f^\eps]}{q}$ between the resulting equation and \eqref{ContinuLaw} leads to the new equation for the concentration $n[f^\eps]$.
\begin{coro}$\;$\\
\label{NewConcen}
Let $\eps>0$. Let $f^\eps$ be a weak solution of the system \eqref{VPFP2d-Scale}, \eqref{Poisson2D-Scale}, \eqref{Initial2D-Scale} provided by Theorem \ref{main_weak_sol}. Then the concentration $n[f^\eps]$ satisfies the following equation
\begin{equation}
\label{ModConcen2D}
\partial_t n[f^\eps] + \Divx\left[n[f^\eps]\left(\dfrac{^\perp E[f^\eps]}{B(x)} -\sigma \dfrac{^\perp \nabla_x \omega_c(x)}{\omega_c(x)^2}\right)\right] = \Divx F^\eps,
\end{equation}
where we denote
\[
F^\eps =  \dfrac{1}{\omega_c (x)}\left(\dfrac{\eps \partial_t {^\perp} j[f^\eps]}{q} + \dfrac{1}{\tau}\dfrac{{^\perp} j[f^\eps]}{q} + {}^\perp \Divx \intvtt{(v\otimes v -\sigma I_2 ) f^\eps}\right).
\]
\end{coro}
\begin{proof}$\;$\\
The proof of the result is obviously by observing that  the momentum flux tensor can be decomposed as
\[
\intvtt{v\otimes v f^\eps} =  \intvtt{(v\otimes v -\sigma I_2)f^\eps} + \sigma I_2 n[f^\eps].
\]
\end{proof} 

Passing formal to the limit in \eqref{ModConcen2D}, as $\eps\searrow 0$, we get
\[
\partial_t n[f] + \Divx\left[n[f]\left(\dfrac{^\perp E[f]}{B(x)} -\sigma \dfrac{^\perp \nabla \omega_c(x)}{\omega_c^2(x)}\right)\right] = 0,
\]
where we have used that $f^\eps$ tends to $f = n(t,x)M(v)$ leading to $n[f^\eps] \to n[f]$, $j[f^\eps]\to j[f] =0$ and $\intvtt{(v\otimes v -\sigma I_2)f^\eps}\to \intvtt{(v\otimes v -\sigma I_2)f} =0$. Therefore the limit model is
\begin{equation}
\label{LimitMod2D}
\partial_t n + \Divx\left[n\left(\dfrac{^\perp E[n]}{B(x)} -\sigma \dfrac{^\perp \nabla \omega_c(x)}{\omega_c^2(x)}\right)\right] = 0,\,\,(t,x)\in \R_+\times\R^2,
\end{equation}
\begin{equation}
\label{PoissonLim2D}
E[n] = -\nabla_x \Phi[n],\,\, -\epsilon_0 \Delta _x \Phi[n] = q (n -D),
\end{equation}
with the initial condition
\begin{equation}
\label{LimitInitial2D}
n(0,x) = n_\mathrm{in}(x),\,\,x\in\R^2.
\end{equation}

We have the following balances for the previous limit model
\begin{pro}$\;$\\
\label{ConserveEnerLim}
Any smooth solution of the limit model \eqref{LimitMod2D}, \eqref{PoissonLim2D}, and \eqref{LimitInitial2D} verifies the mass and free energy conservations
\[
\dfrac{\mathrm{d}}{\mathrm{dt}}\intxtt{n(t,x)} =0,\,\, \dfrac{\mathrm{d}}{\mathrm{dt}}\intxtt{\left\{\sigma n\ln n + \dfrac{\epsilon_0}{2m} |\nabla_x \Phi[n]|^2 \right\}} =0.
\]
\end{pro}
\begin{proof}$\;$\\
Clearly we have the total mass conservation. For the energy conservation, a straightforward computation, the evolution in time of the energy for the limit model can be written as
\[
\intxtt{\sigma \partial_t n (1 + \ln n)} + \intxtt{\dfrac{\epsilon_0}{m}E[n]\cdot\partial_t E[n]}.
\] 
Using the equation \eqref{LimitMod2D} for the first integral in the previous equality, we have
\[
\intxtt{\sigma \partial_t n (1 + \ln n)} = \intxtt{\sigma \left(\dfrac{^\perp E[n]}{B(x)} -\sigma \dfrac{^\perp \nabla \omega_c(x)}{\omega_c^2(x)}\right)\cdot{\nabla n} } = \sigma \intxtt{n \dfrac{^\perp E[n]}{B(x)}\cdot \dfrac{\nabla B}{B(x)}}.
\]
Thanks to Poisson's equation \eqref{PoissonLim2D}, then using again \eqref{LimitMod2D} for the second integral, we get
\begin{align*}
\intxtt{\dfrac{\epsilon_0}{m}E[n]\cdot\partial_t E[n]} = \intxtt{\dfrac{q}{m} \Phi[n]\partial_t n} &= -\dfrac{q}{m} \intxtt{E[n]\cdot n\left(\dfrac{^\perp E[n]}{B(x)} -\sigma \dfrac{^\perp \nabla \omega_c(x)}{\omega_c^2(x)}\right)}\\
&= -\sigma\intxtt{n \dfrac{^\perp E[n]}{B(x)}\cdot \dfrac{\nabla B}{B(x)}}.
\end{align*}
Combining these equalities we obtain the balance of the energy.
\end{proof}
\section{Well-posedness of the limit model}
In this section we focus on the existence, uniqueness and the properties of the solution for the limit model \eqref{LimitMod2D}, \eqref{PoissonLim2D}, and \eqref{LimitInitial2D} with regular initial data. We will construct smooth solution on any time interval $[0,T]$, $T\in\R_+$, following the same arguments as in the well posedness proof for the Vlasov–Poisson problem with external magnetic field, cf. \cite{Bos2019}. We assume that the initial condition $n_\mathrm{in}$ satisfies the hypotheses
\begin{enumerate}
\item[H4)] $n_\mathrm{in} \geq 0, \,\, |x|n_\mathrm{in}\in L^1(\R^2), \,\,n_\mathrm{in}\in W^{1,1}(\R^2)\cap W^{1,\infty}(\R^2)$,
\item[H5)] $\intxtt{n_\mathrm{in}(x)} = \intxtt{D(x)}$,
\end{enumerate}
and the external magnetic field $B(x)$ verifies
\[
B\in C^2_b(\R^2),\,\, \inf_{x\in \R^2} |B(x)| =B_0 >0.
\]
\paragraph*{Solution integrated along the characteristics.}
First, a standard computation, the equation \eqref{LimitMod2D} can be rewritten for the unknown $n/B$ as
\begin{equation}
\label{EquivLimMo2D}
{\partial _t}\left(\dfrac{n}{B}\right)  + \left( {\frac{{{}^ \bot E[n]}}{B} -\sigma \frac{{{}^ \bot \nabla {\omega _c(x)}}}{{\omega _c^2(x)}}} \right) \cdot {\nabla _x}\left(\dfrac{n}{B}\right) = 0.
\end{equation}
For any smooth field $E\in L^\infty(0,T;W^{1,\infty}(\R^2))^2$, we consider the associated characteristics flow of \eqref{EquivLimMo2D} given by
\begin{equation}
\label{equ:CharLimMo2D}
\left\{\begin{array}{l}
  \dfrac{\mathrm{d}}{\mathrm{dt}} X\left( {t;s,x} \right) = \dfrac{{^\perp E}\left( {t, X\left( {t;s,x} \right)} \right)}{B\left( { X\left( {t;s,x} \right)} \right)} -\sigma \dfrac{ {^\perp \nabla} {\omega _c}\left(  X\left( {t;s,x} \right) \right)}{\omega _c^2\left( { X\left( {t;s,x} \right)} \right)},\, t, s\in [0,T],\\
   X\left( {s;s,x} \right) = x, \,\, s\in [0,T],\,\, x\in\R^2,
\end{array}\right.
\end{equation}
where $X(t; s, x)$ is the solution of the equation \eqref{equ:CharLimMo2D}, $t$ represents the time variable, $s$ is the initial time
and $x$ is the initial position. $X(s; s, x) = x$ is our initial condition. Notice that by the hypothesis on the magnetic field $B(x)$, the vector field $\sigma \frac{^ \bot \nabla {\omega _c}}{\omega _c^2(x)}$ is also smooth with respect to $x$ and we have 
\[
\left\|\sigma \frac{^ \bot \nabla {\omega _c}}{\omega _c^2(x)}\right\|_{W^{1,\infty}(\R^2)} \leq C(\sigma, \|B\|_{W^{1,\infty}(\R^2)}, B_0).
\]
Therefore, thanks to Cauchy-Lipschitz theorem, the characteristics in \eqref{equ:CharLimMo2D} are well defined for any $(s,x)\in [0,T]\times\R^2$ and $X(t;s,x)\in W^{1,\infty}\left([0,T]\times [0,T]\times\R^2\right)^2$. Then the equation \eqref{EquivLimMo2D} can be written as
\[
\dfrac{\mathrm{d}}{\mathrm{dt}}\left[  \dfrac{n\left( {t, X\left( t;s,x \right)} \right)}{B(X(t;s,x))} \right] = 0,
\]
which yields the solution of the transport equation \eqref{EquivLimMo2D} given by
\begin{align} 
\label{SolCharac2D}
 n\left( {t,x} \right) =B(x)\dfrac{ n\left( {0, X\left( {0;t,x} \right)} \right)}{B(X(0;t,x))} = B(x)\dfrac{{{n_\mathrm{in}}\left( { X\left( {0;t,x} \right)} \right)}}{{B\left( { X\left( {0;t,x} \right)} \right)}},\,\,t\in[0,T].
\end{align}
\paragraph*{Conservation law on a volume.}
We have the following result
\begin{align}
\label{ConserLaw}
\intxtt{|n(t,x)|} = \intxtt{n_\mathrm{in}(x)},\,\, t\in[0,T].
\end{align}
Indeed, we denote $J(t;s,x)$ is the Jacobian matrix of $X(t; s, x)$ with respect to $x$ at $(t; s, x)$. Then the evolution of determinant for the Jacobian matrix $J(t; s, x)$ is given by
\begin{equation*}
\left\{\begin{array}{l}
\dfrac{\mathrm{d}}{\mathrm{dt}} \mathrm{det}J(t;s,x) = \Divx\left(\dfrac{^\perp E}{B} - \sigma \dfrac{^\perp \nabla {\omega _c}}{\omega _c^2(x)}  \right)(X(t;s,x)) \mathrm{det}J(t;s,x), \\
\mathrm{det}J(s;s,x) = 1,
\end{array}\right.
\end{equation*}
which is equivalent to
\begin{equation}
\label{JacobDeter2D}
\frac{\mathrm{d}}{\mathrm{dt}} \mathrm{det}J(t;s,x) = - \dfrac{^\perp E(t,X(t;s,x))}{B(X(t;s,x))}\cdot \dfrac{\nabla B(X(t;s,x))}{B(X(t;s,x))} \mathrm{det}J(t;s,x).
\end{equation}
On the other hand, using the equation \eqref{equ:CharLimMo2D} we deduce that
\begin{align}
\label{LnBTrajec}
  &\dfrac{\mathrm{d}}{\mathrm{dt}}\ln \left| {B\left( {X\left( {t;s,x} \right)} \right)} \right| \nonumber\\&= \dfrac{B(X(t;s,x))}{|B(X(t;s,x))|} \dfrac{{\nabla B\left( {X\left( {t;s,x} \right)} \right)}}{|{B\left( {X\left( {t;s,x} \right)} \right)}|} \cdot \left( {\dfrac{{{}^ \bot E\left( {t,X\left( {t;s,x} \right)} \right)}}{{B\left( {X\left( {t;s,x} \right)} \right)}} - \dfrac{{^\perp \nabla {\omega _c}\left( {X\left( {t;s,x} \right)} \right)}}{{{\omega _c}{{\left( {X\left( {t;s,x} \right)} \right)}^2}}}} \right)\nonumber  \\
  & = \dfrac{ \nabla B(X(t;s,x))}{B(X(t;s,x))} \cdot \dfrac{{^\perp E}(X(t;s,x))}{B{\left( {X\left( {t;s,x} \right)} \right)}}.
\end{align}
Combining \eqref{JacobDeter2D} and \eqref{LnBTrajec} yields
\[
\frac{\mathrm{d}}{\mathrm{dt}} \mathrm{det}J(t;s,x) = - \frac{\mathrm{d}}{\mathrm{dt}}\ln \left| {B\left( {X\left( {t;s,x} \right)} \right)} \right| \mathrm{det}J(t;s,x),
\]
and together with $ \mathrm{det}J(s;s,x)=1$ one gets $\left| {B\left( {X\left( {t;s,x} \right)} \right)} \right| \mathrm{det}J(t;s,x) = |B(x)|$. Therefore, integrating the equality \eqref{SolCharac2D} with respect to $x$ and then changing the variable $x$ to $X(t; 0, x)$,
we obtain 
\begin{align*}
  \intxtt{|n(t,x)|} &= \intxtt{\left| {B\left( x \right)} \right|\frac{{\left| {{n_\mathrm{in}}\left( X\left( {0;t,x} \right) \right)} \right|}}{{\left| {B\left( X\left( {0;t,x} \right) \right)} \right|}}} \\&= \intxtt {\left| {B\left( X\left( {t;0,x} \right) \right)} \right|\frac{{ {{n_\mathrm{in}}\left( x \right)} }}{{\left| {B\left( x \right)} \right|}}\mathrm{det}J\left( {t;0,x} \right)}\\
  &= \intxtt{n_\mathrm{in}(x)},
\end{align*}
which completes the proof of the equality \eqref{ConserLaw}.
\paragraph*{A priori estimates.}
We establish here a priori estimates on the solution $n(t,x)$ provided by \eqref{SolCharac2D} and its derivative.
\begin{lemma}(The bound in $L^\infty(0,T;W^{1,\infty}(\R^2))$).
\label{BoInftyLimMod2D}
Let $n(t,x)$ be a solution of \eqref{EquivLimMo2D} given by \eqref{SolCharac2D}. Then we have
\begin{equation}
\label{InftyNorm2D}
\|n(t)\|_{L^\infty(\R^2)} \leq C(\|B\|_{L^\infty(\R^2)}, B_0)\|n_\mathrm{in}\|_{L^\infty(\R^2)},\,\, t\in [0,T].
\end{equation}
\begin{equation}
\label{InftyNormGrad2D}
\|\nabla_x n(t)\|_{L^\infty(\R^2)} \leq C\left(1 +  \exp\left(\int_{0}^{t}\|E(s)\|_{W^{1,\infty}(\R^2)}\mathrm{d}s\right)\right),\,\, t\in [0,T],
\end{equation}
for some constant $C(q,m, \|n_\mathrm{in}\|_{W^{1,\infty}(\R^2)},T, \|B\|_{W^{2,\infty}(\R^2)}, B_0,\sigma)$.
\end{lemma}
\begin{proof}$\;$\\
The bound \eqref{InftyNorm2D} is obviouly from the formula \eqref{SolCharac2D} and the hypothesis of the magnetic field.
For the estimate \eqref{InftyNormGrad2D}, taking the derivative with respect to $x$ in \eqref{SolCharac2D}, we have
\begin{align}
\label{GradSolChar}
\nabla _x n(t,x)\nonumber &={^t}\left(\partial _x{X}\right)\left( {0;t,x} \right)\left[ \dfrac{\nabla_x {n_\mathrm{in}}\left( {{X}\left( {0;t,x} \right)} \right)}{{B\left( {X}\left( {0;t,x} \right) \right)}} - \dfrac{n_\mathrm{in}\left( {X}\left( {0;t,x} \right) \right)}{B(X(0;t,x))} \dfrac{\nabla B\left( {X}\left( {0;t,x} \right) \right)}{B{{\left( X\left( {0;t,x} \right) \right)}}} \right]B\left( x \right)\nonumber\\
 &+ \nabla B(x)\dfrac{n_\mathrm{in}\left( X\left( {0;t,x} \right) \right)}{B\left( X\left( {0;t,x} \right) \right)},
\end{align}
which implies that
\begin{align}
\label{IneqGradSolChar}
|\nabla_x  n(t,x)| \leq C(n_\mathrm{in}, B, B_0)(1 +  |\left({\partial _x}{X}\right)\left( {0;t,x} \right)|),
\end{align}
where $C(n_\mathrm{in}, B, B_0)$ is the constant depending on $\|n_\mathrm{in}\|_{W^{1,\infty}(\R^2)}, \|B\|_{W^{1,\infty}(\R^2)}, B_0$. It therefore remains to estimate the first derivative of $X(0;t,x)$ with respect to $x$. Taking the derivative with respect
to $x$ in \eqref{equ:CharLimMo2D}, we deduce that
\begin{align*}
  \dfrac{\mathrm{d}}{\mathrm{dt}}\left( \partial _x X \right)\left( {t} \right)\nonumber &= \left[ \dfrac{\left( {{\partial _x}{}^ \bot E} \right)\left( {t,{X}\left( t \right)} \right)  }{{B\left( {{X}\left( t \right)} \right)}} - \dfrac{{^\perp}E(t,X(t))}{B(X(t))} \otimes \dfrac{\nabla B(X(t))}{B(X(t))} \right] \partial_x X(t) \\
   &+ \left[2\sigma \dfrac{{^\perp}\nabla \omega_c (X(t))\otimes \nabla \omega_c (X(t)) }{\omega _c ^3(X(t))} 
   -\sigma\dfrac{\left( {{\partial _x}{}^ \bot \nabla {\omega _c}} \right)\left( {{X}\left( t \right)} \right) }{\omega _c ^2 \left( {{X}\left( t \right)} \right)}\right]\partial_x X(t),
\end{align*}
and after integrating in time between $s$ and $t$ we find
\[
|(\partial_ x X)(t)| \leq 1 +  C(q,m, B, B_0,\sigma)\int_{s}^{t} (\|E(\tau)\|_{W^{1,\infty}(\R^2)} + 1)|(\partial_ x X)(\tau)|\mathrm{d}\tau,
\]
where we have written $X(t)$ instead of $X(t; s, x)$  for simplicity, and $C(q,m,B,B_0,\sigma)$ stands for the constant depending only on $q,m, \|B\|_{W^{2,\infty}(\R^2)}, B_0, \sigma$. Thanks to Gronwall's inequality we deduce that
\begin{equation}
\label{EstGradCharac}
|(\partial_ x X)(t;s,x)| \leq  C(q,m,T, B, B_0,\sigma)\exp\left( \int_{s}^{t}\|E(\tau)\|_{W^{1,\infty}(\R^2)}\mathrm{d}\tau  \right).
\end{equation}
Therefore, substituting \eqref{EstGradCharac} into \eqref{IneqGradSolChar} we get
\[
|\nabla_x  n(t,x)| \leq C(q,m, n_\mathrm{in},T, B, B_0,\sigma)\left(1 +  \exp\left(\int_{0}^{t}\|E(s)\|_{W^{1,\infty}(\R^2)}\mathrm{d}s\right)\right),
\]
which yields the desired estimate.
\end{proof}
\begin{lemma}(The bound in $L^\infty(0,T;W^{1,1}(\R^2))$).
\label{BoL1LimMod2D}
Let $n(t,x)$ be a solution of \eqref{EquivLimMo2D} given by \eqref{SolCharac2D}. Then we have
\begin{equation}
\label{L1Norm2D}
\|n(t)\|_{L^1(\R^2)} = \|n_\mathrm{in}\|_{L^1(\R^2)},\,\, t\in[0,T].
\end{equation}
\begin{equation}
\label{L1NormGrad2D}
\|\nabla_x n(t)\|_{L^1(\R^2)} \leq C\left(1+ \exp\left( \int_{0}^t \|E(s)\|_{W^{1,\infty}(\R^2)}\mathrm{d}s \right)\right),\,\,t\in [0,T],
\end{equation}
for some constant $C(q,m,\|n_\mathrm{in}\|_{W^{1,\infty}(\R^2)}, \|n_\mathrm{in}\|_{W^{1,1}(\R^2)}, T, \|B\|_{W^{2,\infty}(\R^2)}, B_0,\sigma)$.
\end{lemma}
\begin{proof}$\;$\\
\eqref{L1Norm2D} is clearly from \eqref{ConserLaw}. For the estimate \eqref{L1NormGrad2D}, taking the absolute value on both sides in \eqref{GradSolChar} then integrating with respect to $x$ and changing the variable $x$ to $X(t; 0, x)$, we get
\begin{align*}
\intxtt{|\nabla_x n(t,x)|} &\leq \intxtt{|(\partial_x X)(0;t,\cdot)|\left( |\nabla_x n_\mathrm{in}(x)| + \dfrac{|\nabla B(x)|}{|B(x)|}n_\mathrm{in(x)}\right)}\\
&+ \intxtt{\dfrac{|\nabla B(X(t;0,x))|}{|B(X(t;0,x))|}n_\mathrm{in}(x)},
\end{align*}
which implies that
\[
\intxtt{|\nabla_x n(t,x)|} \leq \left(\sup_{\R^2}|\partial_x X(0;t,\cdot)| + C(\|B\|_{W^{1,\infty}(\R^2)}, B_0) \right)\|n_\mathrm{in}\|_{W^{1,1}(\R^2)}.
\]
Using the inequality \eqref{EstGradCharac} we get the estimate \eqref{L1NormGrad2D}.
\end{proof}
\paragraph*{Global existence of smooth solutions.}
We define the following set of electric field
\[
\Sigma  = \left\{ {E \in {L^\infty }{{\left( {[0,T];{{W}^{1,\infty }}\left( \R^2 \right)} \right)}^2}:{{\left\| E(t) \right\|}_{L_{t,x}^\infty }} \leq M,\,\,{{\left\| {{\partial _x}E\left( t \right)} \right\|}_{L_x^\infty }} \leq \alpha \left( t \right),\,\,t \in \left[ {0,T} \right]} \right\},
\]
where the constant $M>0$ and the function $\alpha(t): [0,T]\to {\mathbb{R}_+} $ will be determined later. Given an electric field $E$ in $\Sigma$. Considering the characteristic solution of \eqref{EquivLimMo2D} on $\R^2$, corresponding to the electric field $E$, denoted by $n^E$ which is given by the formula \eqref{SolCharac2D}. We then construct the
following map $\calF$ on $\Sigma$, whose fixed point gives the solution of the system \eqref{EquivLimMo2D}, \eqref{PoissonLim2D}, \eqref{LimitInitial2D}
\begin{align}
\label{MapFixed2D}
E \to \calF \left(E  \right)(x) =  - \frac{q}{2\pi\epsilon _0}\left( {{\nabla }\ln |\cdot| } \right){ * _x}\left( {n^E - D} \right)(x).
\end{align}
We will show that the map $\calF$ is left invariant on the set $\Sigma$ for a convenient choice of the positive constant $M$ and the function $\alpha(t)$, then we want to establish an estimate like
\begin{align}
\label{Mapconstract2D}
\| { {\calF\left(E \right)(t) - \calF(\tilde{E})} ( t)} \|_{{L^\infty }\left( {{\R^2}} \right)} \leq C_T\int_{0}^{T}{\| {( E - \tilde{E} )( t )} \|}_{L^\infty \left( \R^2 \right)}\mathrm{d}t,\,\,\forall t \in \left[ {0,T} \right],
\end{align}
for some constant $C_T$, not depending on $E, \tilde{E}$. After that, the existence of the solution of the system \eqref{EquivLimMo2D}, \eqref{PoissonLim2D}, and \eqref{LimitInitial2D}  immediately, based on the construction of an iterative method for $\calF$.
Before starting, let us recall the following classical inequality cf. \cite{Degond1986}.
\begin{lemma}
\label{ClassIneq}
Let $\rho(x)$ be a function which belongs to $L^1(\R^2)\cap W^{1,\infty}(\R^2)$ and let $U(x)$ such that 
\[
U(x) = \int_{\R^2}{\dfrac{x-y}{|x-y|^2}\rho(y)}\mathrm{d}y.
\]
Then we have the following estimates
\begin{align}
\|U\|_{L^\infty(\R^2)} &\leq C \|\rho\|^{1/2}_{L^1(\R^2)}\|\rho\|^{1/2}_{L^\infty(\R^2)},\label{LemmaE}\\
\|\nabla_x U\|_{L^\infty(\R^2)} &\leq C ( 1+\|\rho\|_{L^\infty(\R^2)} (1+\ln^+\|\nabla_x \rho\|_{L^\infty(\R^2)})+\|\rho\|_{L^1(\R^2)}),\label{LemmaGradE}
\end{align}
here the notation $\ln^+$ stands for the positive part of $\ln$.
\end{lemma}
\begin{lemma}
\label{ClosedSet2D}
There exists a positive constant $M$ and a function $\alpha(t)$ such that $\calF(\Sigma) \subset \Sigma$.
\end{lemma}
\begin{proof}$\;$\\
Let $E\in \Sigma$. Thanks to \eqref{LemmaE}, \eqref{InftyNorm2D} and \eqref{L1Norm2D} we have
\[
\|\calF(E)(t) \|_{L^\infty(\R^2)} \leq C\left(\|n_\mathrm{in}\|_{L^1(\R^2)} + \|D\|_{L^1(\R^2)}\right)^{1/2} \left( \|n_\mathrm{in}\|_{L^\infty(\R^2)} + \|D\|_{L^\infty(\R^2)} \right)^{1/2},
\]
for some constant $C(q,\epsilon_0, \|B\|_{L^\infty(\R^2)}, B_0)$.
Therefore, choosing the constant $M$ in the set $\Sigma$ as the term on the right hand side of this inequality, 
we conclude that $\sup_{t\in[0,T]}\| \calF(E)(t) \|_{L^\infty(\R^2)} \leq M$, for any $E\in \Sigma$.\\
We estimate now $\|\partial_x\calF(E)(t) \|_{L^\infty(\R^2)}$. Thanks to \eqref{LemmaGradE},  \eqref{InftyNorm2D} and \eqref{L1Norm2D} we have
\[
\|\partial_x\calF(E)(t) \|_{L^\infty(\R^2)} \leq C_0\left(1+ \ln^+(\| \nabla_x n(t) \|_{L^\infty(\R^2) } + \|\nabla_x D\|_{L^\infty(\R^2)})\right),
\]
for some constant $C_0(q,\epsilon_0, \|n_\mathrm{in}\|_{L^\infty(\R^2)}, \|n_{\mathrm{in}}\|_{L^1(\R^2)}, \|D\|_{L^1(\R^2)}, \|D\|_{L^\infty(\R^2)}, \|B\|_{L^\infty(\R^2)}, B_0)$,
which leads to estimate $\ln^+ (\|\nabla_x n(t)\|_{L^\infty(\R^2)}+\|\nabla_x D\|_{L^\infty(\R^2)})$. By inequality \eqref{InftyNormGrad2D}, we deduce that
\[
\|\nabla_x n(t)\|_{L^\infty(\R^2)}+ \|\nabla_x D\|_{L^\infty(\R^2)} \leq  C_1 \left(1 + \exp \left( \int_{0}^{t} \|\partial_x E(s)\|_{L^\infty(\R^2)} \mathrm{d}s \right) \right),
\]
for some constant $ C_1(q,m, \|n_\mathrm{in}\|_{W^{1,\infty}(\R^2)}, \|D\|_{W^{1,\infty}(\R^2)}, T, \|B\|_{W^{2,\infty}(\R^2)}, B_0)$. Using the standard inequality $1+ e^x \leq e^{x+1}$ holds for any $x\geq 0$, we deduce that
\[
\ln^+ (\|\nabla_x n(t)\|_{L^\infty(\R^2)}+ \|\nabla_x D\|_{L^\infty(\R^2)}) \leq (\ln^+ C_1 +1 ) + \int_{0}^{t} \|\partial_x E(s,\cdot)\|_{L^\infty(\R^2)} \mathrm{d}s.
\]
Finally, denoting by $C_2 = \ln^+ C_1 +1$ we have
\[
\|\partial_x\calF(E)(t) \|_{L^\infty(\R^2)} \leq C_0 C_2 +  C_0\int_{0}^{t} \|\partial_x E(s)\|_{L^\infty(\R^2)} \mathrm{d}s. 
\]
Denote by $\alpha(t)$ the
solution on $[0, T]$ of the linear equation $d\alpha/dt = C_0 \alpha(t)$ with the initial condition $\alpha(0)=C_0C_2$. Therefore, choosing the function $\alpha(t) = C_0C_2 e^{C_0 t}$ in the set $\Sigma$, then we have $\|\partial_x\calF(E)(t) \|_{L^\infty(\R^2)} \leq \alpha(t), t\in[0,T]$ for any $E\in\Sigma$.
\end{proof}

Now we are ready to establish the inequality \eqref{Mapconstract2D}. Let us consider $E, \tilde{E} \in \Sigma$ and denote by $n^E , \tilde{n}^{\tilde{E}}$ the characteristics solutions 
 of \eqref{EquivLimMo2D} and \eqref{equ:CharLimMo2D} corresponding to the electric fields $E, \tilde{E}$ respectively. It is easily seen from \eqref{LemmaE} that
\begin{equation}
\label{MapcontractBis}
\| { {\calF\left(E \right)(t) - \calF(\tilde{E})} ( t)} \|_{{L^\infty }\left( {{\R^2}} \right)} \leq C_T \|n^E(t) - \tilde{n}^{\tilde{E}}(t)\|_{L^\infty(\R^2)}^{1/2}\|n^E(t) - \tilde{n}^{\tilde{E}}(t)\|_{L^1(\R^2)}^{1/2},
\end{equation}
where $C_T$ is a positive constant, not depending on $E, \tilde{E}$. Then, the inequality \eqref{Mapconstract2D} is derived from the inequality \eqref{MapcontractBis} and Lemmas \ref{DiffNormInfty2D} and \ref{DiffNormL12D} below.
\begin{lemma}
\label{DiffNormInfty2D}
We have
\[
\|n^E(t) - \tilde{n}^{\tilde{E}}(t)\|_{L^\infty(\R^2)} \leq C_T \int_{0}^{t}\|E(s) - \tilde{E}(s)\|_{L^\infty(\R^2)}\mathrm{d}s,\,\, t\in[0,T],
\]
for some constant $C_T >0$, not depending on $E, \tilde{E}$.
\end{lemma}
\begin{proof}$\;$\\
Let us denote $X^E, \tilde{X}^{\tilde E} $ the characteristic solutions of \eqref{equ:CharLimMo2D} corresponding to $E, \tilde E$ respectively.
Thanks to the formula \eqref{SolCharac2D} we have
\begin{align*}
|n^{E}(t,x) - \tilde{n}^{\tilde{E}}(t,x)| &\leq |B(x)| \dfrac{| n_{\mathrm{in}}(X^{E}(0;t,x))- n_{\mathrm{in}}(\tilde{X}^{\tilde{E}}(0;t,x))|}{|B(X^{E}(0;t,x))|}\\
&+ |B(x)| n_{\mathrm{in}}(X^{\tilde{E}}(0;t,x))\left|\dfrac{1}{B(X^{E}(0;t,x))} - \dfrac{1}{B(\tilde{X}^{\tilde E}(0;t,x))} \right|, 
\end{align*}
which implies that
\begin{align}
\label{DiffDensity2D}
|n^{E}(t,x) - \tilde{n}^{\tilde{E}}(t,x)| &\leq C |X^E(0;t,x)-\tilde{X}^{\tilde E}(0;t,x)|,
\end{align}
for some constant $C(\|n_\mathrm{in}\|_{W^{1,\infty}(\R^2)}, \|B\|_{W^{1,\infty}(\R^2)}, B_0)$.
On the other hand, from the characteristic equation \eqref{equ:CharLimMo2D} we deduce that
\begin{align*}
  \dfrac{\mathrm{d}}{\mathrm{dt}}\left( { X^E  - \tilde{X}^{\tilde{E}} } \right)\left( {t;s,x} \right) &= \dfrac{{{}^ \bot E \left( {t, X^E\left( {t;s,x} \right)} \right)}}{{B\left( { X^E \left( {t;s,x} \right)} \right)}} - \dfrac{{{}^ \bot \tilde{E} ( {t, \tilde{X}^{\tilde E} \left( {t;s,x} \right)} )}}{B( \tilde{X}^{\tilde E} ( {t;s,x}))} \\
   &- \sigma\dfrac{{{}^ \bot \nabla {\omega _c}\left( { X^E \left( {t;s,x} \right)} \right)}}{{\omega _c^2\left( { X^E \left( {t;s,x} \right)} \right)}} + \sigma\dfrac{{{}^ \bot \nabla {\omega _c}(  \tilde{X}^{\tilde E}( {t;s,x} ))}}{{\omega _c^2(  \tilde{X}^{\tilde E} ( {t;s,x} ) )}},\\
  (  X^E  -  \tilde{X}^{\tilde E})\left( {s;s,x} \right) = 0.
\end{align*}
The first term in the right hand side of the previous equality can be estimated by 
\begin{align*}
  \left| \dfrac{{{}^ \bot E \left( {t , X^E \left(t  \right)} \right)}}{{B\left( { X^E \left( t  \right)} \right)}} - \dfrac{{}^ \bot \tilde E ( {t, \tilde{X}^{\tilde E} ( t )})}{B(  \tilde{X}^{\tilde E} ( t ))} \right| &\leq \left| \dfrac{{}^ \bot E ( t , X^{E}( t )) - {}^ \bot \tilde E ( t ,X^E( t ))}{{B( {X^E ( t)})}} \right|  \\
   &+ \left| \dfrac{{}^ \bot \tilde E (t,X^E(t)) - {}^\bot \tilde E \left( {t , \tilde{X}^{\tilde E} \left( t  \right)} \right)}{B\left( { X^E \left( t  \right)} \right)} \right| \\&+ \left| {}^ \bot \tilde E( t , \tilde{X}^{\tilde E}( t))\left( {\dfrac{1}{B( X^E ( t ) )} - \frac{1}{B( \tilde{X}^{\tilde E}( t ) )}}\right) \right|  \\
   \le \dfrac{1}{B_0}\| E( t ) - \tilde E( t) \|_{L^\infty(\R^2)} &+  C(\|B\|_{W^{1,\infty}(\R^2)},B_0, M, T)\left|\tilde X^E ( t ) - \tilde X^{\tilde E} ( t ) \right|,
\end{align*}
since $\tilde E \in \Sigma$ while the second term can be bounded by
\begin{align*}
  \left|\sigma {\dfrac{{{}^ \bot \nabla {\omega _c}\left( { X^E\left( t \right)} \right)}}{{\omega _c^2\left( {X^E \left( t  \right)} \right)}} - \sigma\dfrac{{{}^ \bot \nabla {\omega _c}\left( { \tilde{X}^{\tilde E} \left( t  \right)} \right)}}{{\omega _c^2\left( { \tilde{X}^{\tilde E} \left( t  \right)} \right)}}} \right| 
    \le C(\|B\|_{W^{2,\infty}(\R^2)}, B_0, \sigma) \left| { X^E \left( t \right) -  \tilde{X}^{\tilde E} \left( t  \right)} \right|,
\end{align*}
where we denote $(X^E(t), \tilde{X}^{\tilde E}(t)) = (X^E(t;s,x), \tilde{X}^{\tilde E}(t;s,x))$. Integrating between $s$ and $t$ the differential system of $X^E(t)-\tilde{X}^{\tilde{E}}(t)$, and combining the previous estimations  we find
\[
\left| { X^E \left( t \right) -  \tilde{X}^{\tilde{E}} \left( t \right)} \right| \leq \int_{s}^{t} {\dfrac{1}{{{B_0}}}{{\| {E \left( \tau  \right) - \tilde{E} \left( \tau  \right)} \|}_{{L^\infty(\R^2) }}}} \mathrm{d}\tau  + C\int_{s}^{t} {\left| { X^E \left( \tau  \right) -  \tilde{X}^{\tilde{E}} \left( \tau  \right)} \right|}\mathrm{d}\tau ,
\]
for some constant $C(\|B\|_{W^{2,\infty}(\R^2)}, B_0, \sigma, T, M)$.
Thanks to Gronwall's inequality one gets 
\[
\left| { X^E \left( t ;s,x\right) - \tilde{X}^{\tilde E} \left( {t;s,x} \right)} \right| \leq {e^{C\left| {t - s} \right|}}\frac{1}{{{B_0}}}\int_{s}^{t} {{{\| {E \left( \tau  \right) - \tilde{E}\left( \tau  \right)} \|}_{{L^\infty(\R^2) }}}} \mathrm{d}\tau,
\]
which together with \eqref{DiffDensity2D} yields the desired estimate of the lemma.
\end{proof}
\begin{lemma}
\label{DiffNormL12D}
We have
\[
\|n^E - \tilde{n}^{\tilde{E}}\|_{L^1(\R^2)} \leq C_T \int_{0}^{t}\|E(s,\cdot) - \tilde{E}(s,\cdot)\|_{L^1(\R^2)}\mathrm{d}s,\,\, t\in [0,T],
\]
for some constant $C_T>0$, not depending on $E, \tilde{E}$.
\end{lemma}
\begin{proof}$\;$\\
Since $n^E, \tilde n ^{\tilde E}$ are solutions of \eqref{EquivLimMo2D} corresponding to $E, \tilde E$ thus we deduce that
\begin{equation*}
\left\{\begin{array}{l}
{\partial _t}\left( n^{E} - \tilde{n}^{\tilde E} \right) + B\left( \dfrac{^\perp E}{B} - \sigma\dfrac{^ \perp \nabla {\omega _c}}{\omega _c^2} \right) \cdot \nabla _x\left( \dfrac{n^E - \tilde{n}^{\tilde E}}{B} \right) + \left( ^\perp E - {^\perp \tilde E} \right) \cdot \nabla _x\left( \dfrac{\tilde{n}^{\tilde E}}{B} \right) = 0, \\
\left( n^E - \tilde{n}^{\tilde E} \right)\left( {0,x} \right) = 0.
\end{array}\right.
\end{equation*}
Multiplying this equation by $\mathrm{sign}(n^E - \tilde{n}^{\tilde E})$ and then integrating with respect to $x$ we find
\begin{align}
\label{DeriDiffNormL1}
\dfrac{\mathrm{d}}{\mathrm{dt}}\int_{{\R^2}} {\left| n^E(t) -\tilde{n}^{\tilde{E}}(t) \right|\mathrm{d}x} & + \int_{{\R^2}} {B\left( {\frac{{{}^ \bot {E}}}{B} - \dfrac{{{}^ \bot \nabla {\omega _c}}}{{\omega _c^2}}} \right)}  \cdot {\nabla _x}\left| {\dfrac{{{n^{{E}}} - {\tilde{n}^{\tilde{E}}}}}{B}} \right|\mathrm{d}x \nonumber \\
&+ \int_{{\R^2}} {\mathrm{sign}\left( {{n^{{E}}} - {\tilde{n}^{\tilde{E}}}} \right)}  \left( {{}^ \bot{E} - {}^ \bot\tilde{E}} \right) \cdot {\nabla _x}\left( {\dfrac{{{\tilde{n}^{\tilde{E}}}}}{B}} \right)\mathrm{d}x = 0.
\end{align}
Thanks to Lemma \ref{BoL1LimMod2D} we have $n^E , \tilde n ^{\tilde E}\in W^{1,1}(\R^2)$ a.e $t\in[0,T]$ and since $\Divx\left[B\left( {\frac{{{}^ \bot {E}}}{B} - \frac{{{}^ \bot \nabla {\omega _c}}}{{\omega _c^2}}} \right)  \right] =0$ so by the divergence theorem, we obtain that
\[
\int_{{\R^2}} {B\left( {\dfrac{{{}^ \bot {E}}}{B} - \dfrac{{{}^ \bot \nabla {\omega _c}}}{{\omega _c^2}}} \right)}  \cdot {\nabla _x}\left| {\dfrac{{{n^{{E}}} - {\tilde{n}^{\tilde{E}}}}}{B}} \right|\mathrm{d}x =0.
\]
Then, from \eqref{DeriDiffNormL1} we imply
\[
\dfrac{\mathrm{d}}{\mathrm{dt}}\intxtt{\left| n^E(t) - \tilde{n}^{\tilde{E}}(t) \right|} \leq C(\|B\|_{W^{1,\infty}(\R^2)}, B_0)\|E(t)- \tilde E(t)\|_{L^\infty(\R^2)}\| \tilde{n}^{\tilde E}(t)\|_{W^{1,1}(\R^2)}.
\]
Integrating between $0$ and $t$ of this inequality leads to
\[
\| n^E (t) -\tilde{n} ^{\tilde E}(t) \|_{L^1(\R^2)} \leq C(\|B\|_{W^{1,\infty}(\R^2)},B_0)\sup_{t\in [0,T]}\|\tilde{ n}^{\tilde E}(t,\cdot)\|_{W^{1,1}(\R^2)}  \int_{0}^{t}\|E(s)- \tilde E(s)\|_{L^\infty(\R^2)}\mathrm{d}s.
\]
Finally, thanks to the estimate \eqref{L1NormGrad2D} we conclude the proof. 
\end{proof}

Now, we shall prove that the sequence of iterative method by map $\calF$  converges to a solution of the original problem. First, we consider $E_0 = 0$, then we put  $E_1 = \calF(E_0 ),..., E_{k+1} = \calF(E_k )$ for each $k\in\N$. Applying \eqref{Mapconstract2D} we have
\[
{\left\| {{E_{k + 1}}\left( t \right) - {E_k}\left( t \right)} \right\|_{{L^\infty }\left( {{\R^2}} \right)}} \leq {(C_T) ^k}\frac{{{t^k}}}{{k!}}{\left\| {{E_1}\left( t \right) - {E_0}\left( t \right)} \right\|_{{L^\infty }\left( {{\R^2}} \right)}},
\]
which yields that there exists $E\in L^\infty((0,T)\times\R^2)$ such that $E_k$ tends to $E$ in $L^\infty((0,T)\times\R^2)$ as $k\to +\infty$. Moreover, since $E_k\in \Sigma$ hence we also have $E\in \Sigma$. This allows us to define the action of the map $\calF$ on the vector field $E$ as $\mathcal{F}( E)=-\dfrac{q}{2\pi \epsilon_0}\nabla\ln|\cdot| *\left(n^{ E}-D\right)$ where $n^E$ is the solution of \eqref{EquivLimMo2D} associated with the electric field $E$. Using again \eqref{Mapconstract2D} we find
\[
\left\| E_{k + 1}\left( t \right) - \calF(E)\left( t \right) \right\|_{L^\infty \left(\R^2\right)} = \|\calF(E_k)(t) - \calF (E)(t)  \|_{L^\infty(\R^2)} \leq C_T \left\| E_{k}\left( t \right) - E\left( t \right) \right\|_{L^\infty \left(\R^2\right)},
\] 
which leads to $E_{k+1} \to \calF(E)$ in $L^\infty((0,T)\times\R^2)$ as $k\to\infty$. Therefore we get $\calF(E) =E$ and $n^E$ is the solution of \eqref{LimitMod2D}, \eqref{PoissonLim2D}, and \eqref{LimitInitial2D}. In addition, by Lemmas \ref{BoInftyLimMod2D}, \ref{BoL1LimMod2D} we conclude that $n^E\in L^\infty (0,T;W^{1,\infty}(\R^2)\cap W^{1,1}(\R^2))$. Hence, from \eqref{EquivLimMo2D}, $\partial_t n^E \in L^\infty(0,T;L^1(\R^2)\cap L^\infty(\R^2))$.
Thanks to Lemma \ref{ClassIneq}, we have $\partial_t E\in L^\infty((0,T)\times\R^2)^2$, thus $E\in W^{1,\infty}((0,T)\times\R^2)$. It remains to verify that the electric field $E$ lies in $L^\infty(0,T;L^2(\R^2))$. Applying Lemma \ref{PropSolPoi}, we need to show that $|x|n\in L^\infty(0,T;L^1(\R^2))$. Indeed, by \eqref{SolCharac2D} and the change of variable $x \mapsto X(t;0,x)$ we have
\[
\intxtt{|x||n(t,x)|} = \intxtt{|X(t;0,x)|n_\mathrm{in}(x)}.
\]
On the other hand, from \eqref{equ:CharLimMo2D} we deduce for any $t\in[0,T]$ that
\[
|X(t;0,x)| \leq |x| + C(\|E\|_{L^\infty}, \|B\|_{L^\infty(\R^2)}, B_0)T,
\]
together with $(1+|x|)n_\mathrm{in}\in L^1(\R^2)$ yields the desired result.
\paragraph*{Uniqueness of smooth solutions.}
The uniquenness of smooth solution $n(t,x)$ which belongs to $L^\infty(0,T; W^{1,1}(\R^2)\cap W^{1,\infty}(\R^2))$  is  immediately derived from the inequality \eqref{Mapconstract2D} and Gronwall's inequality. 

Based on the previous details of the arguments we establish the following result.
\begin{pro}
\label{main_sol_Lim}
Let $T>0$. Let $B\in C^2_b(\R^2)$ be a smooth magnetic field, such that $\inf_{x\in\R^2} B(x) = B_0 >0$ and the fixed background density $D$ verifies $|x|D\in L^1(\R^2)$, $D\in W^{1,1}(\R^2)\cap W^{1,\infty}(\R^2)$. Assume that the initial condition $n_\mathrm{in}$ satisfies the hypotheses $H4$, $H5$. There is a unique smooth solution $n(t,x)\geq 0$ on $[0,T]\times\R^2$ of the limit model \eqref{LimitMod2D}, \eqref{PoissonLim2D}, and \eqref{LimitInitial2D} satisfying the following properties:
\[
\intxtt{n(t,x)} = \intxtt{D(x)},\,\, t\in [0,T],
\]
\[
n\in W^{1,\infty}(0,T;L^1(\R^2))\cap W^{1,\infty}((0,T)\times\R^2),\,\, |x|n\in L^\infty(0,T;L^1(\R^2)),
\]
\[
E[n]\in W^{1,\infty}((0,T)\times\R^2),\,\, E[n]\in L^\infty(0,T;L^2(\R^2)).
\]
\end{pro}
\begin{remark}
\label{HighOrder}
From the estimates \eqref{InftyNormGrad2D}, \eqref{LemmaGradE}, and \eqref{IneqGradSolChar} we realize that there is a relation in the $L^\infty$-norm between the following quantities
$
\nabla_x  n,\,\, \partial_x  X,\,\, \partial_x  E.
$
In the same way, we can extend this relation to the higher order derivative
$
\partial_x ^2 n,\,\, \partial_x ^2 X,\,\, \partial_x ^2 E
$
by noting that the inequality \eqref{LemmaGradE} can apply to estimate $\partial_x ^2 U$ given by
\[
\|\partial_x ^2 U\|_{L^\infty(\R^2)}\leq C ( 1+\|\nabla_x \rho\|_{L^\infty(\R^2)}) (1+\ln ^+\|\partial^2 _x \rho\|_{L^\infty(\R^2)} )+\|\nabla_x \rho\|_{L^1(\R^2)}).
\]
\end{remark}

By similar arguments we can prove further regularity results for the strong solution of the limit model. The proof is standard and is left to the reader.
\begin{pro}
\label{Regularity}
Let $T>0$. Let $B\in C^3_b(\R^2)$ be a smooth magnetic field, such that $\inf_{x\in\R^2} B(x) = B_0 >0$ and the fixed background density $D$ verifies $|x|D\in L^1(\R^2)$, $D\in W^{1,1}(\R^2)\cap W^{1,\infty}(\R^2)$. Assume that the initial condition $n_\mathrm{in}$ belongs to $W^{2,1}(\R^2)\cap W^{2,\infty}(\R^2)$ and the background density $D$ lies in   $W^{2,\infty}(\R^2)$. Then the global
in time strong solution $(n, E[n])$ constructed in Proposition \ref{main_sol_Lim} satisfies
\begin{align*}
\partial^2_x n \in L^\infty(0,T; L^\infty(\R^2)),\,\, E[n] \in W^{2,\infty}((0,T)\times\R^2),\\
\partial_t \nabla_x n \in L^\infty((0,T)\times\R^2),\,\, \partial_t^2 n\in L^\infty((0,T)\times\R^2).
\end{align*}
\end{pro}

In the rest of this section, we provide some estimates on $\|\ln n\|_{L^\infty(0,T;W^{2,\infty}(\R^2))}$ if we assume that $\ln n_\mathrm{in}$ belongs to $W^{2,\infty}(\R^2)$. Let us start with the estimate of $\|\ln n\|_{L^\infty((0,T)\times\R^2)}$ in the lemma below.
\begin{lemma}
\label{BoundLoga}
Assume that $\ln n_\mathrm{in}\in L^\infty((0,T)\times\R^2)$ and $B\in C_b(\R^2)$ with $\inf_{x\in\R^2} B(x) = B_0 >0$.
Then, there exists a constant $C > 0$ depends only on $\|\ln n_\mathrm{in}\|_{L^\infty((0,T)\times\R^2)}, \|B\|_{L^\infty(\R^2)}, B_0$ and $T>0$ such that
\[
\|\ln n\|_{L^\infty(\R^2)} \leq C ,\,\,t\in [0,T].
\]
\end{lemma}
\begin{proof}$\;$\\
From the equation \eqref{EquivLimMo2D}, we deduce that
\begin{align}
\label{equ:LogaLimMod2D}
{\partial _t}\ln \left(\dfrac{n}{B}\right)  + \left( {\frac{{{}^ \bot E}}{B} -\sigma \frac{{{}^ \bot \nabla {\omega _c}}}{{\omega _c^2}}} \right) \cdot {\nabla }\ln\left(\dfrac{n}{B}\right) = 0.
\end{align}
Thanks to the formula of the characteristic solution \eqref{SolCharac2D}, we get
\begin{align}
\label{SolLogChar}
\ln\left(\dfrac{n}{B}\right)(t,x) =  \ln\left(\dfrac{n}{B}\right)(0,X(0,t,x))
\end{align}
which gives the estimate in the lemma.
\end{proof}

We next provide higher-order estimates on $\ln n$.
\begin{lemma}
\label{BoundLogHigh}
Assume that $\ln n_\mathrm{in}\in W^{2,\infty}((0,T)\times\R^2)$ and $B\in C^3_b(\R^2)$ with $\inf_{x\in\R^2} B(x) = B_0 >0$. Then we have
\[
\| \partial_t \ln n  \|_{L^\infty(\R^2)} +  \|\nabla_x \ln n\|_{L^\infty(\R^2)} \leq C_1,\,\, t\in [0,T],
\]
\[
\|\partial ^2 _x \ln n\|_{L^\infty(\R^2)} + \|\partial_{t}\nabla_{x}  \ln n\|_{L^\infty(\R^2)}\leq C_2,\,\, t\in [0,T],
\]
where the constants $C_k> 0$, $k=1, 2$ depend only on $\ln n_{\mathrm{in}}$, $ B $ et $B_0$.
\end{lemma}
\begin{proof}$\;$\\
From the equation \eqref{SolLogChar} we have
\begin{align}
\label{GradSolLogChar}
 \nabla_x \ln\left(\dfrac{n}{B}\right) = ({}^{t}\partial_x X)(0;t,x)\left(  \nabla \ln n_\mathrm{in} \right)(X(0;t,x)) -  ({}^{t}\partial_x X)(0;t,x) \left( \dfrac{\nabla B}{B} \right)(X(0;t,x)).
\end{align}
By \eqref{EstGradCharac}, the derivative in $x$ of $X(0;t,x)$ is bounded in $L^\infty((0,T)\times\R^2)$, thus we get the $L^\infty$ bound for the $\nabla_x \ln n$. Moreover, from \eqref{equ:LogaLimMod2D} we deduce that $\partial_t \ln n \in L^\infty((0,T)\times\R^2)$, together with the above discussion gives the first assertion in the lemma. \\
We next estimate $\|\partial ^2_x \ln n\|_{L^\infty(\R^2)}$. We denote by $\partial_i = \partial_{x_i}$ for $i = 1, 2$. Taking the derivative with respect to $x_i$ in the equation \eqref{GradSolLogChar} we get
\begin{align*}
\partial_i \nabla_x \ln\left(\dfrac{n}{B}\right) &= {}^{t}\left[\partial_i \partial_x X\right](0;t,x)\left(  \nabla \ln n_\mathrm{in} \right)(X(0;t,x))\\ &+ {}^{t}\partial_x X(0;t,x)\left\{ \partial_x^2  (\ln n_\mathrm{in} )(X(0;t,x))(\partial_i X)(0;t,x)\right\}\\
&- {}^{t}\left[\partial_i \partial_x X\right](0;t,x)\left( \dfrac{ \nabla B}{B} \right)(X(0;t,x)) \\ &+ {}^{t}\partial_x X(0;t,x)\left( \partial_x \left( \dfrac{ \nabla B}{B} \right)\right)(X(0;t,x))(\partial_i X)(0;t,x).
\end{align*}
By Remark \ref{HighOrder} it is well known that $\partial^2_x X(0;t,x)\in L^\infty((0,T)\times\R^2)$. Hence we obtain the $L^\infty$ bound for $\partial^2_x \ln n$.\\
Finally we estimate $\| \partial_{t}\nabla_{x}  \ln n \|_{L^\infty(\R^2)}$. Taking the time derivative in \eqref{GradSolLogChar} yields
\begin{align*}
\partial_t \nabla_x\left(\ln \dfrac{n}{B}  \right) &= ({}^{t}\partial_x (\partial_{t}X))(0;t,x)\left(  \nabla \ln n_\mathrm{in} \right)(X(0;t,x))\\
& + ({}^{t}\partial_x X)(0;t,x)  \partial ^2 _x \left(\ln n_\mathrm{in} \right)(X(0;t,x))(\partial_{t}X)(0;t,x)\\
& -({}^{t}\partial_x (\partial_{t}X))(0;t,x)\left( \dfrac{\nabla B}{B} \right)(X(0;t,x)) \\
&+ ({}^{t}\partial_x X)(0;t,x) \left( \partial_x \left(\dfrac{\nabla B}{B}\right) \right)(X(0;t,x))(\partial_{t}X)(0;t,x).
\end{align*} 
The $L^\infty$ bounds of $\partial_t X(0;t,x)$ and $\partial_x(\partial_t)X(0;t,x)$ are derived from the equation \eqref{equ:CharLimMo2D} and the regularity of $E$ and $B$. Combining two of the above discussion yields the second estimate in the lemma.
\end{proof}
\section{Convergence results}
We now concentrate on the asymptotic behavior as $\eps \searrow 0$ of the family of weak solutions $(f^\eps, E[f^\eps])_{\eps>0}$ of the Vlasov-Poisson-Fokker-Planck system \eqref{VPFP2d-Scale}, \eqref{Poisson2D-Scale}, and \eqref{Initial2D-Scale} and we establish rigorously the connection to the fluid model \eqref{equ:LimMod2D}, \eqref{LimPoisson2D}, and \eqref{LimInit2D}.
We justify the convergence of the solutions $(n[f^\eps], E[f^\eps])_{\eps>0}$ of the system \eqref{ModConcen2D} towards the solution $(n, E[n])$ of the limit problem  when $\eps$
goes to zero by performing the balance of the relative entropy between
$n^\eps$ and $n$. The proof requires some regularity properties of the limit solutions as well as the convergence of the initial data.\\
We intend to estimate the modulated energy $\calE[n^\eps(t)|n(t)]$, so we will write as
\begin{align}
\label{equ:EntropyDen2D}
\calE[n^\eps|n] &== \intxtt{\sigma n h\left(\dfrac{n^\eps}{n}\right)} + \dfrac{\epsilon_0}{2m}\intxtt{ |\nabla_x \Phi[n^\eps] -\nabla_x\Phi[n]|^2 }\nonumber\\
&= \intxtt{(\sigma n^\eps \ln n^\eps +\frac{\epsilon_0}{2m} |\nabla_x\Phi[n^\eps]|^2)} - \intxtt{(\sigma n \ln n +\frac{\epsilon_0}{2m} |\nabla_x\Phi[n]|^2)}\nonumber\\
&- \intxtt{\left\{ \sigma(1+\ln n)+ \dfrac{q}{m}\Phi[n]\right\}(n^\eps -n)}\nonumber\\
&:= \calE[n^\eps] -\calE[n] - \intxtt{k[n](n^\eps -n)},
\end{align}
where we have been denoted by $k[n] = \sigma(1+\ln n)+ \dfrac{q}{m}\Phi[n]$.
We introduce as well the modulated energy of $f^\eps$ with respect to $n^\eps M$, given by
\begin{align*}
\sigma&\intxvtt{n^\eps M h\left(\dfrac{f^\eps}{n^\eps M} \right)}+ \dfrac{\epsilon_0}{2m}\intxtt{\underbrace{|\nabla_x \Phi[f^\eps]-\nabla_x\Phi[n^\eps M]|^2}_{=0}}\\
&= \sigma \intxvtt{f^\eps \ln f^\eps - f^\eps \ln n^\eps + f^\eps \ln (2\pi\sigma)+ f^\eps \frac{|v|^2}{2\sigma}}\\
&= \intxvtt{\sigma f^\eps \ln f^\eps + f^\eps \frac{|v|^2}{2}} + \dfrac{\epsilon_0}{2m}\intxtt{|\nabla_x \Phi[f^\eps]|^2} \\
&- \intxtt{\sigma n^\eps \ln n^\eps} - \dfrac{\epsilon_0}{2m}\intxtt{|\nabla_x \Phi[n^\eps]|^2} + \sigma \ln(2\pi\sigma)\intxvtt{f^\eps}\\
&= \calE[f^\eps] - \calE[n^\eps]+ \sigma \ln(2\pi\sigma)\intxvtt{f^\eps}.
\end{align*}
Thanks to the free energy balance and mass conservation of the equation \eqref{VPFP2d-Scale} provided by Proposition \ref{WeakFreeEnergy2D} one gets
\begin{align}
\label{equ:BalanEnerDens2D}
\calE[n^\eps(t)] - \calE[n^\eps(0)] &+ \sigma\intxvtt{n^\eps(t) M h\left(\dfrac{f^\eps(t)}{n^\eps(t) M} \right)}\\
&-\sigma\intxvtt{n^\eps(0) M h\left(\dfrac{f^\eps(0)}{n^\eps(0) M} \right)}\nonumber\\
&= - \dfrac{1}{\eps\tau}\int_{0}^{t}\intxvtt{\dfrac{|\sigma \nabla_v f^\eps + v f^\eps|^2}{f^\eps}}\mathrm{d}s\nonumber.
\end{align}
Thanks to Proposition \ref{ConserveEnerLim} and together with \eqref{equ:EntropyDen2D}, \eqref{equ:BalanEnerDens2D} leads to
\begin{align}
\label{BalModEnerDens2D}
&\calE[n^\eps(t)|n(t)] +  \sigma\intxvtt{n^\eps(t) M h\left(\dfrac{f^\eps(t)}{n^\eps(t) M} \right)} + \dfrac{1}{\eps\tau}\int_{0}^{t}\intxvtt{\dfrac{|\sigma \nabla_v f^\eps + v f^\eps|^2}{f^\eps}}\mathrm{d}s\nonumber\\
&= \calE[n^\eps(0)|n(0)] +  \sigma\intxvtt{n^\eps(0) M h\left(\dfrac{f^\eps(0)}{n^\eps(0) M} \right)} - \int_{0}^{t}\dfrac{\mathrm{d} }{\mathrm{ds}}\intxtt{k[n](n^\eps -n)}\mathrm{d}s.
\end{align}
The next task is to evaluate the time derivative of $-\frac{\mathrm{d} }{\mathrm{dt}}\intxtt{k[n](n^\eps -n)}$. To start establishing, let us rewrite the model \eqref{ModConcen2D} for the concentration $n^\eps$ as following
\begin{equation}
\label{EquivModCon2D}
\partial_t n^\eps + \Divx A[n^\eps] = \Divx F^\eps,
\end{equation}
where the flux $A[n^\eps]$ is defined by $A[n^\eps] = n^\eps\left[\frac{^\perp E^\eps}{B(x)} -\sigma \frac{^\perp \nabla\omega_c(x)}{\omega_c^2(x)} \right]$. Similarly, the limit model \eqref{equ:LimMod2D} for the limit concentration $n$ can be rewritten as
\begin{equation}
\label{EquivLimMod2D}
\partial_t n + \Divx A[n] = 0,
\end{equation}
with the flux $A[n] = n\left[\frac{^\perp E[n]}{B(x)} -\sigma \frac{^\perp \nabla\omega_c(x)}{\omega_c^2(x)} \right]$. By direct formal computations, we get
\begin{align*}
-\dfrac{\mathrm{d} }{\mathrm{dt}}\intxtt{k[n](n^\eps-n)}\nonumber &=- \intxtt{\left(\sigma \dfrac{\partial_t n}{n}+\dfrac{q}{m}\partial_t \Phi[n] \right)(n^\eps -n)}-\intxtt{k[n](\partial_t n^\eps -\partial_t n)}\\
&= -\intxtt{\partial_t n \left(\sigma \dfrac{n^\eps - n}{n} + \dfrac{q}{m}(\Phi[n^\eps]-\Phi[n]) \right)}\\
& -\intxtt{\nabla_x k[n] \cdot \left(   A[n^\eps] - A[n]  -  F^\eps \right)}.
\end{align*}
We shall establish this equality for the weak solution of \eqref{EquivModCon2D} and the strong solution of \eqref{EquivLimMod2D}.
\begin{lemma}
\label{EvoluFirstTerm}
With the notations in \eqref{EquivModCon2D}, \eqref{EquivLimMod2D} we have the equality
\begin{align*}
&-\dfrac{\mathrm{d} }{\mathrm{dt}}\intxtt{\sigma(1+ \ln n)(n^\eps -n)} = - \sigma\intxtt{(n^\eps -n)\partial_t \ln n} 
- \intxtt{\left( A[n^\eps] - A[n]  \right)\cdot[\sigma\nabla_x(1+\ln n)]}\\
&+ \dfrac{\mathrm{d} }{\mathrm{dt}} \intxtt{\eps\dfrac{ ^\perp j^\eps}{q\omega_c(x)}\cdot \nabla_x[\sigma(1+\ln n)]} - \intxtt{\eps\dfrac{^\perp j^\eps}{q\omega_c(x)}\partial_{t}\nabla_x[\sigma(1+\ln n)]} \\
&+\dfrac{1}{\tau} \intxtt{\dfrac{^\perp j^\eps}{q\omega_c(x)}\cdot \nabla_x [\sigma(1+\ln n)]} + \intxtt{\left(\intvtt{( v\otimes v -\sigma I_2) f^\eps}\right): \partial_{x}\left[ \dfrac{^\perp \nabla_x[\sigma(1+\ln n)]}{\omega_c(x)} \right]}.
\end{align*}
\end{lemma}
\begin{proof}$\;$\\
From \eqref{EquivModCon2D}, \eqref{EquivLimMod2D}, we find $n^\eps -n$ satisfying the following equation in the sense of distribution
\[
\partial_{t} (n^\eps -n) + \Divx\left( A[n^\eps] - A[n]  \right) = \Divx F^\eps.
\]
Then for any test function $\varphi \in C^1 _0 \left( [0,T[\times\R^2 \right) $ we have
\begin{align}
&\intTxtt{(n^\eps - n)\partial_{t}\varphi} + \intTxtt{\left( A[n^\eps] - A[n]  \right)\cdot\nabla_x\varphi} 
+ \intTxtt{\eps \dfrac{^\perp j^\eps}{q \omega_c (x)}\cdot \partial_{t}\nabla_{x}\varphi}\nonumber \\ &-\dfrac{1}{\tau}\intTxtt{\dfrac{^\perp j^\eps}{q\omega_c (x)}\cdot\nabla_{x}\varphi}
-\intTxtt{{\intvtt{\left( v\otimes v - \sigma I_2  \right)f^\eps}}:\partial_{x}\left( \dfrac{^\perp \nabla_{x} \varphi }{\omega_c(x)}\right)}\nonumber\\
&+ \intxtt{\eps \dfrac{^\perp j^\eps _\mathrm{in}}{q\omega_c(x)}\cdot \nabla_{x}\varphi(0,x)} + \intxtt{(n^\eps _\mathrm{in} - n_\mathrm{in})\varphi(0,x)} =0 \label{WeakDiffDensi2D}.
\end{align}
We test  $\varphi(t,x)$ where $\varphi(t,x) = \theta(t)[\sigma(1+\ln n(t,x))]\chi_{R}(|x|)$ and where $\theta \in C^1_0([0,T[)$, $\chi$ was defined in Lemma \ref{BalanceKin2D}. Notice that by Lemmas \ref{BoundLoga}, \ref{BoundLogHigh}, and a standard computations, the  sequences $\partial_t \varphi, \nabla_x \varphi, \partial_t\nabla_x \varphi, \partial_x (\nabla_x\varphi) $ are uniformly bounded with respect to $R$ in $L^\infty((0,T)\times\R^2)$.
On the other hand, for each $\eps>0$, using the properties on the solution $\mathit{i.e.,}$ taking into account that $(1+|v|^2)f^\eps \in L^\infty(0,T;L^1(\R^2))$, $E^\eps\in L^\infty((0,T)\times\R^2)^2$, we can easily apply the dominated convergence as $R\to\infty$. Passing to the limit as $R \to\infty$, we get for any test function $\theta\in C^1_0([0,T[)$ that
\begin{align*}
&\intTxtt{(n^\eps -n)\partial_t\theta [\sigma(1+\ln n)]}+\sigma\intTxtt{(n^\eps -n)\theta(t)\partial_t \ln n}\\
& + \intTxtt{\left( A[n^\eps] - A[n]  \right)\cdot\theta(t)\nabla[\sigma(1+\ln n)]}+ \intTxtt{\eps\dfrac{ ^\perp j^\eps(t)}{q\omega_c(x)}\cdot \partial_t\theta \nabla[\sigma(1+\ln n)]}  \\&+
\intTxtt{\eps\dfrac{^\perp j^\eps}{q\omega_c(x)}\theta(t)\partial_{t}\nabla[\sigma(1+\ln n)]} 
- \dfrac{1}{\tau}\intTxtt{\dfrac{^\perp j^\eps}{q\omega_c(x)}\cdot \theta(t)\nabla[\sigma(1+\ln n)]} \\&-\intTxtt{\left(\intvtt{( v\otimes v -\sigma I_2) f^\eps}\right): \partial_{x}\left[\theta(t) \dfrac{^\perp \nabla[\sigma(1+\ln n)]}{\omega_c(x)} \right]}\\
&+ \intxtt{\eps \dfrac{^\perp j^\eps _\mathrm{in}}{q\omega_c(x)}\cdot\theta(0)\nabla_{x}[\sigma(1+\ln n _\mathrm{in})]} + \intxtt{(n^\eps _\mathrm{in} - n_\mathrm{in})\theta(0)[\sigma(1+ \ln n_\mathrm{in} )]} =0,
\end{align*}
which implies the desired equality in the lemma.
\end{proof}
\begin{lemma}
\label{EvoluSecdTerm}
With the notations in \eqref{EquivModCon2D}, \eqref{EquivLimMod2D} we have the equality
\begin{align*}
-\dfrac{\mathrm{d} }{\mathrm{dt}}\intxtt{\dfrac{q}{m}\Phi[n](n^\eps -n)} = -\intxtt{(n^\eps -n)\cdot \dfrac{q}{m}\partial_t\Phi[n]} - \intxtt{(A[n^\eps] - A[n])\cdot  \dfrac{q}{m} \nabla_{x}\Phi[n]}\\
+ \dfrac{\mathrm{d} }{\mathrm{dt}}\intxtt{\eps \dfrac{^\perp j^\eps}{q\omega_c(x)}\cdot \dfrac{q}{m}\nabla_{x}\Phi[n]} - \intxtt{\eps \dfrac{^\perp j^\eps}{q\omega_c(x)}\cdot  \dfrac{q}{m}\partial_t\nabla_{x}\Phi[n]}
\\+ \dfrac{1}{\tau}\intxtt{ \dfrac{^\perp j^\eps}{q\omega_c(x)}\cdot \dfrac{q}{m}\nabla_{x}\Phi[n]} +  \intxtt{\left(\intvtt{ (v\otimes v -\sigma I_2) f^\eps}\right):\partial_{x}\left( \dfrac{  \dfrac{q}{m}{^\perp} \nabla_x \Phi[n]}{\omega_c(x)} \right)}.
\end{align*}
\end{lemma}
\begin{proof}$\;$\\
We test $\varphi (t,x)= \frac{q}{m}\theta(t)\Phi[n]\chi_R(|x|)$ in \eqref{WeakDiffDensi2D}. Notice that by Proposition \ref{Regularity} we have $E[n]\in W^{2,\infty}((0,T)\times\R^2)$, which proves that $E[n]$ is continuously differential with respect to $(t,x)$. So we have $\Phi[n] \in C^1([0,T];C^2(\R^2))$. Then  we use the same argument as in Lemma \ref{EvoluFirstTerm} which yields the result of the lemma.
\end{proof}

Now we combine Lemmas \ref{EvoluFirstTerm}, \ref{EvoluSecdTerm} and further computations, we get
\begin{pro}$\;$\\
\label{TimeEvolution2D}
With the notations in \eqref{EquivModCon2D}, \eqref{EquivLimMod2D}, we have the evolution of the following equality
\[
-\dfrac{\mathrm{d} }{\mathrm{dt}}\intxtt{k[n](n^\eps-n)} = \intxtt{\dfrac{{}^\perp\nabla_x k[n]}{B(x)} \cdot (n^\eps -n)(E^\eps -E[n])} + K(t,x),
\]
where we denote by
\begin{align*}
K(t,x) = 
& \dfrac{\mathrm{d} }{\mathrm{dt}} \intxtt{\eps\dfrac{ ^\perp j^\eps}{q\omega_c(x)}\cdot \nabla_x k[n]} - \intxtt{\eps\dfrac{^\perp j^\eps}{q\omega_c(x)}\partial_{t}\nabla_x k[n]} \\
&+ \dfrac{1}{\tau}\intxtt{\dfrac{^\perp j^\eps}{q\omega_c(x)}\cdot \nabla_x k[n]} + \intxtt{\left(\intvtt{( v\otimes v -\sigma I_2) f^\eps}\right): \partial_{x}\left( \dfrac{^\perp \nabla_x k[n]}{\omega_c(x)} \right)}.
\end{align*}
\end{pro}
\begin{proof}
First thanks to Lemma \ref{PropSolPoi} and Poisson's equation, the first term on the right hand side in the equality of Lemma \ref{EvoluSecdTerm} can be written as
\[
-\intxtt{(n^\eps -n) \dfrac{q}{m}\partial_t\Phi[n]} = - \intxtt{\dfrac{q}{m}(\Phi[n^\eps] - \Phi[n])\partial_t n}.
\]
Then we together this equality with the first term on the right hand side in the equation of Lemma \ref{EvoluFirstTerm} to obtain
\begin{align}
\label{First_term_Prop6.1}
& -\intxtt{\partial_t n \left(\sigma \dfrac{n^\eps - n}{n} + \dfrac{q}{m}(\Phi[n^\eps]-\Phi[n]) \right)}= \intxtt{\Divx A[n] \left(\sigma \dfrac{n^\eps - n}{n} + \dfrac{q}{m}(\Phi[n^\eps]-\Phi[n]) \right)}\nonumber\\
 &= \intxtt{\sigma \Divx\left(\dfrac{^\perp E[n]}{B(x)} \right)(n^\eps -n)} + \sigma\intxtt{\dfrac{A[n]}{n}\cdot {\nabla  \ln n}(n^\eps -n)} + \intxtt{A[n] \cdot \dfrac{q}{m}(E^\eps -E[n])}\nonumber\\&
 = \intxtt{\nabla_x k[n]\cdot \dfrac{A[n]}{n}(n^\eps -n)}  + \intxtt{A[n] \cdot \dfrac{q}{m}(E^\eps -E[n])}.
\end{align}
Observe that
\begin{align*}
A[n^\eps]-A[n]-\dfrac{A[n]}{n}(n^\eps -n) &= n^\eps\left[ \dfrac{^\perp E^\eps}{B(x)} -\sigma \dfrac{^\perp \nabla \omega_c (x)}{\omega_c ^2(x)}\right]- n\left[ \dfrac{^\perp E[n]}{B(x)} -\sigma \dfrac{^\perp \nabla \omega_c (x)}{\omega_c ^2(x)} \right]\nonumber\\ &-\left[\   \dfrac{^\perp E[n]}{B(x)} -\sigma \dfrac{^\perp \nabla \omega_c (x)}{\omega_c ^2(x)} \right](n^\eps -n)
\\&= n^\eps \dfrac{{}^\perp (E^\eps -E[n])}{B(x)},
\end{align*}
thus we combine the first term of the last line of \eqref{First_term_Prop6.1} with the second term of the equation of Lemmas 6.1 and 6.2 to get
\[
- \intxtt{\nabla_x k[n]\cdot n^\eps \dfrac{{}^\perp (E^\eps -E[n])}{B(x)} }.
\]
On the other hand, we can write the divergence of the flux $A[n]$ in \eqref{EquivLimMod2D} as
\[
\Divx A[n] = - \Divx\left(\dfrac{n}{\omega_c(x)}{^\perp \nabla_x k[n]}\right),
\]
so the second term of the last line of \eqref{First_term_Prop6.1} can be written as
\[
-  \intxtt{{}^\perp\nabla_x k[n]\cdot n \dfrac{ (E^\eps -E[n])}{B(x)} }.
\]
Finally, we obtain the sum of the first two terms of the equation of Lemmas 6.1 and 6.2 
\begin{align*} 
& \intxtt{{}^\perp\nabla_x k[n]\cdot n^\eps \dfrac{ (E^\eps -E[n])}{B(x)} }-  \intxtt{{}^\perp\nabla_x k[n]\cdot n \dfrac{ (E^\eps -E[n])}{B(x)} }\\
&= \intxtt{\dfrac{{}^\perp\nabla_x k[n]}{B(x)} \cdot (n^\eps -n)(E^\eps -E[n])}.
\end{align*}
So, Proposition \ref{TimeEvolution2D} is proved.
\end{proof}

Coming back to \eqref{BalModEnerDens2D}, the modulated energy balance becomes
\begin{align}
\label{BalModEnerDens2DBis}
&\calE[n^\eps(t)|n(t)] +  \sigma\intxvtt{n^\eps(t) M h\left(\dfrac{f^\eps(t)}{n^\eps(t) M} \right)} + \dfrac{1}{\eps\tau}\int_{0}^{t}\intxvtt{\dfrac{|\sigma \nabla_v f^\eps + v f^\eps|^2}{f^\eps}}\mathrm{d}s\nonumber\\
&= \calE[n^\eps(0)|n(0)] +  \sigma\intxvtt{n^\eps(0) M h\left(\dfrac{f^\eps(0)}{n^\eps(0) M} \right)}\nonumber\\& + \int_{0}^{t}\intxtt{\dfrac{{}^\perp\nabla_x k[n]}{B(x)} \cdot (n^\eps -n)(E[n^\eps] -E[n])}\mathrm{d}s + \int_{0}^{t} K(s,x)\mathrm{d}s
\end{align}
where
\begin{align*}
\int_{0}^{t} K(s,x)\mathrm{d}s &= \int_{0}^{t}\dfrac{\mathrm{d} }{\mathrm{d}s} \intxtt{\eps\dfrac{ ^\perp j^\eps}{q\omega_c(x)}\cdot \nabla_x k[n]}\mathrm{d}s \nonumber\\ &- \int_{0}^{t}\intxtt{\eps\dfrac{^\perp j^\eps}{q\omega_c(x)}\partial_{s}\nabla_x k[n(s)]}\mathrm{d}s + \dfrac{1}{\tau}\int_{0}^{t}\intxtt{\dfrac{^\perp j^\eps}{q\omega_c(x)}\cdot \nabla_x k[n]}\mathrm{d}s\nonumber\\ & + \int_{0}^{t}\intxtt{\left(\intvtt{( v\otimes v -\sigma I_2) f^\eps}\right): \partial_{x}\left[ \dfrac{^\perp \nabla_x k[n]}{\omega_c(x)} \right]}\mathrm{d}s \\&:= K_1 + K_2 + K_3 + K_4 .
\end{align*}
In order to apply Gronwall's lemma, we will estimate the integrals
in the last  line  of \eqref{BalModEnerDens2DBis}.
Thanks to the formula
\begin{align*}
 (n^\eps -n)(E[n^\eps]-E[n]) &= \dfrac{\epsilon_0}{q}[\Divx(E[n^\eps]-E[n])](E[n^\eps]-E[n])\\
 &= \dfrac{\epsilon_0}{q} \Divx \left( (E[n^\eps]-E[n])\otimes (E[n^\eps]-E[n]) - \dfrac{|E[n^\eps]-E[n]|^2}{2}I_2 \right),
\end{align*}
we obtain
\begin{align*}
&\intxtt{\dfrac{{}^\perp\nabla_x k[n]}{B(x)} \cdot (n^\eps -n)(E[n^\eps] -E[n])} \\ &= \dfrac{\epsilon_0}{q}\intxtt{\left( (E[n^\eps]-E[n])\otimes (E[n^\eps]-E[n]) - \dfrac{|E[n^\eps]-E[n]|^2}{2}I_2 \right): \partial_x \left(\dfrac{{}^\perp\nabla_x k[n]}{B(x)} \right)}\\
&\leq \dfrac{\epsilon_0}{m} \left\|\partial_x \left(\dfrac{{}^\perp\nabla_x k[n]}{\omega_c(x)} \right) \right\|_{L^\infty(\R^2)}\left(1+\dfrac{\sqrt{2}}{2}\right)\intxtt{|E[n^\eps]-E[n]|^2},
\end{align*}
where for any matrix $P \in M_{2,2} (\R)$, the notation $\|P\|$ stands for $(P:P)^{1/2}$. Next we shall estmate the integrals $K_i,$ for $i=1,...,4$. For $K_1$, we have
\begin{align*}
K_1 &= \eps \intxtt{\dfrac{ ^\perp j^\eps(t,x)}{q\omega_c(x)}\cdot \nabla_x k[n(t)]} - \eps \intxtt{\dfrac{ ^\perp j^\eps(0,x)}{q\omega_c(x)}\cdot \nabla_x k[n(0)]}\\
&\leq  \dfrac{m}{qB_0}\sqrt{\eps}\intxvtt{(f^\eps(t,x,v) + f^\eps(0,x,v))\left(\eps \dfrac{|v|^2}{2} + \dfrac{\|\nabla_x k[n]\|^2_{L^\infty}}{2}\right)}.
\end{align*}
For $K_2$, an elementary estimate yields
\[
K_2 \leq  \dfrac{m}{qB_0}\|\partial_s \nabla_x k[n]\|_{L^\infty(\R^2)} \eps\int_{0}^{t}\intxvtt{\left(\dfrac{|v|^2}{2} +\dfrac{1}{2}\right)f^\eps(s,x,v)}\mathrm{d}s.
\]
For $K_3$, since $j^\eps = q\intvtt{(\sigma\nabla_v f^\eps + v f^\eps)}$ we have
\begin{align*}
K_3 &= -\dfrac{1}{\tau}\int_{0}^{t}\intxvtt{(\sigma\nabla_v f^\eps + v f^\eps)\cdot\dfrac{^\perp \nabla_x k[n]}{\omega_c(x)}}\mathrm{d}s\\
&\leq \dfrac{1}{4\eps\tau}\int_{0}^{t}\intxvtt{\dfrac{|\sigma\nabla_v f^\eps + v f^\eps|^2}{ f^\eps}}\mathrm{d}s + \dfrac{m}{qB_0 \tau}\|\nabla_x k[n]\|_{L^\infty}\eps \int_{0}^{t}\intxvtt{f^\eps}\mathrm{d}s.
\end{align*} 
For $K_4$, since 
\[
\intv{(v\otimes v -\sigma I_2 ) f^\eps}= \intv{(v f^\eps+\sigma\nabla_v f^\eps)\otimes v}
= \intv{\dfrac{(v f^\eps+\sigma\nabla_v f^\eps)}{\sqrt{ \eps f^\eps}}\otimes v \sqrt{ \eps f^\eps}},
\]
we have
\begin{align*}
K_4 \leq  \dfrac{1}{4\eps\tau}\int_{0}^{t}\intxvtt{\dfrac{|\sigma\nabla_v f^\eps + v f^\eps|^2}{ f^\eps}}\mathrm{d}s +  \left\|\partial_x \left(\dfrac{{}^\perp\nabla_x k[n]}{\omega_c(x)} \right) \right\|_{L^\infty} \eps \tau \int_{0}^{t}\intxvtt{|v|^2 f^\eps}\mathrm{d}s.
\end{align*}
Plugging the above computations in the equality \eqref{BalModEnerDens2DBis}, the modulated energy balance becomes for $0 \leq  t \leq T$
\begin{align*}
&\calE[n^\eps(t)|n(t)] +  \sigma\intxvtt{n^\eps(t) M h\left(\dfrac{f^\eps(t)}{n^\eps(t) M} \right)} + \dfrac{1}{4\eps\tau}\int_{0}^{t}\intxvtt{\dfrac{|\sigma \nabla_v f^\eps + v f^\eps|^2}{f^\eps}}\mathrm{d}s
\\
&\leq \calE[n^\eps(0)|n(0)] +  \sigma\intxvtt{n^\eps(0) M h\left(\dfrac{f^\eps(0)}{n^\eps(0) M} \right)}\\
&+  \left\|\partial_x \left(\dfrac{{}^\perp\nabla_x k[n]}{\omega_c(x)} \right) \right\|_{L^\infty(\R^2)}\left(2+{\sqrt{2}}\right)\dfrac{\epsilon_0}{2m}\intxtt{|E[n^\eps]-E[n]|^2}\\
&+ \left(\dfrac{m}{2 qB_0}\|\partial_s \nabla_x k[n]\|_{L^\infty}+\left\|\partial_x \left(\dfrac{{}^\perp\nabla_x k[n]}{\omega_c(x)} \right) \right\|_{L^\infty}\tau \right) \eps \intTxvtt{|v|^2f^\eps}\\
&+ \dfrac{m}{qB_0}\sqrt{\eps}\sup_{t\in [0,T]}\eps \intxvtt{|v|^2f^\eps}\\
&+ \dfrac{m}{qB_0}\sqrt{\eps} \left( \|\nabla_x k[n]\|_{L^\infty} + \sqrt{\eps}\dfrac{T}{2} \left( \|\partial_s \nabla_x k[n]\|_{L^\infty} +\dfrac{2}{\tau} \| \nabla_x k[n]\|_{L^\infty}\right) \right)\intxvtt{f^\eps(0,x,v)}.
\end{align*}
Thanks to Lemma \ref{BoundKinEner2D} and \eqref{equ:EntropyDen2D} for some constant $C_T$, $0\leq t \leq T$, $0< \eps<1$ we obtain
\begin{align*}
&\calE[n^\eps(t)|n(t)] +  \sigma\intxvtt{n^\eps(t) M h\left(\dfrac{f^\eps(t)}{n^\eps(t) M} \right)} + \dfrac{1}{4\eps\tau}\int_{0}^{t}\intxvtt{\dfrac{|\sigma \nabla_v f^\eps + v f^\eps|^2}{f^\eps}}\mathrm{d}s
\\
&\leq  \calE[n^\eps(0)|n(0)] +  \sigma\intxvtt{n^\eps(0) M h\left(\dfrac{f^\eps(0)}{n^\eps(0) M} \right)} + C_T \int_{0}^{t}\calE[n^\eps(s)|n(s)]\mathrm{d}s + C_T\sqrt{\eps}.
\end{align*}
Applying Gronwall's lemma, we deduce that for  $0\leq t \leq T$, $0< \eps<1$
\begin{align*}
&\calE[n^\eps(t)|n(t)] +  \sigma\intxvtt{n^\eps(t) M h\left(\dfrac{f^\eps(t)}{n^\eps(t) M} \right)} + \dfrac{1}{4\eps\tau}\int_{0}^{t}\intxvtt{\dfrac{|\sigma \nabla_v f^\eps + v f^\eps|^2}{f^\eps}}\mathrm{d}s
\\
&\leq \left[ \calE[n^\eps(0)|n(0)] +  \sigma\intxvtt{n^\eps(0) M h\left(\dfrac{f^\eps(0)}{n^\eps(0) M} \right)} + C_T\sqrt{\eps} \right] e^{C_T t}.
\end{align*}
The above inequality says that the particle density $f^\eps$ remains close to the Maxwellian with the same concentration, $\mathit{i.e.,}$ $n^\eps(t)M$, and $n^\eps(t)$ stays near $n(t)$, provided that analogous behaviour occur for the initial conditions. Therefore, we are ready to prove our main theorem.
\begin{proof}(of Theorem \ref{MainThm2D})$\;$\\
We justify the convergence of $f^\eps$ toward $nM$ in $L^\infty(0,T;L^1(\R^2\times\R^2))$, the other convergences being obvious. We use the Csis\'ar -Kullback inequality in order to control the $L^1$ norm by the relative entropy, cf. \cite{Csi1967, Kul1967}
\[
\int_{\R^n}|g-g_0|\mathrm{d}x \leq 2 \max \left\{ \left( \int_{\R^n}g_0\mathrm{d} x \right)^{1/2}, \left( \int_{\R^n}g\mathrm{d} x\right)^{1/2} \right\}\left(\int_{\R^n}g_0  h\left(\dfrac{g}{g_0} \right)\mathrm{d}x \right)^{1/2}
\]
for any non negative integrable functions $g_0,g: \R^n \to \R$. Applying two times the  Csis\'ar -Kullback inequality we obtain
\begin{align*}
&\intxvtt{|f^\eps(t,x,v) -n(t,x)M(v)|} \\ &\leq \intxvtt{|f^\eps(t,x,v)- n^\eps(t,x)M(v)|} + \intxtt{|n^\eps(t,x)-n(t,x)|}\\
&\leq 2 \sqrt{M_{\mathrm{in}}}\left( n^\eps(t)M(v) h\left(\dfrac{f^\eps(t)}{n^\eps(t)M} \right)\right)^{1/2} \\&+ 2 \max\left\{ \sqrt{M_{\mathrm{in}}},\sqrt{\|n_{\mathrm{in}}\|_{L^1(\R^2)}} \right\} \left( \intxtt{n(t) h \left(\dfrac{n^\eps(t)}{n(t)}\right)}  \right)^{1/2}  \to 0 ,\,\,\mathrm{as}\,\eps\searrow 0.
\end{align*}
\end{proof}
\paragraph*{Appendix A. The Vlasov-Fokker-Planck equation with external magnetic field.} In Section 2, we built the existence of regularized solution of \eqref{eq:VFP2DBis} relied on the existence of solutions to the linear kinetic equations. The purpose of this part is to prove this result. 
We consider the Vlasov-Fokker-Planck equation with a given electric field $E(t,x) = -\nabla_x \Phi(t,x)$
\begin{equation}
\left\{\begin{array}{l}
\label{equ:LiVFPAppen}
\partial_t f + v\cdot \nabla _x f + E(t,x)\cdot \nabla _v f + B(x)  ^\perp v\cdot\nabla _v f = \sigma \Delta_v f + \Divv(v f) ,\\
f(0,x,v) =  f_{\mathrm{in}}(x,v),\,\,(x,v)\in \R^2\times\R^2.
\end{array}\right.
\end{equation}
The results on the existence and uniqueness of solutions are deeply inspired by those given by Degond in \cite{Degond1986}, also taking into account the velocity transport, generated by the external magnetic field. We have the following result
\begin{thm}$\;$\\
\label{ExiVFP2D}
For a given $T\in ]0,\infty[$. Let $f_\mathrm{in}$ be an initial data verifying $\mathrm{H1}$, $\mathrm{H2}$, $E(t,x)$ be an external electric field belongs to $ (L^\infty((0,T)\times\R^2))^2$, and $B\in L^\infty(\R^2)$. Then there exists a unique positive weak solution of the equation \eqref{equ:LiVFPAppen} on the interval $[0,T]$ in the sense of Definition \ref{DefWeakSol} provided by Proposition \ref{PropExiUniq} such that $f\in L^\infty(0,T;L^\infty\cap L^1(\R^2\times\R^2))$. Furthermore, $f$ belongs to $L^2([0,T]\times\R^2_x;H^1(\R^2_v))$ and verifies the following estimates
\begin{align}
\label{InegLp}
\| f \|_{L^\infty (0,T;L^{p}(\R^2\times\R^2))} \leq e^{\frac{p-1}{p}2T} \| f_{\mathrm{in}} \|_{L^{p}(\R^2\times\R^2)},\,\,p\in]1,\infty[, \\
\| f\|_{L^\infty(0,T;L^{1}(\mathbb{R}^2\times\mathbb{R}^2))} = \| f_{\mathrm{in}}\|_{L^{1}(\R^2\times\R^2)},\,\,
\| f\|_{L^\infty((0,T)\times\R^2\times\R^2)}\leq e^{2T}\| f_{\mathrm{in}}\|_{L^\infty(\R^2\times\R^2)},\nonumber
\end{align}
\begin{equation}
\label{InegKinEner}
\sup_{[0,T]}\intxvtt{f(t,x,v)\dfrac{|v|^2}{2}} < C(\|E\|_{L^\infty}, T,\sigma) \intxvtt{f_{\mathrm{in}}(x,v) \dfrac{|v|^2}{2}},
\end{equation}
\begin{equation}
\label{InegPosition}
\sup_{[0,T]}\intxvtt{f(t,x,v)|x|} < C(T) \intxvtt{f_{\mathrm{in}}(x,v) |x|},
\end{equation}
\begin{equation}
\label{InegDissipation}
\|\nabla_v  f^{1/2}\|_{L^2(0,T;L^{2}(\R^2\times\R^2))} \leq C(\|E\|_{L^\infty},T, f_\mathrm{in},\sigma) + \intxvtt{\sigma f_\mathrm{in}|\ln f_\mathrm{in}|}.
\end{equation}
\end{thm}

Let us introduce the Hilbert space
\[
\calH =L^2 ([0,T]\times \R ^2 _x , H^{1}(\R ^2 _v)) = \left\{u\in L^2([0,T]\times\R^2\times\R^2)\,|\,\nabla_v u \in  L^2([0,T]\times\R^2\times\R^2)\right\},
\]
with norm $\|\cdot\|_{\calH}$ and scalar product $\left<\cdot,\cdot \right>_{\calH}$ defined by
\[
\| u \|^2_{\calH} = \intTxvtt{\vert u\vert ^2} + \intTxvtt{\vert \nabla _v u\vert ^2},\,\, u\in \calH,
\]
\[
\left<u,w \right>_{\calH} = \intTxvtt{ u w} + \intTxvtt{ \nabla_v u\cdot\nabla_v w},\,\, u, w\in \calH.
\]
We also denote $\calH'$ is the dual space of $\calH$ which is given by $\calH '= L^2 ([0,T]\times \R ^2 _x , H^{-1}(\R ^2 _v))$. The symbole $\left<\cdot ,\cdot\right>_{\calH ',\calH}$ represents the dual relation between $\calH$ and its dual.\\
We first state a result on the existence and uniqueness of a weak solution of equation \eqref{equ:LiVFPAppen} in an $L^2$ setting, which
can be rewritten in the following form
\[
\partial_t f + \calT f + E(t,x) \cdot\nabla_v f - 2 f -\sigma \Delta_v f = 0,
\]
where $\calT$ denotes the transport operator given by $\calT =  v\cdot\nabla_x  + (B(x) ^\perp v -v)\cdot\nabla_v $. Then we have the following result
\begin{pro}$\;$\\
\label{PropExiUniq}
Under the hypothesis of Theorem \ref{ExiVFP2D}, there exists a unique weak solution f of equation \eqref{equ:LiVFPAppen} in the class of functions $\Y$ defined by
\begin{align}
\label{ClassWeakSol}
\Y= \left\{ u\in \calH |\,\,\,\frac{\partial u}{\partial t}+ \calT u \in \calH '  \right\},
\end{align}
and satisfying the initial condition in the sense of distribution.
\end{pro}
We first recall the theorem of Lions \cite{Lions61}, already used in \cite{Degond1986}.
\begin{thm}$\;$\\
\label{LionsThm}
Let $F$ be a Hilbert space, provided with a norm $\|\cdot\|_{F}$ and  scalar product $\left(,\right)$. Let $\calV$ be a subspace of $F$ with a prehilbertian norm   $\|\cdot\|_{\calV}$ such that the injection  $\calV \hookrightarrow\calH$ is continuous. We consider a bilinear form $\calE$ 
\begin{align*}
\calE: F\times \calV &\to \R\\
   (u ,\phi ) &\mapsto \calE(u,\phi),
\end{align*}
such that  $\calE(\cdot,\phi) $ is continuous on $F$, for any fixed $\phi\in \calV$, and such that
\[
\vert \calE(\phi ,\phi )\vert \ge \alpha \Vert\phi\Vert_\calV ^2,\,\, \phi\in \calV, \,\alpha >0.
\]
Then given a linear form $L$ in $\calV'$, there exists a solution $u$ in $F$ of problem
\[
\calE(u,\phi)= L(\phi),\,\,\text{for any}\,\, \phi\in \calV.
\]
\end{thm}
\begin{proof}(of Proposition \ref{PropExiUniq})$\;$\\
We follow exactly the proof in \cite{Degond1986}. First make the change
of unknown function $\tilde f(t,x,v) =e^{-(\lambda + 2) t}f(t,x, e^{-t}v) $, with any $\lambda >0$ so that $\tilde f$ satisfies the equation
\begin{equation}
\label{equ:NewVFP2D}
\left\{\begin{array}{l}
\dfrac{\partial \tilde{f} }{\partial t} +  e^{-t}v\cdot\nabla_x \tilde{f} + B(x) ^\perp v \cdot\nabla_v \tilde{f} + e^t  E(t,x)\cdot\nabla_v \tilde{f} + \lambda \tilde{f} -\sigma e^{2t}\Delta_v \tilde{f} = 0,\\
\tilde{f}(0,x,v)= \tilde{f}_{\mathrm{in}} (x,v) = f_\mathrm{in}(x,v).
\end{array}\right.
\end{equation}
Now, let $F$ be equal to the space $\calH$ and let $\calV$ be the space $C^{\infty} _{0} ([0,T[\times \R ^2\times\R ^2)$. $\calV$ is equipped with a prehilbertian norm defined by
\[
 \Vert \phi\Vert ^ 2 _{\calV} = \frac{1}{2}\intxvtt{\vert \phi(0,x,v)\vert ^2}+\|\phi\| ^2 _\calH ,\,\, \phi\in \calV.
 \]
 A weak solution of equation \eqref{equ:NewVFP2D} in the distribution sense is a function $\tilde f \in\calH$ such that
\begin{align}
\label{equ:WeakFormNVFP}
\intTxvtt{\tilde{f}\left( -\partial_t \phi - e^{-t}v\cdot\nabla_x\phi - B(x){}^\perp v \cdot\nabla_v\phi +\lambda\phi \right)}\nonumber \\
+ \intTxvtt{\nabla_v\tilde{f}\cdot\left( e^{t}E(t,x)\phi +\sigma e^{2t}\nabla_v\phi  \right)} = \intxvtt{\tilde{f}_\mathrm{in}(x,v) \phi(0,x,v)},
\end{align}
for any $\phi\in \calV$.
 We consider the following bilinear form $\calE$ as the left-hand side of the variational equation \eqref{equ:WeakFormNVFP} defined by
 \begin{align*}
\calE(\tilde{f},\phi) =& \intTxvtt{\tilde{f}\left( -\partial_t \phi - e^{-t}v\cdot\nabla_x\phi - B(x){}^\perp v \cdot\nabla_v\phi +\lambda\phi \right)}\\
&+ \intTxvtt{\nabla_v\tilde{f}\cdot\left( e^{t}E(t,x)\phi +\sigma e^{2t}\nabla_v\phi  \right)},
\end{align*}
and the linear form
\[
L(\phi)=\intxvtt{\tilde{f}_\mathrm{in}(x,v) \phi(0,x,v)}.
\]
Now, let us check $\calE$ satisfies the properties stated in Theorem \ref{PropExiUniq}. It is easily seen that $\calE(\cdot,\phi)$ est continue sur $\calH$ since $E\in (L^\infty((0,T)\times\R^2))^2$. It remains to show that $\calE$ is coercivity on $\calV\times\calV$. Indeed, for any $\phi\in\calV$, by a simple computation we have
\begin{align*}
\calE(\phi,\phi)& = \dfrac{1}{2}\intxvtt{|\phi(0,x,v|^2}+\sigma\intTxvtt{e^{2t}|\nabla_v \phi|^2} \\&+ \lambda\intTxvtt{|\phi|^2} \geq   \min\left(1, \sigma, \lambda \right) \| \phi \|^2_{\calV}.
\end{align*}
Then Lion’s Theorem \ref{LionsThm} applies and we get that variational equation $ \calE(\tilde{f},\phi) = L(\phi)$, for any $ \phi\in \calV$ admits a solution $\tilde f \in\calH$. Moreover, $\tilde f$ satisfies the equation \eqref{equ:WeakFormNVFP} for any $\phi\in\calV$, hence by using the test function $\tilde{\phi} = e^{(\lambda + 2)t} \phi(t,x,e^t v)$ we deduce that  $f(t,x,v) = e^{(\lambda +2) t}\tilde{f}(t,x,e^t v)$ is a weak solution of \eqref{equ:LiVFPAppen} in the sense of distribution. This gives that
\[
\frac{\partial f }{\partial t} + \calT f= - E(t,x) \cdot\nabla_v f + 2f + \sigma \Delta_v f \in \calH ',
\]
so that $f$ belongs to $\Y$.

We shall call the following Lemma to give a meaning to the initial condition, and also, to show the uniqueness. The proof is very close to
the one of Lemma A.1 in \cite{Degond1986} and we have been left behind.
\begin{lemma}$\;$\\
\label{GreenFormulas}
1. For $u\in \Y$, $u$ admits continuous trace values $u(0,x,v)$ and $u(T,x,v)$ in $L^2(\R ^2\times\R ^2)$.
This means that the linear map $u \to (u(0,\cdot,\cdot),u(T,\cdot,\cdot))$ is continuous from $\Y$ to $L^2(\R^2\times\R^2)$.\\
2.  For $f$ and $\tilde{f}$ in $\Y$ we have
\begin{align}
\label{InteGreen}
&\left< \partial_t f +\calT f , \tilde{f}\right>_{\calH ' \times\calH } + \left< \partial_t \tilde f +\calT\tilde{f}, f\right>_{\calH ' \times\calH }
= 2 \intTxvtt{f \tilde{f}}\nonumber \\&+\intxvtt{f(T,x,v)\tilde{f}(T,x,v)}-\intxvtt{f(0,x,v)\tilde{f}(0,x,v)},
\end{align} 
where $\calT=  v\cdot\nabla_x + (B(x) ^\perp v -v)\cdot\nabla_v$.\\
3. Similary, for $f$ and $\tilde{f}$ in $\Y$ we have
\begin{align}
\label{InteGreenBis}
&\left< \partial_t f +\calT' f , \tilde{f}\right>_{\calH ' \times\calH } + \left< \partial_t \tilde f +\calT'\tilde{f}, f\right>_{\calH ' \times\calH }\nonumber
\\&= \intxvtt{f(T,x,v)\tilde{f}(T,x,v)}-\intxvtt{f(0,x,v)\tilde{f}(0,x,v)},
\end{align}
where $\calT'= e^{-t} v\cdot\nabla_x + B(x) ^\perp v \cdot\nabla_v$. 
\end{lemma}

Let us now end the proof of Proposition \ref{PropExiUniq}. Using formula \eqref{InteGreen} to the solution $f$ of equation \eqref{equ:LiVFPAppen} and test function $\phi$ in $\calV$ we have
\begin{align}
\label{GreenBis1}
\left< \partial_t f +  \calT f , \phi\right>_{\calH ' \times\calH } &+ \left<  \partial_t \phi  + \calT\phi, f\right>_{\calH ' \times\calH }
\\&= 2 \intTxvtt{f\phi}-\intxvtt{f(0,x,v)\phi(0,x,v)}\nonumber.
\end{align}
As $f$ is a solution of \eqref{equ:LiVFPAppen} in $\calH'$ then we get
\begin{align*}
\left< \partial_t  f+  \calT f , \phi\right>_{\calH ' \times\calH } &= \left< -E(t,x)\cdot \nabla_v f - (\lambda -2)f +\sigma\Delta_v f,\phi  \right>_{\calH ' \times\calH }
\\&=- \intTxvtt{(E(t,x)\cdot \nabla_v f\phi + (\lambda -2)f\phi +\sigma\nabla_v f \cdot\nabla_v \phi) }.
\end{align*}
Furthermore, $f$ satisfies the variational equality $\calE(f,\phi) = L(\phi)$ thus
\begin{align*}
\left<  \partial_t \phi + \calT\phi, f\right>_{\calH ' \times\calH } &= \intTxvtt{\lambda f\phi + \nabla_v f\cdot (E(t,x)\phi+\sigma\nabla_v\phi)}\\ &-\intxvtt{f_\mathrm{in}(x,v)\phi(0,x,v)}.
\end{align*}
Substituting into \eqref{GreenBis1} which yields
\[
\intxvtt{(f(0,x,v)-f_\mathrm{in}(x,v))\phi(0,x,v)} =0,\,\,\forall \phi\in\calV.
\]
Therefore, the initial condition is satisfied in $L^2(\R^2\times \R^2)$. Now for uniqueness, we assume that $f$ is a solution of \eqref{equ:LiVFPAppen} with $f_\mathrm{in} =0$, which belongs to $\Y$. Proceeding as in the existence solution of Proposition \ref{PropExiUniq}, we define the function $\tilde f$ as $\tilde f(t,x,v) =e^{-(\lambda +2) t}f(t,x, e^{-t}v) $ which verifies equation \eqref{equ:NewVFP2D} with zero initial data.  We apply the formula \eqref{InteGreenBis} to the solution $\tilde{f}$ of equation \eqref{equ:NewVFP2D} which gives
\begin{align*}
0 &= \left< \partial_t \tilde f + \calT'\tilde{f},\tilde{f}\right>_{\calH ' \times\calH } + \left<e^t E(t,x)\cdot\nabla_v \tilde{f} + \lambda \tilde{f} -\sigma  e^{2t}\Delta_v \tilde{f} ,\tilde{ f} \right>_{\calH ' \times\calH }\\
&= \frac{1}{2}\intxvtt{|\tilde{f}(T,x,v)|^2} + \lambda \intTxvtt{|\tilde{f}|^2} +\sigma\intTxvtt{e^{2t}|\nabla_v \tilde{f}|^2}
\\&\geq  \lambda \intTxvtt{|\tilde{f}|^2}.
\end{align*}
Therefore we get $\tilde f =0$, which proves uniqueness.
\end{proof}
\begin{proof}(of Lemma \ref{GreenFormulas})$\;$\\
Let us consider set $Y$ of $C^\infty$ functions of $(x,t)$ in $[0,T]\times\R^2_x$ with values in $H^1(\R^2_v)$ which are compactly supported in $[0,T]\times\R^2\times\R^2$. Following the arguments in Lemma $A.1$ in \cite{Degond1986}, we have that the set $Y$ is dense on $\Y$.\\
Let us take $u\in Y$. Using a partition of unity we can assume, without of loss of generality, that $u$ vanishes on $\left\{(0,x,v): (x,v)\in\R^2\times\R^2 \right\}$ or $\left\{(T,x,v): (x,v)\in\R^2\times\R^2 \right\}$. Assume that $u$ does not vanish on $\left\{(0,x,v): (x,v)\in\R^2\times\R^2 \right\}$. By identity \eqref{InteGreen} we have
\begin{align*}
\intxvtt{|u(0,x,v)|^2} &= -2 \intTxvtt{u\left[\partial_t +  v\cdot\nabla_x + (B(x) ^\perp v -v)\cdot\nabla_v \right]u}\\
&+ 2 \intTxvtt{|u|^2} \\&\leq 2  \left\| \left[\partial_t +  v\cdot\nabla_x + (B(x) ^\perp v -v)\cdot\nabla_v \right]u  \right\|_{\calH '}\|u \|_{\calH } + 2 \|u \|^2_{\calH }  \leq C \|u\|^2_{\Y}.
\end{align*}
The rest of the lemma follows from straightforward arguments involving the density of $Y$ in $\Y$.
\end{proof}

The following Proposition is devoted to a maximum principle and an $L^\infty$ estimate.
\begin{pro}$\;$\\
\label{NormInfty2D}
Assume that the initial condition $f_\mathrm{in}$ is positive and belongs to $L^\infty(\R^2\times\R^2)$. Then the solution f provided by Proposition \ref{PropExiUniq} is positive and satisfying
\[
\|f(t)\|_{L^\infty(\R^2\times\R^2)} \leq e^{2T}\|f_{\mathrm{in}}\|_{L^\infty(\R^2\times\R^2)},\,\,t\in [0,T].
\]
\end{pro}
We start by giving the following Lemmas. The proof of these Lemmas are very close to those given by in \cite{Degond1986}. We leave it to the reader.
\begin{lemma}$\;$\\
\label{LemNormInftyBis1}
Let $f\in\Y$ then $f^+$ and $f^-$ defined by
$
f^+ = \max(f,0)\,\,\,\,\,f^- = \max(-f,0)
$
belong to $\calH$ and
$\nabla_v f^+ = \dfrac{1+{\mathrm{sign}}(f)}{2}\nabla_v f ,\,\,\,\, \nabla_v f^- = \dfrac{-1+\mathrm{sign}(f)}{2}\nabla_v f$.
Furthermore, we have
\begin{align}
&\left<\partial_t f + \calT' f,f^ -\right>_{\calH ' \times\calH }\nonumber \\
&=
\dfrac{1}{2}\left(\intxvtt{f(T,x,v)f^- (T,x,v)}-\intxvtt{f(0,x,v)f^- (0,x,v)}\right) \label{IntMaxMin1}
\end{align}
where $\calT' = e^{-t} v\cdot\nabla_x + B(x)^\perp v \cdot \nabla_v$. Similarly, we also have
\begin{align}
&\left< \partial_t f + \calT f,f^ -\right>_{\calH ' \times\calH } = \intTxvtt{f f^- }\nonumber \\
& +\dfrac{1}{2}\left(\intxvtt{f(T,x,v)f^- (T,x,v)}-\intxvtt{f(0,x,v)f^- (0,x,v)}\right)\label{IntMaxMin2}
\end{align}
where $\calT =  v\cdot\nabla_x + (B(x)^\perp v -v)\cdot \nabla_v$.
\end{lemma}
\begin{lemma}$\;$\\
\label{LemNormInftyBis2}
Let $\mathbf{V}\subset \mathbf{H}\subset \mathbf{V}'$ be a canonical triple of Hilbert spaces. We suppose
that the mapping $u\to u^-$ is a contraction on $\mathbf{V}$. Let $u$ belong to $L^2(0,T;\mathbf{V})\cap C^0([0,T];\mathbf{H})$ such that $\frac{du}{dt}\in  L^2(0,T;\mathbf{V}')$. Then
\begin{align}\label{IntTime}
\int^{T}_{0}{\left< \dfrac{du}{dt} , u^- \right>_{\mathbf{V}'\times\mathbf{V}}}\mathrm{d}t = \dfrac{1}{2}\left( |u^-(0)|^2 _{\mathbf{H}}- |u^-(T)|^2 _{\mathbf{H}} \right).
\end{align}
\end{lemma}
\begin{proof}(of Proposition \ref{NormInfty2D})$\;$\\
We will first show that $f\geq 0 $ a.e. As above, we define $\tilde{f} = e^{-(\lambda +2) t} f(t,x,e^{-t}v)$ with any $\lambda >0$ which solves \eqref{equ:NewVFP2D} with the initial data $f_\mathrm{in}$. It is well known that $\tilde f \in \Y$ since $f\in\Y$ and thus  $\partial_t \tilde f +\calT' \tilde f\in \calH'$. Thanks to Lemma \ref{LemNormInftyBis1} we have $\tilde f^- \in \calH$ which implies from \eqref{equ:NewVFP2D} that
\[
\left<  \partial_t \tilde{f} + \calT' \tilde{f},\tilde{f} ^ -\right>_{\calH ' \times\calH } + \left<e^t E(t,x)\cdot\nabla_v \tilde{f} + \lambda \tilde{f} -\sigma e^{2t} \Delta_v \tilde{f} , \tilde{f}^ -\right>_{\calH ' \times\calH } = 0.
\]
Then we apply the formula \eqref{IntMaxMin1} for the function $\tilde f$ to compute $\left<  \partial_t \tilde{f} + \calT' \tilde{f},\tilde{f} ^ -\right>_{\calH ' \times\calH }$ and Lemma \ref{LemNormInftyBis2} for the second term in the previous equation. Therefore we obtain 
\[
0\leq -\lambda \left<\tilde{f}^-, \tilde{f}^- \right>_{L^2 \times L^2 },
\]
which implies that $\tilde{f}^- =0$ a.e and $\tilde{f} \geq 0$ a.e so $f\geq 0$ a.e.

Now we estimate the bound of $L^\infty$ norm. First, making the change of unknown function $w(t,x,v) = e^{-2t}f(t,x,v)$ in the equation \eqref{equ:LiVFPAppen} we get
\begin{equation*}
\begin{cases}
\frac{\partial w }{\partial t} + \left[ v\cdot\nabla_x w + (B(x) ^\perp v -v)\cdot\nabla_v w \right]+ E(t,x) \cdot\nabla_v w  -\sigma \Delta_v w = 0,\\
w_0(x,v) = f_{\mathrm{in}}(x,v).
\end{cases}
\end{equation*}
We will prove that $\|w(t)\|_{L^\infty} \leq \|w_0\|_{L^\infty}$, $t\in [0,T]$. Putting $w_1(t,x,v) = K(w(t,x,v)- \|w_0\|_{L^\infty})$ where $K$ is a function of class $C^ 2$ satisfying:
\begin{align*}
K(s)=0,\,\, s\leq 0, \,\,K\,\,\text{is increasing},\\
\|K'\|_{L^\infty}\leq C,\,\,\, K''\geq 0.
\end{align*}
We give an example on the function $K$ as $K(y)=\int _0 ^y g(s) \mathrm{d}s  $ with $g(s) =  e^{-\frac{1}{s}}$ if $s> 0$ and $f(s)=0$ if $s\leq 0$. By the construction of $K$ and $w\in \Y$ we deduce that $w_1 \in\calH$ and $\partial_t w_1 + \calT  w_1 = K'(w(t) - \|w_0\|_{\infty})(\partial_t w + \calT w) \in \calH'$. Multiplying the equation for $w$ above by $K'(w(t,x,v)- \|w_0\|_{L^\infty})$ then  $w_1$ belongs to $\Y$ and satisfies the following equation
\begin{align*}
\left\{\begin{array}{l}
\partial_t w_1+ \calT w_1 + E(t,x) \cdot\nabla_v w_1 -\sigma \Delta_v w_1 + \sigma | \nabla_v w |^2 K''(w-\|w_0\|_{L^\infty})= 0,\\
w_1 (0) =K(w(0,x,v)-\|w_0\|_{L^\infty}) =0.
\end{array}\right.
\end{align*}
We then put $w_2 (t,x,v) = e^{-\beta t}w_1 (t,x,v)$, with any $\beta >0$. The function $w_2$ belongs to $\Y$ and satifies the equation
\begin{align*}
\left\{\begin{array}{l}
 \partial_t w_2 + \calT w_2 + E(t,x) \cdot\nabla_v w_2 +\beta w_2  -\sigma \Delta_v w_2 + e^{-\beta t}\sigma  | \nabla_v w |^2 K''(w-\|w_0\|_{L^\infty})= 0,\\
w_2(0) = 0.
\end{array}\right.
\end{align*}
Therefore, $w_2$ satisfies the variational equation
\begin{align*}
&\left<\partial_t w_2 + \calT w_2,w_2 ^ + \right>_{\calH'\times\calH} \\&+ \left< E(t,x) \cdot\nabla_v w_2 +\beta w_2  -\sigma \Delta_v w_2 + e^{-\beta t}\sigma  | \nabla_v w |^2 K''(w-\|w_0\|_{L^\infty}) ,w_2 ^ +\right>_{\calH'\times\calH} =0 .
\end{align*}
Using \eqref{IntMaxMin2} for $\left< \partial_t w_2 + \calT w_2,w_2 ^ +\right>_{\calH ' \times\calH }$, we can easily deduce that
\[
\intTxvtt{|w_2 ^+|^2 } + \beta \left<w_2^+, w_2^+ \right>_{L^2 \times L^2 } +  \sigma\left<\nabla_v w_2 ^+ , \nabla_v w_2 ^+\right>_{L^2 \times L^2 } \leq 0.
\]
This implies that $w_2 ^+ =0$. Thus $w_2 \leq 0$ and $w_1 \leq 0$ which yields $\| w(t)\|_{L^\infty} \leq \|w_0\|_{L^\infty}$. 
\end{proof}
\begin{remark}$\;$\\
\label{RemarkInfty}
If we add the source term $U(t, x, v) $ in the right hand side of \eqref{equ:LiVFPAppen}, that means
\[
\frac{\partial f }{\partial t} +  v\cdot\nabla_x f + (B(x) ^\perp v -v)\cdot\nabla_v f + E(t,x) \cdot\nabla_v f - 2 f -\sigma \Delta_v f = U,
\]
and we assume that $U\in L^1(0,T;L^\infty(\R^2\times\R^2))$. Then we have
\[
\| f(t)\|_{L^\infty(\R^2\times\R^2)} \leq e^{2T}\|f_\mathrm{in}\|_{L^\infty(\R^2\times\R^2)} + \int_{0}^{T}{\|U(s)\|_{L^\infty}}\mathrm{d}s,\,\, t\in [0,T].
\]
\end{remark}

The following estimates relate to the $L^p$ estimate, the kinetic energy and the entropy of equation \eqref{equ:LiVFPAppen}. To establish these estimates, we make the change of unknown function $w(t,x,v) = e^{-2t}f(t,x,e^{-t}v)$.  Then $w$ is the solution of the following equation
\begin{equation}
\label{equ:NewVFP2DBis}
\begin{cases}
\dfrac{\partial w }{\partial t} + e^{-t} v\cdot\nabla_x w + B(x) ^\perp v\cdot\nabla_v w + e^t E(t,x) \cdot\nabla_v w  -\sigma e^{2t} \Delta_v w = 0,\\
w_0(x,v) = f_{\mathrm{in}}(x,v).
\end{cases}
\end{equation}
The solution $w$ satisfies $w\in\calH$ and $\partial_t w + \calT' w \in \calH'$ since $f\in\Y$. The estimates of solutions that we will study can be obtained by choosing of an appropriate sequence of functions in the varational equation of $w$.
\begin{pro}$\;$\\
Assume that the initial data $f_\mathrm{in}$ is positive and belongs to $L^p(\R^2\times\R^2)$, with any $p\in [1,\infty[$. Then solution $f$ provided by Proposition \ref{PropExiUniq} satisfies
\begin{align}
\| f \|_{L^\infty(0,T;L^{p}(\R^2\times\mathbb{R}^2))} \leq \mathrm{e}^{\frac{p-1}{p}2T} \|f_\mathrm{in}\|_{L^p (\R^2\times\R^2)},\,\,1\leq p<\infty,\label{LpNormAppen}\\
\|\nabla_v  f^{p/2}\|_{L^2(0,T;L^{2}(\R^2\times\R^2))} \leq  \sqrt{\dfrac{p}{4(p-1)\sigma}} \mathrm{e}^{(p-1)T} \|f_\mathrm{in}\|_{L^p (\R^2\times\R^2)},\,\,1<p<\infty.\label{LpGradNormAppen}
\end{align}
\end{pro}
\begin{proof}$\;$\\
First we consider the case $p=2$. Since $w$ in $\calH$ satisfies \eqref{equ:NewVFP2DBis}, we deduce that
\[
\left< \partial_t w + \calT' w , w \right>_{\calH'\times\calH} = \left< -e^t E(t,x)\cdot\nabla_v w  + \sigma e^{2t}\Delta_v w, w \right>_{\calH'\times\calH}.
\]
Since $w\in\calH$ the divergence theorem implies that the integral of $-e^{t}E(t,x)\cdot\nabla_v w$ vanish on $\R^2\times\R^2$.
Then we apply \eqref{InteGreenBis} for $\left< \partial_t w + \calT' w , w \right>_{\calH'\times\calH}$ to obtain
\[
2\left< \partial_t w + \calT' w , w \right>_{\calH'\times\calH}= \intxvtt{|w(T,x,v)|^2|} - \intxvtt{|w(0,x,v)|^2}.
\]
Therefore we get for any $T>0$ that
\[
\intxvtt{|w(T,x,v)|^2} + 2\sigma \intTxvtt{e^{2t}|\nabla_v w|^2} = \intxvtt{|w(0,x,v)|^2}
\]
which yields the bounds of \eqref{LpNormAppen} and \eqref{LpGradNormAppen} when $p=2$.

Next, we consider the case $1 \leq p< \infty$ and $p\neq 2$. We establish a class of function of approximation $C^2$ of $px^{p-1}$, $x \geq 0$ (indeed, the function $pw^{p-1}$ does not belong to $\calH$ hence we can not define $\left< \partial_t w + \calT' w, pw^{p-1}\right>_{\calH',\calH}$ so we need to modify the function $px^{p-1}$) verifies 
\begin{enumerate}
\item[(i)]$p=1:$
 $\psi_\eps (s) = 0$ if $s\leq 0$,\,\,$\psi_\eps (s) = 1 $ if\,\,$\eps \leq s$ and $ \psi _\eps (s)$ is increasing in $[0,\eps]$.
\item[(ii)] $1< p< \infty$, $p\neq 2:$ 
 $\psi_\eps (s) = 0 $ if  $s\leq \eps$,\,\, $\psi_\eps (s) = p s^{p-1}$ if $\eps \leq s\leq \dfrac{1}{\eps}$ and
 $\psi '_\eps (s) =0$ on  $[1/\eps,+\infty)$. 
\end{enumerate}
It is easily seen that $\psi_\eps\in C^2$  with $\psi_\eps '\in L^\infty(\R)$ and $\psi_\eps(0) =0$. Let $\varphi_\eps (s)$ be a primitive of $\psi_\eps (s)$ defined by
$
\varphi_\eps (t) = \int_{-\infty} ^t \psi_\eps (s)\mathrm{d}s.
$
Since $w\in\calH$ we imply that  $\psi_\eps (w)$ and $\varphi_\eps (w)$ belong to $\calH$ and $\nabla_v \varphi_\eps(w) = \psi_\eps(w)\nabla_v w$. Moreover, the function $w$ in $\calH$ satisfies \eqref{equ:NewVFP2DBis}, we deduce that
\begin{align}
\label{equ:VarApprox}
\left<   \partial_t w + \calT'w,\psi_\eps(w)\right>_{\calH ' \times\calH } + \left<e^t E(t,x)\cdot\nabla_v w  -\sigma e^{2t}\Delta_v w ,\psi_\eps(w) \right>_{\calH ' \times\calH } = 0,
\end{align}
where $\calT'w = e^{-t}v\cdot\nabla_x w  + B(x) ^\perp v \cdot\nabla_v w$. In the same way of Lemma \ref{GreenFormulas} we also have
\[
\left<\partial_t w + \calT' w, \psi_\eps (w) \right>_{\calH'\times\calH} 
= \intxvtt{\varphi_\eps (w(T,x,v))} - \intxvtt{\varphi_\eps (w(0,x,v))}.
\]
Since $w \in\calH$ the divergence theorem implies the integral of $e^t E(t,x)\cdot\nabla_v w$ vanish on $\R^2\times\R^2$.
If $p=1$, we apply again the divergence theorem to $\left<-\sigma e^{2t}\Delta_v w, \psi_\eps(w) \right>_{\calH'\times\calH}$ we have
\[
\left< -\sigma e^{2t}\Delta_v w  ,\psi_\eps(w) \right>_{\calH ' \times\calH } = \sigma\intTxvtt{e^{2t}|\nabla_v w|^2\psi' _\eps (w)\mathds{1}_{\left\{0\leq w\leq \eps\right\}}}.
\]
Then the equation \eqref{equ:VarApprox} gives
\begin{align*}
\intxvtt{\varphi_\eps(w(T,x,v))} &+ \sigma\intTxvtt{e^{2t}|\nabla_v w|^2\psi' _\eps (w)\mathds{1}_{\left\{0\leq w\leq \eps\right\}}} \\
&=  \intxvtt{\varphi_\eps(w(0,x,v))}.
\end{align*}
Since $\psi'_\eps \geq 0 $ and by using Fatou's Lemma and the dominated convergence theorem we get for any $T>0$ that
\[
\| w(T)\|_{L^1(\R^2\times\R^2)} = \intxvtt{w(T,x,v)} \leq \|w_0\|_{L^1(\R^2\times\R^2)},
\]
which yields \eqref{LpNormAppen} with $p=1$.\\
If $1< p<\infty$ and $p\neq 2$, by the construction of $\psi_\eps$ we have
\begin{align*}
\left< -\sigma e^{2t}\Delta_v w  ,\psi_\eps(w) \right>_{\calH ' \times\calH }&= \sigma\intTxvtt{e^{2t}|\nabla_v w|^2\psi' _\eps (w)\mathds{1}_{\left\{\eps\leq w\leq 1/\eps\right\}}} \\&= \sigma p(p-1)\intTxvtt{e^{2t}|\nabla_v w|^2 w^{p-2}\mathds{1}_{\left\{\eps\leq w\leq 1/\eps\right\}}} 
\\&= \dfrac{4(p-1)}{p}\sigma\intTxvtt{e^{2t}|\nabla_v w^{p/2}|^2\mathds{1}_{\left\{\eps\leq w\leq 1/\eps\right\}}}.
\end{align*}
Then the equation \eqref{equ:VarApprox} becomes
\begin{align*}
\intxvtt{\varphi_\eps(w(T,x,v))}+ \dfrac{4(p-1)}{p}\sigma\intTxvtt{e^{2t}|\nabla_v w^{p/2}|^2\mathds{1}_{ \left\{\eps\leq w\leq 1/\eps\right\}}} \\= \intxvtt{\varphi_\eps(u(0,x,v))}.
\end{align*}
Using Fatou's Lemma and the dominated convergence theorem we get for any $T>0$ that
\[
\|w(T)\|^p_{L^p(\R^2\times\R^2)} + \dfrac{4(p-1)}{p}\sigma\|\nabla_v  w^{p/2}\|^2_{L^2(0,T;L^{2}(\R^2\times\R^2))} 
\leq \|w_0\|^p_{L^p(\R^2\times\R^2)},
\]
which yields the estimates of \eqref{LpNormAppen} and \eqref{LpGradNormAppen} when $1< p <\infty$ and $p\neq 2$.
\end{proof}

Next we provide the estimates of the kinetic energy and the entropy. First we consider the truncation function $\chi(s)\in C^\infty _{0}(\R)$ such that
\[
\chi(s) = 1 \,\,\text{if}\,\,|s|\leq 1,\,\,\,\chi(s) = 0\,\,\text{if}\,\,|s|\geq 2,\,\,\|\chi\|_ {W^{1,\infty}(\R)}\leq 1
\]
and we define $\chi_R (z) = \chi\left( \frac{|z|}{R} \right),  z\in\R^2, R>0$. We then consider a function of class $C^\infty(\R) \cap L^\infty(\R)$ satisfying $\psi_\eps (s) = 0$ if $s\leq 0$, $ \psi_\eps (s) = 1$ if $s\geq \eps$ and $\psi_\eps$ is increasing on $[0,\eps]$. Let $\varphi_\eps$ be a primitive of $\psi_\eps$ as $\varphi_\eps (t) = \int _{-\infty} ^t \psi_\eps (s)\mathrm{d}s$.
\begin{pro}$\;$\\
\label{EsKinEnerVFP2D}
Assume that the initial data $f_\mathrm{in}$ is positive and $(1+|v|^2 /2)f_\mathrm{in}\in L^1(\R^2\times\R^2)$. Then the solution given by Proposition \ref{PropExiUniq} satisfies
\[
\intxvtt{\frac{|v|^2}{2} f(t)} \leq C_1 + C_2\intxvtt{\frac{|v|^2}{2} f_\mathrm{in}},\,\, t\in  [0,T],
\]
for some constants $C_1$ and $C_2$, depending only on $\|E\|_{L^\infty},f_\mathrm{in},T,\sigma$.
\end{pro}
\begin{proof}$\;$\\
Since $w(t) = e^{-2t}f(t,x,e^{-t}v )$ in $\calH$ satisfies the equation \eqref{equ:NewVFP2DBis}, we deduce for any function $h\in\calH$ that
\begin{equation}
\label{equ:VariForm} 
\left< \partial_t w + \calT' w, h \right>_{\calH'\times\calH} + \left<e^{t}E(t,x)\cdot\nabla_v w - \sigma e^{2t}\Delta_v w,h \right>_{\calH'\times\calH} = 0,
\end{equation}
where $\calT'  = e^{-t}v\cdot\nabla_x w + B(x)^\perp v\cdot\nabla_v$.
Taking in \eqref{equ:VariForm} the function $h =  \chi_R (|v|)\frac{|v|^2}{2} \psi_\eps(w)$. It is easily seen that $h\in\calH$ since the function $\chi_R(|v|)\frac{|v|^2}{2}\in L^\infty(\R)$ and $\psi_\eps(w)\in \calH$ by $w\in \calH$. In the same way of Lemma \ref{GreenFormulas}, we have the following formula for the first term in \eqref{equ:VariForm}
\begin{align*}
\left<  \partial_t w + \calT' w,h\right>_{\calH ' \times\calH }
= \intxvtt{\left[ \varphi_\eps(w(T,x,v))
 - \varphi_\eps( w(0,x,v))\right]\chi_R (|v|)\frac{|v|^2}{2}} . 
\end{align*}
Before estimating the other terms in \eqref{equ:VariForm} we need to observe that $
\varphi_\eps (w) = w\Phi_\eps (w)$ with $\Phi_\eps (w) = \int _0 ^1 \psi_\eps (\theta w)\mathrm{d}\theta,
$
which implies that
\[
\varphi_\eps(w)=|\varphi_\eps(w)| \leq w \int _0 ^1 | \psi_\eps (\theta w)|\mathrm{d}\theta \leq w,\,\,\forall \eps > 0.
\]
Moreover, the solution $w$ belongs to $L^1((0,T)\times\R^2\times\R^2)$ beacause $w\in L^\infty(0,T;L^1(\R^2\times\R^2))$. On the other hand, since $w\in\calH$ the divergence theorem implies that the term $\left<e^t E(t,x)\cdot\nabla_v w, h \right>_{\calH\times\calH'}$ can be estimated as
\begin{align*}
\left<e^t E(t,x)\cdot\nabla_v w , h \right>_{\calH ' \times\calH } 
&\leq \| \chi\|_{W^{1,\infty}(\R^2)}\|E\|_{L^\infty}e^T \intTxvtt{ w(t,x,v) |v|}\\
&\leq \dfrac{1}{2}C(\|E\|_{L^\infty},T)\left(\intTxvtt{w} + \intTxvtt{w |v|^2} \right).
\end{align*}
It remains to estimate the contribution of $\left<-\sigma e^{2t}\Delta_v w, h \right>_{\calH\times\calH'}$ in \eqref{equ:VariForm}. Similarly,  applying the divergence theorem and by direct computations we get
\begin{align*}
&\left< -\sigma e^{2t}\Delta_v w, h \right>_{\calH ' \times\calH }
\geq \sigma\intTxvtt{e^{2t}\nabla_v w \cdot\left[ \left( \chi' \left(\frac{|v|}{R}\right) v\frac{|v|}{2R}\mathds{1}_{\left\{|v|\leq 2R\right\}}+\chi_R v \right)\psi_\eps(w) \right]}\\
&=\sigma\intTxvtt{e^{2t}\nabla_v \varphi_\eps(w) \cdot \left(\chi' \left(\frac{|v|}{R}\right) v \frac{|v|}{2R}\mathds{1}_{\left\{|v|\leq 2R\right\}}+\chi_R v \right) }\\
&=-\sigma \intTxvtt{e^{2t}\varphi_\eps(w) \left[ \left( \chi''\left(\frac{|v|}{R}\right)\frac{|v|^2}{2R^2}\mathds{1}_{\left\{|v|\leq 2R\right\}}+ 2 \chi'\left(\frac{|v|}{R}\right)\frac{|v|}{R}\mathds{1}_{\left\{|v|\leq 2R\right\}} + 2\chi_R \right) \right]}\\
&\rightarrow - 2\sigma\intTxvtt{e^{2t} w }, \,\,\text{when}\,\,\eps\searrow 0, R\to \infty,
\end{align*}
where we have used the dominated convergence theorem in the last integral. 
Since 
\[
\intTxvtt{w} \leq \intTxvtt{w_0} = T \|f_\mathrm{in}\|_{L^1(\R^2\times\R^2)},
\]
we finally deduce from the above estimations that
\begin{align*}
\intxvtt{\varphi_\eps(w(T,x,v))\chi_R(|v|)\dfrac{|v|^2}{2}} \leq \intxvtt{\varphi_\eps(w(0,x,v))\chi_R(|v|)\dfrac{|v|^2}{2}} \\
+ C(\|E\|_{L^\infty}, T,\sigma, f_\mathrm{in}) + C(\|E\|_{L^\infty},T)\intTxvtt{w\frac{|v|^2}{2}}.
\end{align*}
Using Fatou's Lemma and then the dominated convergence theorem when $\eps\searrow 0$, $R\to\infty$ we get for any $T>0$ that
\begin{align*}
\intxvtt{w(T,x,v)\dfrac{|v|^2}{2}} \leq \intxvtt{w(0,x,v)\dfrac{|v|^2}{2}} \\
+ C(\|E\|_{L^\infty}, T,\sigma, f_\mathrm{in}) + C(\|E\|_{L^\infty},T)\intTxvtt{w\frac{|v|^2}{2}}.
\end{align*}
Thanks to Gronwall's inequality we complete the proof.
\end{proof}

In the same way as for the proof of Proposition \ref{EsKinEnerVFP2D}, if we take the function $h$ in the equation \eqref{equ:VariForm} given by $h(t,x,v) = \chi_R(|x|)|x|\psi_\eps(w)$, we can obtain the following Proposition
\begin{pro}$\;$\\
\label{BoundPosit}
Assume that the initial data $f_\mathrm{in}$ belongs to $L^1(\R^2\times\R^2)$ and satisfies $(|x|+|v|^2/2)f_\mathrm{in}\in L^1(\R^2\times\R^2)$. Then the solution $f$ is given by Proposition \ref{PropExiUniq} satisfies
\[
\intxvtt{|x|f(t)} \leq C_1 + C_2 \intxvtt{|x|f_\mathrm{in}},\,\, t\in [0,T],
\]
for some constants $C_1$ and $C_2$, depending only on $f_\mathrm{in}, T$.
\end{pro}
\begin{pro}$\;$\\
\label{EntropyVFP2D}
Assume that the initial function $f_\mathrm{in}$ is positive and verifies $(1+|x|+|v|^2/2)f_\mathrm{in}\in L^1(\R^2\times\R^2)$. Then the solution $f$ of Proposition \ref{PropExiUniq} satisfies
\[
\intxvtt{\sigma f|\ln f|} \leq C + \intxvtt{\sigma f_\mathrm{in}|\ln f_\mathrm{in}|},\,\, t\in [0,T],
\]
\[
\intxvtt{\dfrac{|\sigma\nabla_v f|^2}{f}} \leq C +\intxvtt{\sigma f_\mathrm{in}|\ln f_\mathrm{in}|},\,\, t\in [0,T],
\]
for some constant $C$, depending only on $\|E\|_{L^\infty}, f_\mathrm{in}, T, \sigma$.
\end{pro}
\begin{proof}$\;$\\
As before, we will work on $w(t,x,v) = e^{-2t}f(t,x,e^{-t}v)$ which is satisfied by equation \eqref{equ:NewVFP2DBis} and variational equation \eqref{equ:VariForm}.
 For any $\eps >0$, we define the function $g_\eps(w)$ as: 
 \[ g_\eps(w):= \mathds{1}_{\left\{  \eps \leq w\leq 1/\eps \right\}}\ln \varphi_\eps(w) =  \mathds{1}_{\left\{  \eps \leq w\leq 1/\eps \right\}}\ln w  ,
 \]
and it is obvious that it belongs to $ L^\infty((0,T)\times\R^2\times\R^2)$. It is easily seen that 
\[
\partial_t w (1 +\mathds{1}_{\left\{  \eps \leq w\leq 1/\eps \right\}}\ln \varphi_\eps(w)) = \partial_t (w g_\eps(w)),
\]
and
\[
\calT' w (1 + \mathds{1}_{\left\{  \eps \leq w\leq 1/\eps \right\}}\ln \varphi_\eps(w) )= \calT'(w g_\eps(w)).
\]
Multiplying the equation \eqref{equ:NewVFP2DBis} by $\sigma(1 + \mathds{1}_{\left\{  \eps \leq w\leq 1/\eps \right\}}\ln \varphi_\eps(w))$ and then passing to the variational equation with $h = \psi_\eps(w)\in\calH$ we get
\begin{align}
\label{equ:VariFormBis}
\sigma\left< \partial_t (w g_\eps(w)) + \calT' (w g_\eps(w)), \psi_\eps \right>_{\calH'\times\calH}\nonumber\\ + \sigma\left<[ e^{t}E(t,x)\cdot\nabla_v w  - \sigma e^{2t}\Delta_v w] (1+g_\eps(w)),\psi_\eps(w) \right>_{\calH'\times\calH} = 0.
\end{align}
Since $\psi_\eps(w) = 1$ on $\eps\leq w \leq 1/\eps$ so in the same way of Lemma \ref{GreenFormulas}, we have the following formula for the first term in \eqref{equ:VariFormBis}
\begin{align*}
\sigma \left< \partial_t (w g_\eps(w)) + \calT' (w g_\eps(w)), \psi_\eps \right>_{\calH'\times\calH} &= \intxvtt{\sigma w(T,x,v)\ln w(T,x,v)\mathds{1}_{\left\{ \eps \leq w\leq {1}/{\eps }  \right\}}} \\&- \intxvtt{\sigma w(0,x,v)\ln w(0,x,v)\mathds{1}_{\left\{ \eps \leq w\leq {1}/{\eps }  \right\}}}.
\end{align*}
We estimate now the other terms in \eqref{equ:VariFormBis}. Since $w\in\calH$ so the divergence theorem implies that
\begin{align*}
\sigma \left< e^t E(x)\cdot\nabla_v w (1+ g_\eps(w)), \psi_\eps(w) \right>_{\calH'\times\calH}  =0,
\end{align*}
and
\begin{align*}
 \left<- \sigma^2 e^{2t}\Delta_v w (1+g_\eps(w)),\psi_\eps(w) \right>_{\calH'\times\calH} =  \intTxvtt{e^{2t}\dfrac{|\sigma\nabla_v w|^2}{w}\mathds{1}_{\left\{ \eps \leq w\leq {1}/{\eps }  \right\}}}.
 \end{align*}
 Finally, from \eqref{equ:VariFormBis} we obtain for any $T>0$ that
 \begin{align*}
 \intxvtt{\sigma w(T,x,v)\ln w(T,x,v)\mathds{1}_{\left\{ \eps \leq w\leq {1}/{\eps }  \right\}}} + \intTxvtt{e^{2t}\dfrac{|\sigma\nabla_v w|^2}{w}\mathds{1}_{\left\{ \eps \leq w\leq {1}/{\eps }  \right\}}}\\
 = \intxvtt{\sigma w(0,x,v)\ln w(0,x,v)\mathds{1}_{\left\{ \eps \leq w\leq {1}/{\eps }  \right\}}}.
 \end{align*}
By standard argument, there exists a constant $C>0$, (see \cite{PouPerSol}, Lemma 2.3) such that
\[
|u \ln u| = u\ln u - 2 u \ln u \mathds{1}_{\left\{ 0 \leq u \leq 1 \right\}} \leq u\ln u + \dfrac{1}{4} (|x| +|v|^2)u + C e^{-\frac{|x|+|v|^2}{2}},
\]
therefore
\begin{align*}
 \intxvtt{\sigma w(T,x,v)|\ln w(T,x,v)|\mathds{1}_{\left\{ \eps \leq w\leq {1}/{\eps }  \right\}}} + \intTxvtt{e^{2t}\dfrac{|\sigma\nabla_v w|^2}{w}\mathds{1}_{\left\{ \eps \leq w\leq {1}/{\eps }  \right\}}}\\
 \leq \intxvtt{\sigma w(0,x,v)|\ln w(0,x,v)|} + \dfrac{1}{4}\intxvtt{(|x| +|v|^2)w} + C 8\pi
\end{align*}
where we have used that $\intxvtt{e^{-\frac{|x|+|v|^2}{2}}}
 = 8\pi$. Thanks to the hypothesis on the initial data $f_\mathrm{in}$ we infer that $\frac{1}{4}\intxvtt{(|x| +|v|^2)w}\leq C(\|E\|_{L^\infty}, f_\mathrm{in}, T, \sigma)$. Therefore, Fatou's Lemma implies that
 \begin{align*}
& \intxvtt{\sigma w(T,x,v)|\ln w(T,x,v)|  } + \intTxvtt{e^{2t}\dfrac{|\sigma\nabla_v w|^2}{w}}\\&
 \leq \intxvtt{\sigma w(0,x,v)|\ln w(0,x,v)|}.
\end{align*}
Substitutively $w = e^{-2t}f(t,x,e^{-t}v)$ leads to the desired result.
\end{proof}
\paragraph{Acknowledgement}$\;$\\
This work has been carried out within the framework of the EUROfusion Consortium, funded by the European Union via the Euratom Research and Training Programme (Grant Agreement No 101052200 — EUROfusion). Views and opinions expressed are however those of the author(s) only and do not necessarily reflect those of the European Union or the European Commission. Neither the European Union nor the European Commission can be held responsible for them.

\end{document}